\documentclass{amsart}
\usepackage{amsmath, amsfonts, amssymb}
\usepackage{enumitem}

\usepackage[margin = 3.5cm]{geometry}

\usepackage{accents}
\usepackage[all]{xy}
\usepackage{tikz}
\usepackage{xcolor}

\usepackage[colorlinks=true, linkcolor=blue,
  citecolor=blue, urlcolor=blue]{hyperref}
\usepackage{comment}

\def\newthm#1#2{\newtheorem{#1}[dummy]{#2}%
  \expandafter\def\csname#2\endcsname##1{\hyperref[#1:##1]{#2~\ref*{#1:##1}}}}
\newthm{thm}{Theorem}
\newthm{lemma}{Lemma}
\newthm{prop}{Proposition}
\newthm{cor}{Corollary}
\newthm{conj}{Conjecture}
\newthm{cons}{Consequence}
\newthm{fact}{Fact}
\newthm{obs}{Observation}
\newthm{guess}{Guess}
\theoremstyle{definition}
\newthm{defn}{Definition}
\newthm{remark}{Remark}
\newthm{example}{Example}
\newthm{question}{Question}
\newthm{notation}{Notation}
\newtheorem*{ack}{Acknowledgements}

\newcommand{\Section}[1]{\hyperref[sec:#1]{Section~\ref*{sec:#1}}}
\newcommand{\Table}[1]{\hyperref[tab:#1]{Table~\ref*{tab:#1}}}
\newcommand{\eqn}[1]{\hyperref[eqn:#1]{(\ref*{eqn:#1})}}

\DeclareMathOperator{\GRCF}{GRCF}

\DeclareMathOperator{\GIF}{GIF}

\DeclareMathOperator{\Sp}{Sp}

\DeclareMathOperator{\Gr}{Gr}
\DeclareMathOperator{\Fl}{Fl}
\DeclareMathOperator{\LG}{LG}

\DeclareMathOperator{\IG}{IG}

\DeclareMathOperator{\Ker}{Ker}

\DeclareMathOperator{\QH}{QH}
\DeclareMathOperator{\QK}{QK}

\DeclareMathOperator{\codim}{codim}

\DeclareMathOperator{\Pic}{Pic}

\DeclareMathOperator{\pr}{pr}
\DeclareMathOperator{\Rk}{Rk}

\newcommand{\CC}{{\mathbb{C}}}
\newcommand{\cW}{{\mathcal{W}}}
\newcommand{\fp}{{\mathfrak{P}}}

\newcommand\oZ{\accentset{\circ}{Z}}

\newcommand{\bP}{{\mathbb P}}
\renewcommand{\P}{{\mathbb P}}
\newcommand{\PP}{{\mathbb P}}
\newcommand{\C}{{\mathbb C}}

\newcommand{\Q}{{\mathbb Q}}
\newcommand{\Z}{{\mathbb Z}}

\newcommand{\cF}{{\mathcal F}}
\newcommand{\cU}{{\mathcal U}}

\newcommand{\cO}{{\mathcal O}}
\newcommand{\cP}{{\mathcal P}}
\newcommand{\cS}{{\mathcal S}}

\newcommand{\du}{{\underline{d}}}

\newcommand{\xu}{{\underline{x}}}

\newcommand{\euler}[1]{\chi_{_{#1}}}

\newcommand{\id}{\text{id}}

\newcommand{\ev}{\operatorname{ev}}

\newcommand{\ov}{\overline}

\newcommand{\ignore}[1]{}

\newcommand\oX{\accentset{\circ}{X}}

\keywords{Quantum K-theory, Gromov--Witten invariants, symplectic Grassmannians, Seidel representation, Chevalley formula, rational connectedness.}


\usepackage{hyperref}

\begin{document}

\title{Quantum $K$-theory of $\IG(2,2n)$}
\author{V.~Benedetti, N. Perrin and W. Xu}
\begin{abstract}
    We prove that the Schubert structure constants of the quantum \(K\)-theory rings of symplectic Grassmannians of lines have signs that alternate with codimension and vanish for degrees at least 3. We also give closed formulas that characterize the multiplicative structure of these rings, including the Seidel representation and a Chevalley formula.
\end{abstract}

\maketitle

\setcounter{tocdepth}{1}
\tableofcontents

\section{Introduction}

To a complex projective manifold $X$ one can attach various algebraic structures capturing different kinds of information. For instance, cohomology rings encode the behavior of (classes of) subvarieties: how they intersect (ordinary or equivariant cohomology) or how they are connected by rational curves (quantum cohomology). In this paper we are interested in $K$-theory rings, which in turn encode information about coherent (or locally free) sheaves: how they form short exact sequences and the behavior of their tensor product. It turns out that classical $K$-theory can be seen as a refinement of classical cohomology - via the Chern character homomorphism. Moreover, one can define a quantum $K$-theory ring which is both a deformation of the classical $K$-theory ring and a refinement of the quantum cohomology ring. Indeed, if quantum cohomology is an intersection theory of the space of genus zero stable maps to $X$, quantum $K$-theory, which was introduced by \cite{givental:wdvv}, deals with the $K$-theory of this space and requires a more careful study of stable maps with reducible sources. Quantum cohomology and K-theory have gathered some attention in the context of homogeneous spaces where it generalizes classical Schubert calculus. We will focus on this case.

From now on, $X = G/P_X$ will be a rational projective homogeneous space. It is nowadays classical that the transitive action of $G$ on $X$ implies many positivity properties in cohomology and its various generalizations (see \cite{graham:positivity,brion:positivity,buch:littlewood,mihalcea:positivity,graham.kumar:positivity,anderson.griffeth.ea:positivity,buch.chaput.ea:positivity} and references therein). From previous positivity results and first explicit computations, it was soon conjectured that positivity should also hold in quantum $K$-theory. To state a precise conjecture, we introduce a few notations. Schubert varieties in $X$ are indexed by $W^X$ and for  $u \in W^X$, let $X^u$ be the Schubert variety with $\codim_X(X^u)= \ell(u)$ (see Subsection \ref{subsection:schubert}). Denote by $\cO^u = [\cO_{X^u}]$ its class in $\QK(X)$, the (small) quantum $K$-theory ring of $X$. These classes form a basis for $\QK(X)$ and their products are defined by structure constants as follows:
$$\cO^u \star \cO^v = \sum_d\sum_{w \in W^X} N_{u,v}^{w,d}q^d \cO^w,$$
where the first sum runs over all effective classes $d \in H_2(X,\Z)$ (see Subsection \ref{subsection:quantum} {for more details}). It can be proved (see \cite{buch.chaput.ea:finiteness,kato:loop,anderson.chen.ea:on}) that this sum is actually finite. The structure constants $N_{u,v}^{w,d}$ are expected to satisfy the following positivity property (see \cite{lenart.maeno:quantum,
buch.mihalcea:quantum, buch.chaput.ea:chevalley}):

\begin{conj}
\label{conj:pos-intro}
For $u,v,w \in W^X$ and $d \in H_2(X,\Z)$, we have
$$(-1)^{\ell(uvw) + \int_d c_1(T_X)} N^{w,d}_{u,v} \geq 0.$$
\end{conj}

This conjecture has been proven in a few cases: for minuscule spaces in \cite{buch.chaput.ea:positivity} and for the point/hyperplane incidence varieties in \cite{rosset:quantum} and \cite{xu:quantum}. In this paper, we consider the case where $X = \IG(2,2n)$ is the variety of isotropic lines in a symplectic vector space. Our first result is:

\begin{thm}
\label{main} 
Conjecture \ref{conj:pos-intro} is true for $X = \IG(2,2n)$.
\end{thm}

We also obtain several more precise results on the structure of $\QK(X)$. 

\begin{prop}[see Corollary \ref{cor:dleq2}  and Remark \ref{rmk:interval}]
  For $X = \IG(2,2n)$ and $u,v \in W^X$, in the product $\cO^u \star \cO^v$:
  \begin{enumerate}
    \item Terms with \(q^d\) occur only for \(d\in[0,2]\);
    \item The powers of \(q\) that occur form an interval.
  \end{enumerate}
\end{prop}

Although these properties also hold for the product $[X^u] \star [X^v]$ in quantum cohomology, the maximum power of $q$ appearing in  $\cO^u \star \cO^v$ is sometimes greater than that appearing in $[X^u] \star [X^v]$. See Remark \ref{rmk:power} for more details.

In addition, we prove the following closed formula for multiplying with the unique Schubert divisor class. Recall that Schubert varieties in $X = \IG(2,2n)$ are indexed by pairs $(p_1,p_2)$ of integers with $1 \leq p_1 < p_2 \leq 2n$ and $p_1 + p_2 \neq 2n+1$. Explicitly, let $(E_k)_{k \in [1,2n]}$ be an isotropic flag and set
$$X_{p_1,p_2} := \{ x \in X \ | \ V_x \cap E_{p_1} \neq 0 \textrm{ and } V_x \subset E_{p_2} \},$$
where $V_x$ is the $2$-dimensional subspace in $\C^{2n}$ representing $x \in X$. For \(a<b\leq 2n\) such that \(a+b\not\equiv 1\) mod \(2n\), we recursively set 
\[
  \cO_{a,b}=
  \begin{cases}
    [\cO_{X_{a,b}}] & \text{if }1\leq a<b\leq 2n\\
    q\cO_{{b,a+2n}} & \text{if }a\leq 0,\\
    0 & \text{otherwise}.
  \end{cases}
\] 

\begin{thm}[see Theorem \ref{thm_quantum_chevalley}]
  \label{thm_quantum_chevalley-intro}
  The product \(\cO_{2n-2,2n}\star\cO_{q_1,q_2}\) equals:
  \begin{enumerate}
    \item \(\cO_{q_1-1,q_2}\) if \(q_1=q_2-1\);
    \item \(\cO_{q_1-1,q_2}+\cO_{q_1,q_2-1}-\cO_{q_1-1,q_2-1}\) if \(q_1<q_2-1\) and \(q_1+q_2\neq 2n+2, 2n+3\);
    \item \(\cO_{q_1-1,q_2}+\cO_{q_1,q_2-1}-\cO_{q_1-1,q_2-2}-\cO_{q_1-2,q_2-1}+\cO_{q_1-2,q_2-2}\) if \(q_1<q_2-1\) and \(q_1+q_2=2n+3\);
    \item \(2\cO_{q_1-1,q_2-1}+\cO_{q_1-2,q_2}+\cO_{q_1,q_2-2}-2\cO_{q_1-2,q_2-1}-2\cO_{q_1-1,q_2-2}+\cO_{q_1-2,q_2-2}\) when \(q_1+q_2=2n+2\) and \(q_1\neq 2,\ n\);
    \item[(4.1)]  \(2\cO_{q_1-1,q_2-1}+\cO_{q_1,q_2-2}-2\cO_{q_1-1,q_2-2}-\cO_{q_1-2,q_2-1}+\cO_{q_1-2,q_2-2}\) when \(q_1=2,\ q_2=2n\);
    \item \(2\cO_{q_1-1,q_2-1}+\cO_{q_1-2,q_2}-2\cO_{q_1-2,q_2-1}-\cO_{q_1-1,q_2-2}+\cO_{q_1-2,q_2-2}\) when \(q_1=n,\ q_2=n+2\). 
  \end{enumerate}
  \end{thm}

  For quantum cohomology, Seidel \cite{seidel:pi_1} proved that there is a representation of $\pi_1({\rm Aut}(X))$ in the group $\QH(X)_{\rm loc}^\times$ of invertibles in $\QH(X)_{\rm loc}$, the localization of quantum cohomology $\QH(X)$ where he quantum parameters are inverted. This representation was made explicit first for cominuscule spaces and then in general in \cite{chaput.manivel.ea:hidden,chaput.manivel.ea:affine,chaput.perrin:affine}. Later, in \cite{buch.chaput.ea:positivity} these results were extended to the quantum $K$-theory of cominuscule spaces. We prove that the Seidel representation in quantum $K$-theory also exists for $X = \IG(2,2n)$. In this case, the group $\pi_1({\rm Aut}(X))$ has order $2$ and Seidel representation maps the unit to $\cO_{2n-1,2n} = 1 \in \QK(X)$ and the unique non-trivial element to $\cO_{n-1,n} = [\cO_{X_{n-1,n}}] \in \QK(X)$.

\begin{thm}[see Theorem \ref{thm:seidel-Ktheo}] 
\label{prop:conjecture-intro}
We have \(\cO_{n-1,n} \star \cO_{p_1,p_2} = \cO_{p_1-n,p_2-n}\).
\end{thm}

A geometric version of the above formula was conjectured in \cite{buch.chaput.ea:seidel} for all projective homogeneous spaces, and we prove it for $X = \IG(2,2n)$ in Theorem \ref{thm:seidel-geom-intro}. We call a homogeneous space \(Y=G/P_Y\) cominuscule if the unipotent radical of the parabolic subgroup $P_Y$ is abelian. Let \(W\) be the Weyl group of \(G\) and 
let $w^Y$ be the minimal length representative of the longest Weyl group element modulo \(W_Y\), the Weyl group of \(P_Y\). A classical result states that $\pi_1({\rm Aut}(X)) \simeq W^{\rm comin} = \{1\} \cup \{ w^Y \in W \ | \ Y \textrm{ cominuscule} \}$. For $G = {\rm Sp}_{2n}$, this group is $\{1,w^{\LG(n,2n)} \}$, where $\LG(n,2n)$ is the Lagrangian Grassmannian of maximal isotropic subspaces. For $u,w \in W$, {let $d_{\rm min}(u,w)$ be the smallest power of $q$ {appearing in the quantum cohomology product $[X^u] \star [X^w]$}}, and for $A,B \subseteq X$, let $\Gamma_d(A,B)$ be the locus swept out by degree $d$ curves in $X$ meeting $A$ and $B$. 

 \begin{thm}[see Theorem \ref{thm:seidel-geom}]
 \label{thm:seidel-geom-intro}
 Let $u \in W$, $w \in W^{\rm comin}$, then $$\Gamma_{d_{\rm min}(u,w)}(X_{w_0w},X^u) = w^{-1}.X^{wu}.$$
 \end{thm}

Let us sketch a few steps in our strategy for proving the positivity result (Theorem \ref{main}). Following \cite{buch.chaput.ea:finiteness} and \cite{mihalcea:lectures}, we prove in Section \ref{section:strategy} that to compute the quantum product $\cO^u \star \cO^v$, we only need to deal with irreducible source curves or source curves with two components: a main component of degree $d-1$ and a tail of degree $1$. To compute structure constants in $K$-theory, which are defined by pushforwards, we apply a result of Koll\'ar \cite{kollar:higher} (see Theorem \ref{thm:push-forward}), which requires rational connectedness of general fibers of restrictions of the evaluation map. A key ingredient in proving rational connectedness of these fibers is Proposition \ref{prop:replace}, which says that it suffices to prove rational connectedness of the general fibers of restrictions of the projection \(X^3\to X\). Then, using a result of Brion \cite{brion:positivity} we reduce Theorem \ref{main} to the statement that the image in $X$ of stable maps as above - i.e. the curve neighborhoods - have rational singularities and nice dimensional properties. We explain these reductions in Section \ref{section:replacing}, and then prove rational connectedness results in Sections \ref{sec:d=3} (\(d\geq3\)), \ref{sec:d=2} (\(d=2\)), and \ref{sec:d=1} (\(d=1\)). 

While proving Proposition \ref{prop:replace}, we also obtain the following result of general interest:

\begin{prop}[see Proposition \ref{prop:diagram}]
\label{prop:diagram-intro}
Let $A$ be a projective, irreducible and normal variety and $g : A \to B$, $h : B \to C$ be surjective morphisms. If the general fibers of the maps $g$ and $h$ are rationally connected, so are the general fibers of the composition $h \circ g$.
\end{prop}

However, for maps with non-irreducible sources, there are a few explicit cases for which rational
connectedness does not hold. We deal with these cases in Section \ref{sec:2-to-1}. We use the definition of structure constants to explicitly compute the product and conclude
the proof of Theorem \ref{main}. Seidel representation is discussed in Section \ref{section:seidel}, and we prove
the Chevalley formula in Section \ref{section:chevalley}.

Varieties will always be assumed to be complex, of finite type and reduced, but not necessarily irreducible.

\begin{ack}
The authors are partially supported by CATORE ANR-18-CE40-0024. The first two authors are also partially supported by FanoHK ANR-20-CE40-0023. 
\end{ack}

\section{Preliminaries and proof strategy}
\label{section:strategy}

In this section we introduce some notations and prove our main results assuming some rational connectedness results. The proof of these results will be provided in Sections \ref{sec:d=3},  \ref{sec:d=2},  \ref{sec:d=1} and  \ref{sec:2-to-1}.

\subsection{Schubert varieties}
\label{subsection:schubert}

Let $Z$ be a projective rational variety homogeneous under the action of a reductive group $G$. Let $B \subset G$ be a Borel subgroup and let $T \subset B$ be a maximal torus. Denote by $B^-$ the Borel subgroup opposite to $B$ defined by $B \cap B^- = T$. The variety $Z$ contains a unique $B$-invariant point, and we denote the parabolic subgroup stabilizing this point by $P_Z \subset G$ and the point itself by $1.P_Z$.  We identify $Z$ with the quotient $G/P_Z$. 

Let $W = N_G(T)/T$ be the Weyl group of $(G,T)$, let $W_Z$ be the Weyl group of $P_Z$ and let $W^Z \subset W$ be the subset of minimal representatives of the cosets in $W/W_Z$.  Each element $u \in W$ defines a
$T$-fixed point $u.P_Z \in Z$, the \emph{Schubert  cells} { and \emph{opposite Schubert cells} are $\oZ_u = Bu.P_Z$ and $\oZ^u = B^-u.P_Z$, respectively, and the \emph{Schubert varieties} and \emph{opposite Schubert varieties} are $Z_u =\ov{\oZ_u}$ and $Z^u = \ov{\oZ^u}$, respectively.} These varieties and cells depend only on the coset $u W_Z$ in $W/W_Z$, and for $u \in W^Z$ we have $\dim(Z_u) = \codim(Z^u,Z) = \ell(u)$.

Let $\leq$ denote the Bruhat order on $W$.  For $u, v \in W^Z$ we then have $v
\leq u$ $\Leftrightarrow$ $Z_v \subset Z_u$ $\Leftrightarrow$ $Z_u \cap Z^v \neq
\emptyset$.  In this case the intersection $Z_u^v = Z_u \cap Z^v$ is called a
\emph{Richardson variety}. This variety is reduced, irreducible, rational, and
$\dim(Z_u^v) = \ell(u) - \ell(v)$ \cite{richardson:intersections}. The Poincar\'e dual element of $u \in W^Z$ is $u^\vee = w_0 u w_{0,Z} \in W^Z$, which satisfies $Z^{u^\vee} = w_0.Z_u$ where $w_0 \in W$ and $w_{0,Z} \in W_Z$  are the longest elements.

For a variety $Y$, denote by $K(Y)$ its $K$-homology group and by $\chi_Y(\cF)$ the sheaf Euler-characteristic of $\cF \in K(Y)$. We set $\cO_u = [\cO_{Z_u}] \in K(Z)$ and $\cO^v = [\cO_{Z^v}] \in K(Z)$. Then $(\cO_u)_{u \in W^Z}$ and $(\cO^v)_{v \in W^Z}$ are two basis of $K(Z)$. The following result was proved in \cite{brion:positivity}, and we will use it in Sections \ref{sec:2-to-1} and \ref{section:chevalley}.

\begin{thm}\label{thm:brion}
  For $Y \subset Z$ a Cohen-Macaulay closed subvariety and for $g \in G$ general, we have
  $$[\cO_Y] = \sum_{w \in W^Z} \chi_{Y \cap g Z^w}(\cO_{Y \cap g Z^w}) [\cO_{Z_w}(-\partial Z_w)].$$
\end{thm}
When \(Y=Z_u\), we obtain: 
\begin{cor}\label{cor:schub_expand}
  For \(u\in W^Z\), we have \[\cO_u=\sum_{w\in W^Z:\ w\leq u}[\cO_{Z_w}(-\partial Z_w)].\]
\end{cor}

\subsection{Quantum $K$-theory and Curve neighborhoods}
\label{subsection:quantum}

For any effective degree $d \in H_2(Z,\mathbb{Z})$, let $M_{d,n} := \overline{{M}}_{0,n}(Z,d)$ be the moduli space of genus zero degree $d$ stable maps with $n$ marked points and $M_d = M_{d,3}$ (we always assume $n \geq 3$ if $d=0$). Let $\ev_i : M_d \to Z$ for $i \in [1,3]$ be the evaluation map at the corresponding marked point. Let $\du = (d_0,\cdots,d_r)$ be a sequence of {effective curve classes} and set $|\du| = \sum_{i=0}^r d_i$. Let $M_\du$ be the closure in $M_d$ of maps from a chain of irreducible curves of degrees $(d_0,\cdots,d_r)$ such that the first two marked points lie on the first component and the third marked point lies on the last component. The varieties \({M}_{\du}\) are irreducible with rational singularities (see \cite{thomsen:irreducibility,kim.pandharipande:the,fulton.pandharipande:notes}).

Given $K$-theory classes $(\cF_i)_{i \in [1,n]}$ in $K(Z)$, define the $n$-pointed
 $K$-theoretic Gromov--Witten invariant
$$\langle \cF_1, \cdots , \cF_n \rangle_{d} =
  \euler{M_{d,n}}\left(\prod_{i = 1}^n\ev_i^*(\cF_i)\right).$$
The quantum $K$-theory of $Z$ is defined to be $\QK(Z) = K(Z) \otimes \Q[q]$ as a module. Let $Z_u$ and $Z^v$ be Schubert and opposite Schubert varieties. We write $(\cO_u^\vee)_{u \in W^Z}$ for the dual basis of $(\cO_u)_{u \in W^Z}$ for the Euler pairing in $K(Z)$. Following Givental \cite{givental:wdvv}, we may define an associative $\Q[q]$-bilinear product $\star$ on $\QK(Z)$ via
$$\cO_u \star \cO^v = \sum_{d \geq 0}\sum_{w \in W^Z}\kappa_{u,v}^{w,d} q^d \cO_w,$$
with 
$$\kappa_{u,v}^{w,d} =\sum(-1)^r\langle \cO_u,\cO^v, \cO_{w_1}^\vee \rangle_{d_0} \prod_{i = 1}^r \langle \cO_{w_{i}},\cO_{w_{i+1}}^\vee \rangle_{d_i},$$
where the sum is over all decompositions $d = d_0 + \cdots + d_r$ with $d_i > 0$ for $i > 0$ and all elements $w_i \in W^Z$ for $i \in [1,r]$ where we set $w_{r+1} = w$. {Note that this product is well-defined since the above sum is finite (see \cite{buch.chaput.ea:finiteness,kato:loop,anderson.chen.ea:on}).}

We reformulate Conjecture~\ref{conj:pos-intro} as follows:
\begin{conj}
  \label{conj:pos}
  For $u,v,w \in W^Z$, we have
  $$(-1)^{\ell(uvw) + \int_d c_1(T_Z)} \kappa^{w,d}_{u,v} \geq 0.$$
\end{conj}
\begin{remark}
  We have the equality
$$N_{u,v}^{w,d} = \kappa_{u^\vee,v}^{w^\vee,d},$$
with $u^\vee = w_0uw_{0,Z}$ and $\ell(u^\vee) = \dim Z - \ell(u)$. In particular, the constants $N_{u,v}^{w,d} $ and $\kappa_{u,v}^{w,d}$ are conjectured to share the same positivity properties.
\end{remark} 

The  Grassmannian $X = \IG(2,2n)$ of isotropic lines for a symplectic form $\omega$ on $\C^{2n}$ is homogeneous for the symplectic group $G = \Sp_{2n}$. One of our main results is a proof of Conjecture \ref{conj:pos} for $X = \IG(2,2n)$.
From now on we assume that $Z = X = \IG(2,2n)$. For $x \in X$, we denote by $V_x \subset \C^{2n}$ the $2$-dimensional subspace corresponding to $x$. {Note that $\Pic(X)\cong \mathbb{Z}$ so that we can assume that the degree $d$ of any stable map is a non-negative integer.}

\subsection{A first reduction}

Given $n$ subvarieties $A_i \subset X$ for $i \in  [1,n]$, define the $n$-point degree $d$ \emph{Gromov--Witten variety} as $M_{d,n+1}(A_1,\cdots,A_n) = \ev_1^{-1}(A_1) \cap \cdots \cap \ev_n^{-1}(A_n)$ and the $n$-point degree $d$ \emph{curve neighborhood} as $\Gamma_{d}(A_1,\cdots,A_n) = \ev_{n+1}(M_{d,n+1}(\ev_1^{-1}(A_1) \cap \cdots \cap \ev_n^{-1}(A_n))$. For $n = 2$, we extend this definition to non-irreducible curves and drop the reference to $n$: set $M_\du(A,B) = M_{d,3}(A,B) \cap M_{\du}$ and $\Gamma_{\du}(A,B) = \ev_3(M_\du(A,B))$.

We will mainly consider these definitions for Schubert varieties and pairs of opposite Schubert varieties. In particular, it follows from \cite[Proposition 3.2]{buch.chaput.ea:finiteness} that $\Gamma_d(X_u)$ and $\Gamma_d(X^v)$ are Schubert varieties. Define $u(d),v(-d) \in W^X$ via $X_{u(d)} = \Gamma_d(X_u)$ and  $X^{v(-d)} = \Gamma_d(X^v)$. For opposite Schubert varieties $X_u,X^v \subset X$, we have the Gromov--Witten varieties $M_d(X_u,X^v) = \ev_1^{-1}(X_u) \cap \ev_2^{-1}(X^v)$ and $M_{\du}(X_u,X^v)  = M_d(X_u,X^v)  \cap M_{\du}$ with surjective evaluation maps
$$\ev_3^\du(u,v): M_\du(X_u,X^v) \to \Gamma_\du(X_u,X^v).$$ {\cite[Corollaries 3.1 and 3.3]{buch.chaput.ea:finiteness} and the fact that $M_\du$ is irreducible with rational singularities imply that the Gromov--Witten varieties $M_\du(X_u,X^v)$ are irreducible with rational singularities. }
It follows from \cite[Lemma 5.1]{buch.chaput.ea:finiteness} and the projection formula that the product in $\QK(X)$ can be reformulated using the above maps via the formula
$$\cO_u \star \cO^v = \sum_{d \geq 0} \sum_{\du, |\du| = d}(-1)^rq^d(\ev_3)_*[\cO_{M_\du(Z_u,Z^v)}].$$

We prove that, in the above formula, for the quantum product we only need to consider curves of degree $\du = d$ or $\du = (d-1,1)$. We start with the following geometric result on curves.

\begin{lemma}
  \label{lem:split-deg2}
  Let $x,y \in X$ be general points. Then there exists a conic as well as a chain of two lines through $x$ and $y$.
  \end{lemma}
  
  \begin{proof}
  Let $V = V_x \oplus V_y$ be the four dimensional subspace of $\C^{2n}$ spanned by $x$ and $y$. Because \(x\) and \(y\) are general, the symplectic form has rank \(4\) on \(V\). Then $X \cap \Gr(2,V)$ is a smooth quadric of dimension $3$ containing $x$ and $y$, and the result follows.
  \end{proof}

  \begin{cor}\label{cor_join_through_lines}
    Any point in $\Gamma_d(X_u)$ can be joint to $X_u$ by a chain of $d$ lines.
  \end{cor}

\begin{remark}
This result is true for any cominuscule variety \cite[Lemma 4.7]{buch.chaput.ea:finiteness} and for any coadjoint non-adjoint variety \cite[Proposition 3.15]{chaput.perrin:on}. However, it is not true for homogeneous spaces in general (see \cite[Proposition 3.10]{chaput.perrin:on} for the case of adjoint varieties for which this does not hold).
\end{remark}

We include a proof of  \cite[Corollary 2.5]{mihalcea:lectures} for completeness.

\begin{prop}
\label{prop_formula_simplification}
{Assume that for all $d \geq 0$, any point in $\Gamma_d(X_u)$ can be joint to $X_u$ by a chain of $d$ lines. Then} 
$$\kappa_{u,v}^{w,d} = \langle \cO_u,\cO^v, \cO_{w}^\vee \rangle_{d} - \sum_{\kappa} \langle \cO_u,\cO^v, \cO_{\kappa}^\vee \rangle_{d-1}  \langle \cO_\kappa,\cO_{w}^\vee \rangle_{1}.$$
\end{prop}

\begin{proof}
Define
$$S_{d-d_0} = \sum(-1)^r \prod_{i = 1}^r \langle \cO_{w_{i}},\cO_{w_{i+1}}^\vee \rangle_{d_i}$$
where the sum is over all decompositions $d - d_0 = d_1+ \cdots + d_r$ with $d_i > 0$ for $i > 0$ and all elements $w_i \in W^X$ for $i \in [1,r]$. We have 
$$\kappa_{u,v}^{w,d} = \sum_{w_1 \in W^X}\sum_{d_0 = 0}^d \langle \cO_u,\cO^v, \cO_{w_1}^\vee \rangle_{d_0} S_{d-d_0},$$
thus we only need to prove that $S_{d - d_0} = 0$ for $d - d_0 \geq 2$. 
As an easy consequence of \cite[Proposition 3.2]{buch.chaput.ea:finiteness}, we have $\langle \cO_u,\cO_w^\vee \rangle_d = \delta_{u(d),w}$. Hence, 
$$S_{d - d_0} =  \sum(-1)^r \prod_{i = 1}^r \delta_{w_i(d_i),w_{i+1}}.$$
{For any $i$ and $w$, $\Gamma_i(X_w)=\Gamma_1(\Gamma_1(\cdots \Gamma_1(X_w)\cdots ))$, where $\Gamma_1$ is repeated $i$ times in the right-hand side. {By induction on $i_1$ and $i_2$,} this implies that $w(i)=(w(i_1))(i_2)$ for any $i_1+i_2=i$. Thus, for any decomposition as above, we have $(\cdots(w_1(d_1))(d_2)\cdots )(d_r)=w_1(d_1+\cdots +d_r)$. This ultimately implies that $\prod_{i = 1}^r \delta_{w_i(d_i),w_{i+1}} = \delta_{w_1(d_1+\cdots+d_r),w} = \delta_{w_1(d - d_0),w}$ and is independent of the decomposition.} We thus have that $S_{d - d_0} = 0$ for $d - d_0 \geq 2$. 
\end{proof}

{To simplify notation, we set $(\cO_u \star \cO^v)_d = \sum_{\du, |\du| = d}(-1)^r(\ev_3)_*[\cO_{M_\du(X_u,X^v)}]$ so that 
$$\cO_u \star \cO^v = \sum_{d \geq 0} q^d(\cO_u \star \cO^v)_d.$$
Proposition \ref{prop_formula_simplification}, the projection formula, and \cite[Lemma 5.1]{buch.chaput.ea:finiteness} imply the following.}
\begin{cor}
\label{cor:final-formula}
$(\cO_u \star \cO^v)_d = (\ev_3^d)_*[\cO_{M_d(X_u,X^v)}] - (\ev_3^{d-1,1})_*[\cO_{M_{d-1,1}(X_u,X^v)}].$
\end{cor}

\subsection{Proof outline}

In this subsection, we prove Theorem \ref{main} assuming some rationality properties that allow us to use the following theorem, which is a consequence of results of Koll\'ar \cite{kollar:higher}, see for example \cite[Proposition 5.2]{buch.chaput.ea:finiteness}. We will prove these rationality properties in the next sections.

\begin{thm}
\label{thm:push-forward}
If $f : A \to B$ is a projective morphism of varieties with cohomologically trivial general fiber (e.g. rationally connected) and such that $A$ and $B$ have rational singularities, then $f_*[\cO_A] = [\cO_B]$ in $K$-theory. 
\end{thm}

We will prove the following theorem in Sections \ref{sec:d=3}, \ref{sec:d=2} and \ref{sec:d=1} (see Theorems \ref{thm:d=3}, \ref{thm:d=2} and \ref{thm:d=1}). A {dominant} projective morphism $f : A \to B$ with $A$ irreducible is called $\GRCF$ (Generically Rationally Connected Fibers) if a general fiber of $f$ is rationally connected (see Definition \ref{def:grcf}).

\begin{thm}
\label{thm:all-d}
For $d \geq 0$, we have
\begin{enumerate}
\item The map $\ev_3^d(u,v): M_d(X_u,X^v) \to \Gamma_d(X_u,X^v)$ is $\GRCF$,
\item Assume that  $\ev_3^d(u,v): M_d(X_u,X^v) \to \Gamma_d(X_u,X^v)$ is birational, then 
\begin{enumerate}
\item $\ev_3^{d-1,1}(u,v): M_{d-1,1}(X_u,X^v) \to \Gamma_{d-1,1}(X_u,X^v)$ is {birational};
\item $\Gamma_{d-1,1}(X_u,X^v) \subset \Gamma_d(X_u,X^v)$ is a divisor. 
\end{enumerate}
\item Assume that  $\ev_3^d(u,v): M_d(X_u,X^v) \to \Gamma_d(X_u,X^v)$ is not birational, then $\Gamma_d(X_u,X^v) = \Gamma_{d-1,1}(X_u,X^v)$ has rational singularities.
\item If $d \geq 3$, then the map $\ev_3^{d-1,1}(u,v): M_{d-1,1}(X_u,X^v) \to \Gamma_{d-1,1}(X_u,X^v)$ is $\GRCF$, and $\Gamma_{d}(X_u,X^v) = \Gamma_{d-1,1}(X_u,X^v)$  has rational singularities.
\end{enumerate}
\end{thm}

\begin{cor}
\label{cor:dleq2}
We have $(\cO_u \star \cO^v)_d = 0$ for $d \geq 3$.
\end{cor}

\begin{proof}
Corollary \ref{cor:final-formula} and Theorems \ref{thm:push-forward}, \ref{thm:all-d} imply the equalities $$(\cO_u \star \cO^v)_d = (\ev_3)_*[\cO_{M_d(X_u,X^v)}] - (\ev_3)_*[\cO_{M_{d-1,1}(X_u,X^v)}] = [\cO_{\Gamma_d(X_u,X^v)}] - [\cO_{\Gamma_{d-1,1}(X_u,X^v)}] = 0$$ {for $d \geq 3$}.
\end{proof}

For $d \in [1,2]$, the map $\ev_3^{d-1,1}(u,v): M_{d-1,1}(X_u,X^v) \to \Gamma_{d-1,1}(X_u,X^v)$ may fail to be $\GRCF$. To deal with these cases, we need to give a more explicit description of the Schubert varieties $X_u$ and $X^v$. Choose a basis $(e_i)_{i \in [1,2n]}$ of $\C^{2n}$ such that $\omega(e_i,e_j) = \delta_{i,2n+1-j}$ for $i \in [1,n]$. Define $E_p = \langle e_i \ | \ i \in [1,p] \rangle$ and $E^p = \langle e_i \ | \ 2n+1 - i \in [1,p] \rangle$. It is a classical result (see for example \cite{buch.kresch.ea:quantum}) that \(W^X\) is in bijection with pairs of integers $p_1 < p_2$ such that $p_1 + p_2 \neq 2n+1$. Explicitly, there exist integers $p_1 < p_2$ with $p_1 + p_2 \neq 2n+1$ and $q_1 < q_2$ with $q_1 + q_2 \neq 2n+1$ such that 
$$\begin{array}{l}
X_u =X_{p_1,p_2}= \{x \in X \ | V_x \cap E_{p_1} \neq 0 \textrm{ and } V_x \subset E_{p_2} \}, \\
X^v=X^{q_1,q_2}= \{x \in X \ | V_x \cap E^{q_1} \neq 0 \textrm{ and } V_x \subset E^{q_2} \}. \\
\end{array}$$ 
{We will also use the corresponding Schubert cells 
$$\begin{array}{l}
  {\oX}_u = \{x \in X \ | \dim(V_x \cap E_{p_1}) =1 \textrm{ and } V_x \subset E_{p_2} \}, \\
  \oX^v= \{x \in X \ | \dim(V_x \cap E^{q_1}) =1 \textrm{ and } V_x \subset E^{q_2} \}. \\
  \end{array}$$}
  
We set $\cO_{p_1,p_2} = \cO_u = [\cO_{X_u}]$ and define 
$$\delta_p = \left\{\begin{array}{ll}
0 & \textrm{if } p_1 + p_2 < 2n+1 \\
1 & \textrm{if } p_1 + p_2 > 2n+1 \\
\end{array}\right. \textrm{ and }
\delta_q = \left\{\begin{array}{ll}
0 & \textrm{if } q_1 + q_2 < 2n+1 \\
1 & \textrm{if } q_1 + q_2 > 2n+1. \\
\end{array}\right.$$
Note that \(\delta_p=0\) if and only if \(E_{p_1}^\perp\supseteq E_{p_2}\), and similarly for \(\delta_q\). We have \(
 \dim X_u= p_1+p_2-3-\delta_p\), and similarly for \(X^v\).

Consider the following conditions on the pairs $(p_1,p_2)$ and $(q_1,q_2)$:
\begin{equation}
\label{eqn:deg1}
\tag{C1} p_1 + q_1 = 2n = p_2 = q_2
\end{equation}
\begin{equation}
\label{eqn:deg2}
\tag{C2} \left\{\begin{array}{l}
p_1+q_2 =2n=p_2+q_1, \\
p_2 - p_1 = q_2 - q_1 \geq 2 \textrm{ and } \max(\delta_p,\delta_q) = 1.\\
\end{array}\right.
\end{equation}

The following is proved in Sections \ref{sec:d=2} and \ref{sec:d=1}, see Theorems \ref{thm:d=2} and \ref{thm:d=1}. We denote by (C$d$) the condition \eqref{eqn:deg1} for $d = 1$ and \eqref{eqn:deg2} for $d = 2$.

\begin{thm}
\label{thm:cd}
Assume that $d \in [1,2]$.
\begin{enumerate}
\item If \emph{(C$d$)} holds, then $\ev_3^{d-1,1}(u,v): M_{d-1,1}(X_u,X^v) \to \Gamma_{d-1,1}(X_u,X^v)$ is generically finite of degree $2$.
\item Otherwise, the map $\ev_3^{d-1,1}(u,v): M_{d-1,1}(X_u,X^v) \to \Gamma_{d-1,1}(X_u,X^v)$ is $\GRCF$.
\end{enumerate}
\end{thm}

Using this result, we compute $(\cO_u \star \cO^v)_q$ when (C$d$), for $d \in [1,2]$, does not hold.

\begin{cor}\label{cor:notCd}
Assume that neither \eqref{eqn:deg1} nor \eqref{eqn:deg2} holds, then
\begin{enumerate}
\item $(\cO_u \star \cO^v)_d = 0$ if and only if $\ev_3^d(u,v): M_d(X_u,X^v) \to \Gamma_d(X_u,X^v)$ is not birational.
\item $(-1)^{\ell(uvw) + \int_d c_1(T_X)} \kappa^{w,d}_{u,v} \geq 0.$
\end{enumerate}
\end{cor}

\begin{proof}
(1) This follows from Corollary \ref{cor:final-formula} and Theorems  \ref{thm:push-forward}, \ref{thm:all-d} and \ref{thm:cd}.

(2) If the map $\ev_3^d(u,v): M_d(X_u,X^v) \to \Gamma_d(X_u,X^v)$ is birational, then $\ev_3^{d-1,1}(u,v): M_{d-1,1}(X_u,X^v) \to \Gamma_{d-1,1}(X_u,X^v)$ is also birational by Theorem \ref{thm:all-d}. By a result of Brion (see \cite[Theorem 8.11]{buch.chaput.ea:positivity}) and that $M_d(X_u,X^v)$ and $M_{d-1,1}(X_u,X^v)$ have rational singularities, we have that the coefficients $c_w$ and $d_w$ in the expansions
$$(\ev_3)_*[\cO_{M_d(X_u,X^v)}] = \sum_w c_w \cO^w \textrm{ and } (\ev_3)_*[\cO_{M_{d-1,1}(X_u,X^v)}] = \sum_w d_w\cO^w$$
satisfy $(-1)^{\ell(w) - c}c_w \geq 0$ and $(-1)^{\ell(w) - d}d_w \geq 0$, where $c$ and $d$ are the codimensions of $\Gamma_d(X_u,X^v)$ and $\Gamma_{d-1,1}(X_u,X^v)$ in $X$. The result follows from this, Corollary \ref{cor:final-formula}, and Theorem \ref{thm:all-d}.(2).(b).
\end{proof}

When (C$d$) holds for $d \in [1,2]$, we explicitly compute the product $\cO_u \star \cO^v$, where \(\cO_{p_1,p_2}^{q_1,q_2}=\cO_u^v=[\cO_{X_u^v}]\).

\begin{prop}[see Proposition \ref{prop:2-to-1}]\label{prop1:2-to-1}
The following holds in \(\QK(X)\):
\begin{enumerate}
\item If \eqref{eqn:deg1} holds, then  $\cO_u \star \cO^v = \cO_u^v -q + q \cO_{2n-2,2n}$.
\item If \eqref{eqn:deg2} holds, then  
$$\cO_{p_1,p_2} \star \cO^{q_1,q_2} = q \cO_{p_1+q_1,2n}^{p_2+q_2-2n,2n}  - q^2 + q^2 \cO_{2n-2,2n}.$$ 
\end{enumerate}
\end{prop}
 
 It is easy to verify that the above products satisfy Conjecture \ref {conj:pos}, proving Conjecture \ref {conj:pos} in the remaining cases, thus completing the proof of Theorem \ref{main}. 

  \begin{remark}\label{rmk:interval}
    Assume that \(d\in[1,2]\) and neither \eqref{eqn:deg1} nor \eqref{eqn:deg2} holds. Then \((\cO_u\star\cO^v)_d\neq 0\) if and only if $\ev_3^d(u,v): M_d(X_u,X^v) \to \Gamma_d(X_u,X^v)$ is birational. If \(\ev_3^2(u,v)\) is birational, then by Lemma \ref{lemma:deg2-birat}, we have \(p_1+q_1\leq 2n-1\), which implies that \(\ev_3^1(u,v)\) is birational (see Section \ref{sec:d=1}). Together with Corollary \ref{cor:dleq2} and Proposition \ref{prop1:2-to-1}, this implies that the powers of \(q\) appearing in \(\cO_u\star\cO^v\) form an interval. The same argument implies that this \emph{interval property} also holds in quantum cohomology.
  \end{remark}

\section{Replacing stable maps by evaluations}
\label{section:replacing}

For proving the $\GRCF$ results stated in Theorems \ref{thm:all-d} and \ref{thm:cd}, we need a reduction step replacing stable maps by the images of marked points.

\subsection{Some results on general fibers}\label{sec:GRCF}

A property is true for the general fiber of $f$ if it is true for $f^{-1}(b)$ for $b$ in a dense open subset of $B$. We start with the following classical result.

\begin{lemma}
\label{lemma:fibre-meets-open}
Let $f: A \to B$ be a dominant morphism between irreducible varieties and let $\Omega \subset A$ be open and dense. Then $\Omega$ meets any component of a general fiber of $f$.
\end{lemma}

\begin{proof}
Set $W = A \setminus \Omega$. If $f(W)$ is not dense in $B$, then for $b \in B \setminus f(W)$, we have  $f^{-1}(b) \subset \Omega$. Otherwise, by the Fiber Dimension Theorem let $U \subset B$ be the open subset such that any component of $f^{-1}(b)$ has dimension $\dim A - \dim B$ and any component of $f^{-1}(b) \cap W$ has dimension $\dim W - \dim B$ for $b \in U$. Then, since $\dim W < \dim A$, any component of $f^{-1}(b)$ meets $\Omega$ for $b \in U$. 
\end{proof}

For convenience, we introduce the following notations.

\begin{defn}
\label{def:grcf}
Let $f : A \to B$ be a dominant morphism between irreducible varieties.
\begin{enumerate}
\item $f$ is called $\GIF$ (Generically Irreducible Fibers) if a general fiber of $f$ is irreducible.
\item $f$ is called $\GRCF$ (Generically Rationally Connected Fibers) if a general fiber of $f$ is rationally connected.
\end{enumerate}
\end{defn}

\begin{prop}
\label{prop:diagram}
Consider the following commutative diagram of {dominant morphisms between irreducible varieties}, with $A$ projective and normal:
$$\xymatrix{A \ar[rd]^-{f} \ar[d]_-{g} & \\
B \ar[r]^-h & C. \\}$$
\begin{enumerate}
\item Assume that $g$ and $h$ are $\GIF$, then $f$ is $\GIF$.
\item Assume that $g$ and $h$ are $\GRCF$, then $f$ is $\GRCF$.
\end{enumerate}
\end{prop}

\begin{proof}
(1) Since $A$ is normal, the general fiber of $f$ is normal by \cite[Lemma 3]{brion:positivity}, so it is enough to prove that this general fiber is connected. Let $U_B \subset B$ be an open dense subset such that $g^{-1}(b)$ is irreducible for $b \in U_B$ and let $U_C \subset C$ be an open dense subset such that $h^{-1}(c)$ is irreducible and $g^{-1}(U_B)$ meets any irreducible component of $f^{-1}(c)$ for $c \in U_C$ (use Lemma \ref{lemma:fibre-meets-open}). Let $c \in U_C$ and $f^{-1}(c) = E_1 \coprod \cdots \coprod E_n$ be its decomposition into connected components. Since $A$ is projective, $g(E_i) \subset h^{-1}(c)$ is a closed subset for all $i \in [1,n]$ and $h^{-1}(c) = \cup_{i = 1}^n g(E_i) $. Since $h^{-1}(c)$ is irreducible, we get that $g(E_i) = h^{-1}(c)$ for some $i \in [1,n]$. Up to reordering the indices, we may assume that $g(E_1) = h^{-1}(c)$. For any $i \in [1,n]$ and any $a \in E_i \cap g^{-1}(U_B)$, we have $g(a) \in g(E_1)$. Therefore, in the decomposition $g^{-1}(g(a)) = (g^{-1}(g(a)) \cap E_1)  \coprod \cdots \coprod (g^{-1}(g(a)) \cap E_n)$, we have $(g^{-1}(g(a)) \cap E_1) \neq \emptyset \neq (g^{-1}(g(a)) \cap E_i)$. But since $g(a) \in U_B$, the fiber $g^{-1}(g(a))$ is irreducible, thus $E_i  = E_1$ and $f^{-1}(c)$ is connected and irreducible.

(2) From (1), we have that the general fiber of $f$ is irreducible. We also have a map $g : f^{-1}(c) \to h^{-1}(c)$ with base and general fibers rationally connected. The result follows from \cite[Corollary 1.3]{graber.harris.starr}.
\end{proof}

\begin{lemma}
\label{lemma:restriction-to-open}
Let $f: A \to B$ be a {dominant morphism between irreducible varieties}. The following conditions are equivalent:
\begin{enumerate}
\item $f$ is $\GRCF$.
\item $f\vert_\Omega$ is $\GRCF$ for any dense open subset $\Omega \subset A$.
\item $f\vert_\Omega$ is $\GRCF$ for some dense open subset $\Omega \subset A$.
\end{enumerate}
\end{lemma}

\begin{remark}
  Here GRCF can be replaced by the property of having generically rational fibers.
\end{remark}

\begin{proof}
Assume that $f$ is $\GRCF$ and let $\Omega \subset A$ be a dense open subset. Let $U \subset B$ be a dense open subset such that any component of $f^{-1}(b)$ meets $\Omega$ for $b \in U$. Then $f^{-1}(b) \cap \Omega$ is dense in $f^{-1}(b)$ and therefore has to be rationally connected, proving the implication (1) $\Rightarrow$ (2). The implication (2) $\Rightarrow$ (3) is clear. Assume that $f\vert_\Omega$ is $\GRCF$ for some dense open subset $\Omega \subset A$. Let $U \subset B$ be a dense open subset such that any component of $f^{-1}(b)$ meets $\Omega$ for $b \in U$. Then $f^{-1}(b) \cap \Omega$ is dense in $f^{-1}(b)$, thus $f^{-1}(b)$ is irreducible and rationally connected, proving the last implication.
\end{proof}

\subsection{Rational curves passing through points}

{We prove some lemmas on the geometry of curves in $X$ useful for our reduction step proved in Section~\ref{sec:replace}.}

\begin{lemma}
  \label{lem:split-deg4}
  Let $x,y \in X$ be two points in general position and $d \geq 4$, then $\Gamma_d(x,y) = X = \Gamma_{d-1,1}(x,y)$.
  \end{lemma}
  
  \begin{proof}
  Fix $z \in X$ general and let $V = V_x \oplus V_y$. Since $x$ and $y$ are general, we have $\dim V = 4$ and the symplectic form has rank $4$ on $V$. Pick $t \in X$ general such that $V_t \subset V$. Set $W = V_t + V_z$. Then we have $\dim W = 4$ and the symplectic form has rank $4$ on $W$. We thus have $x,y,t \in \IG(2,V)$ and $t,z \in \IG(2,W)$. Furthermore, $\IG(2,V)$ and $\IG(2,W)$ are smooth quadrics of dimension $3$, therefore there exists a conic passing through $x$, $y$ and $t$, and a conic as well as a chain of two lines passing through $t$ and $z$, proving the claim.
  \end{proof}

  \begin{lemma}
    \label{lem:split-deg3}
    Let $x,y \in X$ be two points in general position and let $z \in X$. We have the equivalence:
    $$z \in \Gamma_3(x,y)\ \Longleftrightarrow\ \dim(V_x + V_y + V_z) \leq 5\ \Longleftrightarrow\ z \in \Gamma_{2,1}(x,y).$$
    \end{lemma}
    
    \begin{proof}
      If $z \in \Gamma_3(x,y)$, then by the Ker/Span technique (see \cite{buch:quantum,buch.kresch.ea:gromov-witten}), we must have  $ \dim(V_x + V_y + V_z) \leq 5$. Assume that this condition holds, it is enough to prove that $z \in \Gamma_{2,1}(x,y)$. Let $V = V_x \oplus V_y$. Since $x$ and $y$ are general, we have $\dim V = 4$ and the form has rank $4$ on $V$. Since $\dim(V + V_z) \leq 5$, there exists a non-zero vector $v \in V \cap V_z$. Pick $w \in {V_z}^\perp \cap V \setminus \langle v \rangle$ and define $t \in X$ via $V_t = \langle v,w \rangle$. We have $V_t \subset V$ thus $x,y,t \in \IG(2,V)$. {Since $\IG(2,V)$ is a smooth quadric,} there is a conic passing through \(x, y, t\). On the other hand, we have $v \in V_t \cap V_z$ thus $t$ and $z$ are on a line, proving the claim.
      \end{proof}    

  \begin{lemma}
    \label{lem:deg2}
    Let $x,y \in X$ be two points in general position and let $z \in X$. We have the equivalence: $$z \in \Gamma_2(x,y)\ \Longleftrightarrow\ \dim(V_x + V_y + V_z) \leq 4.$$
    \end{lemma}
    
    \begin{proof}
    If $z \in \Gamma_2(x,y)$, then by the Ker/Span technique (see \cite{buch:quantum,buch.kresch.ea:gromov-witten}), we must have  $ \dim(V_x + V_y + V_z) \leq 4$. Assume that this condition holds, we prove $z \in \Gamma_{2}(x,y)$. Let $V = V_x \oplus V_y$. Since $x$ and $y$ are general, we have $\dim V = 4$ and the symplectic form has rank $4$ on $V$. We have $V_z \subset V$ thus $z \in \Gr(2,V) \cap X$. Since $\Gr(2,V) \cap X$ is a smooth quadric of dimension $3$, there is a conic passing through \(x, y, z\).
    \end{proof}

\subsection{Forgetting maps}\label{sec:replace}

Recall the definition of $M_\du$ and the evaluation maps $\ev_i : M_\du \to X$ for $i \in [1,3]$. {Let us start by fixing some notation on evaluation maps.
\medskip

\begin{notation}[More evaluation maps]\label{def:ev}
For \(1\leq i \leq 3\) define $\ev_{(i)}^\du : M_\du \to X^i$ by $\ev_{(i)}^\du = \prod_{j=1}^{i}\ev_j$. Set $\Gamma_\du^{(i)} = \ev_{(i)}^\du(M_\du)$. We have surjective morphisms:
\[
  \ev_{(i)}^\du : M_\du \to \Gamma_\du^{(i)} \text{ and } \pi_3^\du :\Gamma_\du^{(3)} \to  X,
\]
where the latter is the projection onto the last component.
Set $\Gamma_\du^{(3)}(X_u,X^v) = \ev_{(3)}(M_\du(X_u,X^v))$. The maps \(\ev_{(3)}^\du\) and $\pi_3^\du$ induce surjective maps  
$$\begin{array}{l}
\ev_{(3)}^\du(u,v): M_\du(X_u,X^v) \to \Gamma_\du^{(3)}(X_u,X^v)\text{ and }\pi_3^\du(u,v): \Gamma_d^{(3)}(X_u,X^v) \to \Gamma_\du(X_u,X^v).
\end{array}$$
\end{notation}}

\begin{prop} 
\label{prop:md-to-gammad3}
The map $\ev_{(3)}^d : M_d \to \Gamma_d^{(3)}$ is $\GRCF$.
\end{prop}

\begin{proof}
We will prove that \(\ev_{(3)}^d\) has rational general fibers and is birational if and only if \(d\leq 2\).

For \(d=0\), the map $\ev_i$ is an isomorphism for \(i\in [1,3]\). For $d > 0$, by Lemma \ref{lemma:restriction-to-open}, we can restrict to the open subset of $M_d$ given by stable maps of the form $f: \bP^1 \to X$ of degree $d$ which are birational onto their image. A degree \(1\) map \(\P^1\to X\) is an isomorphism onto its image, and \(3\) general points in \(\Gamma_1^{(3)}\) determine a line. In particular, for \(d=1\), the map {$\ev_{(3)}^d$} is birational. A birational degree \(2\) map \(\P^1\to X\) is an isomorphism onto its image. For $(x,y,z) \in \Gamma_2^{(3)}$ general, we have $\dim(V_x + V_y + V_z) = 4 =\Rk(\omega\vert_{V_x + V_y + V_z})$ by Lemma~\ref{lem:deg2}. Set $E = V_x + V_y + V_z$, then $x,y,z \in \IG(2,E)$ which is a smooth quadric of dimension $3$. Any conic passing though $x,y,z$ is contained in $\IG(2,E)$ and there is therefore a unique such conic. In particular, for \(d=2\), the map {$\ev_{(3)}^d$} is birational. 

Assume \(d\geq 3\). Let $\fp_k \subset \C^{2n}[s,t]$ be the subspace of homogeneous polynomials of degree $k$ with coefficients in $\C^{2n}$ and let $(p_1,p_2,p_3) = (0,1,\infty) \in (\bP^1)^3$. For $a = \lfloor \frac{d}{2} \rfloor$ and $b = \lceil \frac{d}{2} \rceil$, we set $\xu = (x_1,x_2,x_3) \in X^3$ and
$$Z =  \left\{(P,Q,\xu) \in \fp_a\times \fp_b \times X^3 \left| \!
\begin{array}{l}
\textrm{For } [s,t] \in \bP^1, \left\{ 
\begin{array}{l}
(P \wedge Q)(s,t) \neq 0, \\
\omega(P,Q)(s,t) = 0,  \\
\end{array}\right. \\
 \textrm{For } i \in [1,3], \ [(P \wedge Q)(p_i)] = x_i \\
\end{array}
\right.
\!\!\!\right\}.$$
 We  have a dominant map $Z \to M_d$, $(P,Q,\xu) \mapsto ([s:t] \mapsto [P(s,t) \wedge Q(s,t)])$ with {general} fibers {irreducible} of constant dimension $4$. In particular, $\dim Z = d(2n-1) + 4n - 1$ and there is a unique irreducible component $Z_\circ$ of $Z$ of dimension $\dim Z$. The component $Z_\circ$ dominates $M_d$. Let $\pi : Z_\circ \to \Gamma_d^{(3)}$ be the projection on the last factors. By Lemma \ref{lemma:fibre-meets-open}, the general fiber of $\pi$ dominates the general fiber of $\ev^d_{(3)}$. We therefore only need to prove that $\pi$ is $\GRCF$.
 
 Let $Y = \{(P,\xu) \in \fp_a \times X^3 \ | \ P(s,t) \neq 0 \textrm{ for all } [s,t] \in \bP^1 \textrm{ and } P(p_i) \in V_{x_i}  \textrm{ for } i \in [1,3] \}$. Projections induce maps $p : Z \to Y$, $p_\circ : Z_\circ \to Y$ and $\ev_P : Y \to X^3$ such that $\pi = \ev_P \circ p_\circ$. Note that $\ev_P^{-1}(\xu)$ is an open subset of a vector space and is therefore rational.
 
 We first prove that $Y$ is irreducible. Let $S = \{P \in \fp_a \ | \ P(s,t) \neq 0 \textrm{ for all } [s,t] \in \bP^1 \}$. The projection on the first factor induces a map $\sigma : Y \to S$, and we have a cartesian diagram:
 $$\xymatrix{Y \ar[r]^-{\theta} \ar[d]_-\sigma & \Fl^3 \ar[d]^{{\pr}_2} \\
 S \ar[r]^-\tau& (\bP^{2n-1})^3, \\}$$
 where $\tau : S \to (\bP^{2n-1})^3$ is defined by $\tau(P) = ([P(p_1)],[P(p_2)],[P(p_3)])$, $\Fl = \{(x,y) \in X \times \bP^{2n-1} \ | \ y \subset V_x \}$ is the incidence variety between $X$ and $\bP^{2n-1}$, and the map $\theta : Y \to \Fl^3$ is defined by $\theta(P,\xu) = (\xu,[P(p_1)],[P(p_2)],[P(p_3)])$. 
Since ${\pr}_2$ is locally trivial with fibers $(\bP^{2n-3})^3$, the same is true for $\sigma$. Since furthermore $S$ is an open subset of $\fp_a$, we get that $Y$ is irreducible of dimension  $2n(a+1) + 3(2n-3)$. 

We now consider the fibers of the map $p : Z \to Y$. These fibers are given by an open subset in the set of solutions of the linear system on $Q$ given by
\begin{equation}\label{eqn:syst}
  \left\{ \begin{array}{l}
\omega(P,Q)(s,t) = 0 \textrm{ for all $[s,t] \in \bP^1$, } \\
 Q(p_i) \in V_{x_i} \textrm{ for $i \in [1,3]$.} \\
 \end{array}\right.
\end{equation}
  Note that this linear system is generically of rank 
  \begin{equation*}\label{eqn:bound}
    \dim\fp_b - (\dim Z - \dim p(Z)) \leq  d-2 + 3(2n-2)=\dim\fp_b - (\dim Z - \dim Y).
  \end{equation*}
    Let $(e_i)_{i \in [1,2n]}$ be the standard basis of $\C^{2n}$ so that the symplectic form is given by $\omega(e_i,e_j) = \delta_{j,2n+1-i}$ for $i \in [1,n]$, let $P(s,t) = s^a e_1 + t^a e_{2n}$ and let $V_{x_1}= \langle e_1,e_2 \rangle$,  $V_{x_2}= \langle e_1 + e_{2n},e_2 + e_{2n-1} \rangle$ and $V_{x_3}= \langle e_{2n-1},e_{2n} \rangle$. Then $(P, \xu) \in Y$. For $Q(s,t) = s^b e_2 + t^b e_{2n-1}$, we have $(P,Q,\xu) \in p^{-1}(P,\xu)$ and an easy check proves that the fiber $p^{-1}(P, \xu)$ is given by an open subset of the solutions of a linear system of rank $d-2 + 3(2n-2)$. 

    Let $Y_\circ \subset Y$ be the open subset 
    where the system \eqref{eqn:syst} has maximal rank $d-2 + 3(2n-2)$. Then the map $p : p^{-1}(Y_\circ) \to Y_\circ$ is locally trivial with fiber an open subset of a vector space of dimension $\dim Z - \dim Y$. This implies that $Z_\circ$ is the unique component of $Z$ dominating $Y$ and that the map $p : Z_\circ \cap p^{-1}(Y_\circ) \to Y_\circ$ is locally trivial with fiber an open subset of a vector space of dimension $\dim Z - \dim Y$. We also have \(\dim \Gamma_d^{(3)} = \dim \pi(Z_\circ)=\dim\ev_P(Y)\).

{Pick $\xu  \in \Gamma_d^{(3)}$ general. Note that $\ev_P^{-1}(\xu)$ is an open subset of a vector space and thus rational. Set $F = \pi^{-1}(\xu)$. Any irreducible component of $F$ meets $p^{-1}(Y_\circ)$ non-trivially by Lemma~\ref{lemma:fibre-meets-open} and is of dimension $\dim Z - \dim \Gamma_d^{(3)}$. Let $F'$ be any such irreducible component. We have a map $p : F' \to \ev_P^{-1}(\xu)$ whose general fiber {over the image }has dimension $\dim Z - \dim Y = \dim F' - \dim \ev_P^{-1}(\xu)$. In particular, $p : F' \to \ev_P^{-1}(\xu)$ is dominant. On the other hand, since \(\ev_P^{-1}(\xu)\) meets \(Y_\circ\), the map \(p:F\to \ev_P^{-1}(\xu)\) is a locally trivial fibration with rational fibers over a dense open subset of \(\ev_P^{-1}(\xu)\). Therefore, \(F\) must be rational.}

By Lemma \ref{lem:split-deg3}, we have \(\Gamma_3^{(3)}\) dominates \(X^2\) with fibers of dimension at most \(2n\). Hence, \(\dim\Gamma_3^{(3)}\leq 2(4n-5)+2n=10n-10\). By Lemma \ref{lem:split-deg4}, we have \(\dim\Gamma_d^{(3)}=3(4n-5)\) for \(d\geq4\). Therefore, \(\dim Z-\dim \Gamma_3^{(3)}\geq 6\) and \(\dim Z-\dim \Gamma_d^{(3)}=d(2n-1)-8n+14\geq 10\) when \(d\geq 4\). This implies that when \(d\geq 3\), the fibers of \(\ev_{(3)}^d\) are of dimension at least \(2\).
\end{proof}

{Recall that $M_{d-1,1} = M_{d-1} \times_X \overline{M}_{0,2}(X,1)$ and consider the map $\ev_{(4)} : M_{d-1,1} \to X^4$ defined by $\ev_{(4)} = \ev_1 \times \ev_2 \times \ev_s\times \ev_3 $, where $\ev_s$ is the map induced by the fiber product structure and is obtained by evaluation at the singular point of the curve. Define $ \Gamma_{d-1,1}^{(4)} = \ev_{(4)}(M_{d-1,1})$. We have a commutative diagram
\begin{equation}\label{eqn:diag}
  \xymatrix{
M_{d-1,1} \ar[rrd]^{\ev_{(4)}} \ar[d]_{{{\ev^{d-1}_{(3)}\times\id}}} & & \\
 \Gamma_{d-1}^{{{(3)}}} \times_X\overline{M}_{0,2}(X,1) \ar[rr]_-{\id \times \ev_2} & & \Gamma_{d-1,1}^{(4)}.}
\end{equation}}

  \begin{lemma}\label{lem:pr}
    The projection $\pr : \Gamma_{d-1,1}^{(4)} \to \Gamma_{d-1,1}^{(3)}$ is $\GRCF$.
  \end{lemma}
  \begin{proof}
    We need to discuss cases depending on the degree $d$. If $d > 4$, then by Lemma \ref{lem:split-deg4}, $\pr^{-1}(x,y,z) = \Gamma_1(z)$ is a Schubert variety thus rational. If $d = 1$, then $\pr^{-1}(x,y,z) = \{x\}$. For $d = 2$ and $(x,y,z)$ general in $\Gamma_{1,1}^{(3)}$, then $E = V_x + V_y + V_z$ has dimension $4$ and the restriction $\omega\vert_E$ has rank $4$. Therefore, $x,y,z$ are points on a $3$-dimensional quadric $Q$, given by $Q = \IG(2,E)$, with $x$ and $y$ on a line $(xy)$ and $z$ general. It is easy to check that $\pr^{-1}(x,y,z) = (xy) \cap z^\perp$ is a unique point on $Q$.  For $d = 3$ and $(x,y,z)$ general in $\Gamma_{2,1}^{(3)}$, then $V_x + V_y + V_z$  has dimension $5$, the restriction $\omega\vert_{V_x+V_y}$ has rank $4$ and the intersection $V_z \cap (V_x + V_y)$ has dimension $1$ (by Lemma \ref{lem:split-deg3}). It is easy to check that $\pr^{-1}(x,y,z) = \{t \in X \ |\ (V_x+V_y) \cap V_z \subset V_t \subset V_x + V_y \} \simeq \bP^1$. Finally, for $d = 4$, we have  $\Gamma_{3,1}^{(3)} = X^3$ (by Lemma \ref{lem:split-deg4}). For $(x,y,z) \in X^3$ general, the fiber is given by $\pr^{-1}(x,y,z) = \{t \in X \ |\ V_t = \langle a,b \rangle \textrm{ with } a \in V_x + V_y, b\in V_z \textrm{ and } \omega(a,b) = 0\}$, which is a general hyperplane section of $\bP^3 \times \bP^1$ in its Segre embedding and therefore rational.
  \end{proof}

\begin{cor}
\label{cor:d-1-1}
The map $\ev_{(3)}^{d-1,1} : M_{d-1,1} \to \Gamma_{d-1,1}^{(3)}$ is $\GRCF$.
\end{cor}

\begin{proof}
All maps in \eqref{eqn:diag} are surjective and $M_{d-1,1}$ is irreducible, projective and normal. Furthermore, $\id\times \ev_2$ is birational and the vertical map is $\GRCF$ by Proposition \ref{prop:md-to-gammad3}. By Proposition \ref{prop:diagram}, the map $\ev_{(4)}$ is therefore $\GRCF$. Finally, use Lemma~\ref{lem:pr}.
\end{proof}

We conclude with the main reduction result of this section. 

\begin{prop}\label{prop:replace} We have the following implications. 
\begin{enumerate}
\item If $\pi_3^{d}(u,v)$ is $\GRCF$, then $\ev_3^{d}(u,v)$ is $\GRCF$;
\item If $\pi_3^{d-1,1}(u,v)$ is $\GRCF$, then $\ev_3^{d-1,1}(u,v)$ is $\GRCF$.  
\end{enumerate}
\end{prop}

\begin{proof}
We have the following diagrams of surjective maps
$$\xymatrix{M_d(X_u,X^v) \ar[rd]
 \ar[d]
   & \\
\Gamma_d^{(3)}(X_u,X^v) \ar[r]
 & \Gamma_d(X_u,X^v) \\}  \textrm{ and }
\xymatrix{M_{d-1,1}(X_u,X^v) \ar[rd]
 \ar[d]
 & \\
\Gamma_{d-1,1}^{(4)}(X_u,X^v) \ar[r]
 & \Gamma_{d-1,1}(X_u,X^v). \\}$$
Since the \(G\)-orbit of \(X_u\times X^v\) is dense in \(X^2\) and the evaluation maps are \(G\)-equivariant, the general fibers of the vertical maps are general fibers of the maps 
$$\ev_{(3)}^d : M_{d} \to \Gamma_{d}^{(3)} \textrm{ and } \ev_{(4)}^{d-1,1} : M_{d-1,1} \to \Gamma_{d-1,1}^{(4)}.$$
Furthermore, $M_{d}(X_u,X^v)$ and $M_{d-1,1}(X_u,X^v)$ are projective, irreducible and normal, thus the result follows from Propositions \ref{prop:diagram} and \ref{prop:md-to-gammad3}, {Lemma~\ref{lem:pr},} and Corollary \ref{cor:d-1-1}.
\end{proof}

\begin{cor}
\label{cor_strategy}
If $\Gamma_d(x,y)=\Gamma_{d-1,1}(x,y)\neq \emptyset$ for $x,y\in X$ in general position, then we have $\Gamma_d(X_u,X^v)=\Gamma_{d-1,1}(X_u,X^v)$. If furthermore $\pi_3^d(u,v)$ is $\GRCF$, then $\ev_3^{d-1,1}(u,v)$ is $\GRCF$.
\end{cor}

\begin{proof}
Since $M_{d}(X_u,X^v)$ and $M_{d-1,1}(X_u,X^v)$ are irreducible, the same is true for $\Gamma^{(3)}_{d}(X_u,X^v)$ and $\Gamma^{(3)}_{d-1,1}(X_u,X^v)$. {Let \(\Omega\) be the open subset of \(X^2\) such that \(\Gamma_d(x,y)=\Gamma_{d-1,1}(x,y)\neq \emptyset\) for \((x,y)\in \Omega\). Since the \(G\)-orbit of \(X_u\times X^v\) is dense in \(X^2\) and the evaluation maps are \(G\)-equivariant, we may assume \(\Omega\cap (X_u\times X^v)\neq\emptyset\).} Then 
$$\displaylines{
\{(x,y,z) \in X_u \times X^v \times X \ | \ (x,y) \in \Omega \textrm{ and } z \in \Gamma_d(x,y) \}
\hfill\cr\hfill
= \{(x,y,z) \in X_u \times X^v \times X \ | \ (x,y) \in \Omega \textrm{ and } z \in \Gamma_{d-1,1}(x,y) \}
\cr}$$
is a dense open subset of both $\Gamma_d^{(3)}(X_u,X^v)$ and $\Gamma_{d-1,1}^{(3)}(X_u,X^u)$, proving that $\Gamma^{(3)}_{d}(X_u,X^v)=\Gamma^{(3)}_{d-1,1}(X_u,X^v)$. We deduce that $\Gamma_d(X_u,X^v)=\Gamma_{d-1,1}(X_u,X^v)$ and $\pi_3^d(u,v)=\pi_3^{d-1,1}(u,v)$. The last assertion follows from Proposition \ref{prop:replace}.
\end{proof}

\subsection{Results on curve neighborhoods}

Before going on, let us prove some preliminary results on curve neighborhoods. More precisely, the next result will be useful for proving that $\Gamma_d(X_u,X^v)$ and $\Gamma_{d-1,1}(X_u,X^v)$ have rational singularities when they are equal and $d$ is small.

\begin{prop}
\label{prop_rat_sing_smalldegreecurves}
Let $V_k$ be any $k$-dimensional subspace of $\CC^{2n}$. The subvariety $\{z\in X\mid V_z\cap V_k \neq 0\}\subset X$ has rational singularities.
\end{prop}

\begin{proof}
Let $S=G(k,2n)$ and for $s \in S$, denote by $V_s \subset \CC^{2n}$ the corresponding subspace. Define  $Y = \{ (z,s) \in X \times S \ | \ V_z\cap V_s \neq 0 \}$. 
The map $Y \to X$ is $G$-equivariant and therefore locally trivial. Furthermore, 
the fiber of this map is the Schubert variety $\{ s \in S \ | \ V_s \cap V_z \neq 0\}$ and therefore has rational singularities. In particular $Y$ has rational singularities. On the other hand, the fibers of the map $Y \to S$ are of the form $\{ z \in X \ | \ V_z \cap V_s \neq 0\}$, and it is easy to check that these varieties have constant dimension $k + 2n - 4$ (a birational model is given by $\{(a,z) \in \bP(V_s) \times X \ | \ a \subset V_z \subset a^\perp\}$, which is a locally trivial $\bP^{2n-3}$-bundle over $\bP(V_s)$). In particular, by Miracle Flatness, the projection $Y \to S$ is flat. {Finally, the family $(Y_s)_{s\in S}$ admits a global resolution of singularities given by the variety $\tilde{Y} = \{ (l,z,s) \in \PP^{2n-1}\times X \times S \ | \ l\subset V_z\cap V_s \}$. We can thus apply} \cite[Theorem 3]{Elkik}, and we deduce that $Y_s$ has rational singularities for any $s\in S$. 
\end{proof}

Out next result will apply when $\ev^d_3$ is birational, giving a useful relationship among curve neighborhoods for degree $(d-1,1)$ curves and degree $d$ curves. 

\begin{lemma}
Let us assume that $d\leq 2$ and that $\ev^d_3(u,v)$ and $\ev^{d-1,1}_3(u,v)$ are birational. Then $\Gamma_{d-1,1}(X_u,X^v)$ is a divisor inside $\Gamma_{d}(X_u,X^v)$.
\end{lemma}

\begin{proof}
Recall that $\Gamma_{d-1,1}(X_u,X^v)$ and $\Gamma_{d}(X_u,X^v)$ are irreducible. We thus only need to compute the codimension of the former in the latter. By the birationality assumption, the dimension of $\Gamma_d(X_u,X^v)$ and $\Gamma_{d-1,1}(X_u,X^v)$ are equal to the dimension of $M_{d}(X_u,X^v)$ and $M_{d-1,1}(X_u,X^v)$. Note that $M_{d-1,1}$ is a union of boundary divisors (see \cite{fulton.pandharipande:notes}) an is therefore a divisor in $M_d$. Now by Kleiman-tBertini applied to $\ev_1 \times \ev_2 : M_d \to X^2$, we get that if $M_{d-1,1}(X_u,X^v)$ is non-empty it is a divisor in $M_d(X_u,X^v)$.
\end{proof}

\section{Large degrees}
\label{sec:d=3}

The aim of this section is to prove the following result.

\begin{thm}
\label{thm:d=3}
For $d \geq 3$, we have
\begin{enumerate}
\item The map $\ev_{3}^d(u,v): M_d(X_u,X^v) \to \Gamma_d(X_u,X^v)$ is $\GRCF$ and never birational;
\item The map $\ev_3^{d-1,1}(u,v): M_{d-1,1}(X_u,X^v) \to \Gamma_{d-1,1}(X_u,X^v)$ is $\GRCF$;
\item $\Gamma_d(X_u,X^v) = \Gamma_{d-1,1}(X_u,X^v)$ has rational singularities.
\end{enumerate}
\end{thm}

\begin{proof}
The results follow from Propositions \ref{prop:md-to-gammad3}, \ref{prop:replace}, \ref{prop:fibre-d-geq4}, \ref{prop:fibre-d-3} and Lemmas \ref{lem:split-deg4} and \ref{lem:sing-rat3}.
\end{proof}

In all cases of the Theorem above our strategy will be to use Proposition \ref{prop:replace} (and this strategy will be used also for degree two and degree one curves). Since the case of degree $d\geq 4$ curves is much easier, let us deal with it separately.

\begin{prop}
\label{prop:fibre-d-geq4}
If $d\geq 4$, then the morphisms $\ev^d_3(u,v) : M_d(X_u,X^v) \to \Gamma_d(X_u,X^v)$ and $\ev^{d-1,1}_3(u,v) : M_{d-1,1}(X_u,X^v) \to \Gamma_{d-1,1}(X_u,X^v)$ are $\GRCF$.
\end{prop}

\begin{proof}
By assumption and Lemma \ref{lem:split-deg4}, we have $\Gamma^{(3)}_d(X_u,X^v)=\Gamma^{(3)}_{d-1,1}(X_u,X^v)=X_u\times X^v\times X$ and $\Gamma_d(X_u,X^v) =\Gamma_{d-1,1}(X_u,X^v) = X$. The fibers of $\pi_3^d(u,v)$ and $\pi_3^{d-1,1}(u,v)$ are isomorphic to $X_u\times X^v$ and are thus rationally connected (even rational). The result follows from Proposition \ref{prop:replace}.
\end{proof}

Notice that when $d\geq 4$, $\Gamma_d(X_u,X^v)=\Gamma_{d-1,1}(X_u,X^v)=X$ is smooth. We now focus on degree $d=3$ curves. We start with a general statement describing the image of the evaluation map for cubics. Recall that $X_u$ and $X^v$ as incidence varieties can be defined as $X_u = \{ V_2 \in X \ | \ \dim(V_2 \cap E_{p_1}) \geq 1 \textrm{ and } V_2 \subset E_{p_2} \}$ and $X^v = \{ V_2 \in X \ | \ \dim(V_2 \cap E^{q_1}) \geq 1 \textrm{ and } V_2 \subset E^{q_2} \}$ and let $\oX_u \subset X_u$ and $\oX^v \subset X^v$ be the Schubert cells.

\begin{lemma}
\label{lem:sing-rat3}
We have $\Gamma_{2,1}(X_u,X^v)=\Gamma_3(X_u,X^v) = \{ z \in X \ | \ \dim(V_z \cap (E_{p_2} + E^{q_2})) \geq 1 \}$ and this variety has rational singularities.
\end{lemma}

\begin{proof}
The first equality is a consequence of Lemma \ref{lem:split-deg3} and of the fact that the $G$ orbit of $X_u\times X^v$ is dense in $X^2$. The fact that $ \{ z \in X \ | \ \dim(V_z \cap (E_{p_2} + E^{q_2})) \geq 1 \}$ has rational singularities is a consequence of Proposition \ref{prop_rat_sing_smalldegreecurves}. So let us prove the second equality. Assume that $z \in \Gamma_3(X_u,X^v)$, then there exists $x \in X_u$, $y \in X^v$ and $V_5 \in \Gr(5,2n)$ such that $V_x + V_y + V_z \subset V_5$. In particular $\dim(V_z \cap (V_x + V_y))\geq 1$ and $V_x + V_y \subset E_{p_2} + E^{q_2}$, proving the inclusion of the LHS in the RHS. Assume now that $z \in X$ satisfies $\dim(V_z \cap (E_{p_2} + E^{q_2})) \geq 1$ and pick a non-zero $a \in V_z \cap (E_{p_2} + E^{q_2})$. Write $a = a' + a''$ with $a' \in E_{p_2}$ and $a'' \in E^{q_2}$. For $z$ general, we may assume that $a' \not\in E_{p_1}$ and $a'' \not\in E^{q_1}$.

Assume first that $p_1 \neq 1 \neq q_1$. Then $E_{p_1} \cap (a')^\perp$ and $E_{q_1} \cap (a'')^\perp$ are non-trivial, and we may pick $b' \in E_{p_1} \cap (a')^\perp$ and $b'' \in E_{q_1} \cap (a'')^\perp$. Then set $V_x = \langle a',b' \rangle$ and $V_y = \langle a'',b'' \rangle$.  We have $x \in X_u$ and $y \in X^v$ and $\dim(V_x + V_y + V_z) \leq 5$ thus $z \in \Gamma_3(x,y) \subset \Gamma_3(X_u,X^v)$. If $p_1 = 1$ {{then}} $p_2 < 2n$. We have $E_{p_2} \subset E_{p_1}^\perp$ and $E_{p_1} \cap (a')^\perp = E_{p_1}$, so we may proceed as above. The same works if $q_1 = 1$. 
\end{proof}

In order to prove Theorem \ref{thm:d=3}, we are left to prove the following result.

\begin{prop}
\label{prop:fibre-d-3}
The morphisms $\ev^3_3(u,v) : M_3(X_u,X^v) \to \Gamma_3(X_u,X^v)$ and $\ev^{2,1}_3(u,v) : M_{2,1}(X_u,X^v) \to \Gamma_{2,1}(X_u,X^v)$ are $\GRCF$.
\end{prop}

We want to use Proposition \ref{prop:replace}, but we also want to avoid dealing with ``badly behaving'' curves. Because of Lemma \ref{lemma:restriction-to-open}, we can restrict to studying the fibers of $\pi_3^3(u,v)|_{\cU_3}:\Gamma_3^{(3)}(X_u,X^v) \to \Gamma_3(X_u,X^v)$, where $\mathcal{U}_3\subset \Gamma_3^{(3)}(X_u,X^v)$ is a well-chosen dense open subset. We postpone the proof of Proposition \ref{prop:fibre-d-3} to the construction of $\cU_3$ and the description of its properties.\medskip

For $(x,y,z) \in \Gamma_3^{(3)}(X_u,X^v)$, consider the following conditions:
\begin{enumerate}
\item\label{item:1} $V_x \cap V_y = 0$, 
\item\label{item:2} $\dim(V_x + V_y + V_z) = 5$, 
\item\label{item:xy}\(V_x\in\oX_u\) and \(V_y\in\oX^v\), 
\item\label{item:5} $V_z \cap (V_x + V_y) \not\subseteq [(V_x \cap E_{p_1}) +V_y] \cup [V_x +(V_y \cap E^{q_1})]$.
\end{enumerate} 
Condition \eqref{item:1} is a non-empty open condition, and by Lemma \ref{lem:split-deg3}, 
\begin{multline*}
  \{(x,y,z)\in \Gamma_3^{(3)}(X_u,X^v)\ |\ (V_x,V_y,V_z)\text{ satisfies }\eqref{item:1}\}=\\=\{(x,y,z)\in X_u\times X^v\times X\ |\ V_x\cap V_y=0,\ \dim(V_x+V_y+V_z)\leq5\}.
\end{multline*}
Thus, \eqref{item:2} and \eqref{item:xy} are non-empty open conditions, and 
\begin{multline*}
  \{(x,y,z)\in \Gamma_3^{(3)}(X_u,X^v)\ |\ (V_x,V_y,V_z)\text{ satisfies \eqref{item:1}-\eqref{item:xy}}\}=\\=\{(x,y,z)\in \oX_u\times\oX^v\times X\ |\ V_x\cap V_y=0,\ \dim(V_x+V_y+V_z)=5\}.
\end{multline*}
For \((x,y,z)\) in the latter set, \(\dim(V_x\cap E_{p_1})=\dim(V_y\cap E^{q_1})=1\), and $(V_{x} \cap E_{p_1}) + V_{y}$ and $V_{x} + (V_{y} \cap E^{q_1})$ are both \(3\)-dimensional subspaces of the $4$-dimensional space $V_x+ V_y$, so \eqref{item:5} is a non-empty open condition and 
\[
  \mathcal{U}_3:=\{(x,y,z)\in\Gamma_3^{(3)}(X_u,X^v)\ |\ (V_x,V_y,V_z)\text{ satisfies \eqref{item:1}-\eqref{item:5}} \}
\]
is an open dense subset of \(\Gamma_3^{(3)}(X_u,X^v)\). 

\begin{prop}\label{prop:pi3_unirat}
  For \(z\in\Gamma_3(X_u,X^v)\) general, \((\pi_3^3(u,v))^{-1}(z)\cap\mathcal{U}_3\) is unirational.
\end{prop}
\begin{proof} 
Set
$$\begin{array}{l}
E_\circ  = \left\{
\begin{array}{ll} 
E_{p_2} \setminus E_{p_1} & \textrm{ if } 
E_{p_2} \subseteq E_{p_1}^\perp\\
E_{p_2} \setminus (E_{p_1} \cup E_{p_1}^\perp)  & \textrm{ otherwise, } \\
\end{array}\right.\\
E^\circ  = \left\{\begin{array}{ll} 
E^{q_2} \setminus E^{q_1} & \textrm{ if } 
E^{q_2} \subseteq (E^{q_1})^\perp\\
E^{q_2} \setminus (E^{q_1} \cup (E^{q_1})^\perp)  & \textrm{ otherwise.} \\
\end{array}
\right.
\end{array}$$
We have that $E_\circ$ and $E^\circ$ are open subsets in \(E_{p_2}\) and \(E^{q_2}\), respectively. 
Consider the map $p : E_{p_2} \times E^{q_2} \to E_{p_2} + E^{q_2}$ defined by $(a',a'') \mapsto a' + a''$. Note that, since $p_2 \geq 2$, $q_2 \geq 2$, and $E_\bullet$ and $E^\bullet$ are opposite flags, this map is surjective with fiber isomorphic to $E_{p_2} \cap E^{q_2}$. The restriction $p_\circ :  E_\circ  \times  E^\circ \to E_{p_2} + E^{q_2}$ is therefore dominant with fibers over an element in the image isomorphic to a non-empty open subset of $E_{p_2} \cap E^{q_2}$. Fix a general \(z\) in $$\Gamma_3(X_u,X^v) = \{z \in X \ | \ \dim(V_z \cap (E_{p_2} + E^{q_2})) \geq 1\}$$ (see Lemma~\ref{lem:sing-rat3}); we have that $p_\circ^{-1}(V_z)= (E_\circ \times E^\circ) \cap p^{-1}(V_z)$ is a non-empty open subset of a vector space. 

Define 
\begin{multline*}
  \cP:=\{ ((a',a''),b,c) \in p_\circ^{-1}(V_z) \times (\C^{2n})^2 \ \mid b\in E_{p_1} \cap (a')^\perp\setminus0 ,\ c \in E^{q_1} \cap (a'')^\perp\setminus0  \}
\end{multline*}
and set $$\cP^\circ = \{ ((a',a''),b,c) \in \cP \ | \ \dim \langle a',a'',b,c \rangle = 4,\ \dim (V_z + \langle a',a'',b,c \rangle) = 5\}.$$ Note that $\cP^\circ$ is an open subset of $\cP$. Lemma \ref{lem:surj} implies that $\cP^\circ$ is non-empty. Then Proposition \ref{prop:pi3_unirat} follows from Lemma~\ref{lem:surj} and Lemma~\ref{lem:P_unirat} below.
\end{proof}

\begin{lemma}\label{lem:surj}
For \(z\in\Gamma_3(X_u,X^v)\) general, the map $\cP^\circ \to (\pi_3^3(u,v))^{-1}(z)\cap\mathcal{U}_3$ defined by $((a',a''),b,c) \mapsto (x,y)$ with $V_x = \langle a',b \rangle$ and $V_y = \langle a'',c \rangle$ is surjective.
\end{lemma}

\begin{proof}
Let $(x,y,z) \in (\pi_3^3(u,v))^{-1}(z)\cap\cU_3$. Then by assumptions \eqref{item:1} and \eqref{item:2}, $\dim(V_z \cap (V_x + V_y)) = 1$. Let $a \in V_z \cap (V_x + V_y)$ be a non-zero element and write $a = a' + a''$ with $a' \in V_x$ and $a'' \in V_y$. By assumption \eqref{item:1}, this decomposition is unique. 

We claim that \((a',a'')\in p_\circ^{-1}(V_z)\). Clearly, we have \(a'\in E_{p_2}\) and \(a''\in E^{q_2}\). By assumption \eqref{item:5}, we have $a' \not\in E_{p_1}$ and $a'' \not\in E^{q_1}$. 
Furthermore, if $E_{p_2} \not\subseteq E_{p_1}^\perp$, then since $x \in \oX_u$, we have {$V_x \cap E_{p_1}^\perp = V_x \cap E_{p_1}$} and $a' \not\in E_{p_1}^\perp$. By a symmetric argument, if $E^{q_2} \not\subseteq (E^{q_1})^\perp$, we have $a'' \not\in (E^{q_1})^\perp$. We thus have $(a',a'') \in  p_\circ^{-1}(V_z)$. 

Let $b \in V_x \cap E_{p_1}$ and $c \in V_y \cap E^{q_1}$ be non-zero elements. By assumption \eqref{item:5}, we have $V_x = \langle a',b\rangle$ and $V_y = \langle a'',c \rangle$, thus $b \in (a')^\perp$ and $c \in (a'')^\perp$. Furthermore, by assumptions \eqref{item:1} and \eqref{item:2}, we have $((a',a''),b,c) \in \cP^\circ$ mapping to \((x,y)\), proving the result.
\end{proof}

\begin{lemma}\label{lem:P_unirat}
The set $\cP$ is unirational.
\end{lemma}

\begin{proof}
Consider the map $\cP \to p_\circ^{-1}(V_z),\ ((a',a''),b,c) \mapsto (a',a'')$. We first prove that this map is surjective. Note that the choice of $c$ is independent of the choice of $b$, so by symmetry, we can concentrate on the map $((a',a''),b) \mapsto (a',a'')$. We only need to prove that $E_{p_1} \cap (a')^\perp$ is always non-zero. If $p_1 > 1$ the result follows. If $p_1 = 1$, then $E_{p_2} \subset E_{p_1}^\perp$ thus $a' \in E_{p_1}^\perp$ and $E_{p_1} \cap (a')^\perp = E_{p_1}$, proving the result. Now, by construction of $E_\circ$ and $E^\circ$, this map is the composition of vector bundles (with zero section removed),  proving the statement.
\end{proof}

Finally, we can prove Proposition \ref{prop:fibre-d-3}.

\begin{proof}[Proof of Proposition \ref{prop:fibre-d-3}]
The result about degree $d=3$ curves follows from Proposition~\ref{prop:pi3_unirat}, Lemma \ref{lemma:restriction-to-open} and Proposition~\ref{prop:replace}. The result about degree $(2,1)$ curves follows from Proposition~\ref{prop:pi3_unirat} and Corollary~\ref{cor_strategy}.
\end{proof}

We now deal with small degree (i.e. degree two and one) curves.

\section{Degree 2}
\label{sec:d=2}

Recall that there exists integers $p_1 < p_2$ with $p_1 + p_2 \neq 2n+1$ and $q_1 < q_2$ with $q_1 + q_2 \neq 2n+1$ such that 
$$\begin{array}{l}
X_u = \{x \in X \ | V_x \cap E_{p_1} \neq 0 \textrm{ and } V_x \subset E_{p_2} \} \\
X^v= \{x \in X \ | V_x \cap E^{q_1} \neq 0 \textrm{ and } V_x \subset E^{q_2} \}. \\
\end{array}$$
Recall also the definition of $\delta_p$ and $\delta_q$ and condition \eqref{eqn:deg2} 
$$ \left\{\begin{array}{l}
p_1+q_2 =2n=p_2+q_1, \\
p_2 - p_1 = q_2 - q_1 \geq 2 \textrm{ and } \max(\delta_p,\delta_q) = 1.\\
\end{array}\right.$$
In this section we prove the following result.

\begin{thm}
\label{thm:d=2}
Assume that $d = 2$, we have
\begin{enumerate}
\item\label{item:deg2} The map $\ev_3^d(u,v): M_d(X_u,X^v) \to \Gamma_d(X_u,X^v)$ is $\GRCF$,
\item\label{item:not_birat} If  $\ev_3^d(u,v): M_d(X_u,X^v) \to \Gamma_d(X_u,X^v)$ is not birational, then we have $\Gamma_d(X_u,X^v) = \Gamma_{d-1,1}(X_u,X^v)$ has rational singularities.
\item\label{item:deg1,1}If \eqref{eqn:deg2} does not hold, then $\ev_3^{d-1,1}(u,v): M_{d-1,1}(X_u,X^v) \to \Gamma_{d-1,1}(X_u,X^v)$ is $\GRCF$.
\item\label{item: 2to1} If \eqref{eqn:deg2} holds, then the map $\ev_3^d(u,v): M_d(X_u,X^v) \to \Gamma_d(X_u,X^v)$ is not birational, and $\ev_3^{d-1,1}(u,v): M_{d-1,1}(X_u,X^v) \to \Gamma_{d-1,1}(X_u,X^v)$ is generically finite of degree $2$.
\end{enumerate}
\end{thm}

\begin{proof}
  By Proposition~\ref{prop:replace}, to prove part \eqref{item:deg2}, it suffices to prove \[\pi_3^2(u,v): \Gamma_2^{(3)}(X_u,X^v)\to \Gamma_2(X_u,X^v)\] is GRCF. We are going to construct a birational morphism \(f:\cW\to \Gamma_2^{(3)}(X_u,X^v)\) from an irreducible variety \(\cW\), and prove that the general fiber of \(\pi = \pi_3^2(u,v)\circ f \) is unirational. More precisely, part \eqref{item:deg2} follows from Proposition~\ref{prop:replace}, Lemma~\ref{lemma:restriction-to-open}, Lemma~\ref{lemma:W}, Lemma~\ref{lemma:f-birat}, and Proposition~\ref{prop:deg2}.

  Part~\eqref{item:not_birat} follows from Lemmas~\ref{lemma:deg2-birat}, \ref{lemma:gamma2}, and Corollary~\ref{cor:gamma1,1}.

  By Proposition~\ref{prop:replace}, to prove \eqref{item:deg1,1}, it suffices to prove \(\pi_3^{1,1}(u,v): \Gamma_{1,1}^{(3)}(X_u,X^v)\to \Gamma_{1,1}(X_u,X^v)\) is GRCF, which is Proposition~\ref{prop:1,1}.

  When \eqref{eqn:deg2} holds, by Corollary~\ref{cor:gamma1,1}, we have \(\Gamma_{1,1(X_u,X^v)}=X\). Moreover, \(\dim M_{1,1}(X_u,X^v)=\dim X\). Part~\eqref{item: 2to1} follows from Lemma~\ref{lemma:deg2-birat}, Corollary~\ref{cor:d-1-1}, Proposition~\ref{prop:1,1}, and Lemma~\ref{lemma:fibre-meets-open}.
\end{proof}

Define the following open subsets of projective spaces:
\begin{itemize}
\item If $\delta_p = 0$, set $\bP(E_{p_1})_\circ = \bP(E_{p_1})$ and $\bP(E_{p_2})_\circ = \bP(E_{p_2})$.
\item If $\delta_p = 1$, set $\bP(E_{p_1})_\circ = \bP(E_{p_1}) \setminus \bP(E_{p_2}^\perp)$ and $\bP(E_{p_2})_\circ = \bP(E_{p_2})\setminus \bP(E_{p_1}^\perp)$.
\item If $\delta_q = 0$, set $\bP(E^{q_1})_\circ = \bP(E^{q_1})$ and $\bP(E^{q_2})_\circ = \bP(E^{q_2})$.
\item If $\delta_q = 1$, set $\bP(E^{q_1})_\circ = \bP(E^{q_1}) \setminus \bP({E^{p_2}}^\perp)$ and $\bP(E^{q_2})_\circ = \bP(E^{q_2})\setminus \bP({E^{q_1}}^\perp)$.
\end{itemize}
For simplicity, set $$Y = \bP(E_{p_1})_\circ \times \bP(E^{q_2})_\circ \times \bP(E^{q_1})_\circ \times \bP(E_{p_2})_\circ.$$ For any element $([a_1],[a_2],[b_1],[b_2],z) \in Y \times X$, consider the following conditions:
\begin{enumerate}
\item\label{item:W1} $\omega(a_1,b_2) = 0 = \omega(a_2,b_1)$.
\item\label{item:W2} $V_z \cap \langle a_1, a_2 \rangle \neq 0 \neq V_z \cap \langle b_1, b_2 \rangle$.
\item\label{item:W3} $V_z \cap  \langle a_1, b_2 \rangle = 0 = V_z \cap \langle b_1, a_2 \rangle$.
\item\label{item:W4} $\omega(a_1,b_1)\neq 0 \neq \omega(a_2,b_2)$.
\item\label{item:W5} $\dim(\langle a_1, a_2, b_1, b_2 \rangle) = {\rm Rk}(\omega\vert_{\langle a_1, a_2, b_1, b_2 \rangle}) = 4$.
\end{enumerate}
\begin{remark}\label{rmk:dim1}
  Conditions \eqref{item:W2} and \eqref{item:W3} imply that $$\dim(V_z\cap \langle a_1,a_2 \rangle)=1=\dim(V_z\cap\langle b_1,b_2 \rangle).$$ 
\end{remark}

Set $\cW = \{([a_1],[a_2],[b_1],[b_2],z) \in Y \times X \ | \ \textrm{\eqref{item:W1}-\eqref{item:W5} are satisfied} \}.$ Note that $\cW$ is locally closed in $Y\times X$.

Define $p : E_{p_2} \times E^{q_1} \to E_{p_2} + E^{q_1}$ and $q : E_{p_1} \times E^{q_2} \to E_{p_1} + E^{q_2}$. 
\begin{remark}
  Denote the second projection from $\cW$ by $\pi':\cW\to X$. By Remark~\ref{rmk:dim1}, the fiber \(\pi'^{-1}(z)\) can be identified with \(\{[a_1,a_2],[b_1,b_2]\in\bP(q^{-1}(V_z))\times \bP(p^{-1}(V_z))\ |\ ([a_1],[a_2],[b_1],[b_2])\in Y\text{ and satisfies }\eqref{item:W1},\eqref{item:W4}, \eqref{item:W5}\}\).
\end{remark}

\begin{lemma}\label{lemma:W}
  The subset $\cW \subset Y\times X$ is irreducible and $\dim \cW = \dim M_2(X_u,X^v)$.
  \end{lemma}
  
  \begin{proof}
  Note that conditions \eqref{item:W3}-\eqref{item:W5} are open. Define $$\cS = \{([a_1],[a_2],[b_1],[b_2]) \in Y\ | \ \eqref{item:W1}, \eqref{item:W4} \textrm{ and } \eqref{item:W5} \textrm{ are satisfied} \}$$ and let $p\colon\cS \to  \bP(E_{p_1})_\circ \times \bP(E^{q_2})_\circ$ be the projection on the first two factors. The map $p$ is surjective with fibers of constant dimension $p_2 + q_1 - 2 - \delta_p - \delta_q$. Since these fibers are given by linear conditions, we get that $\cS$ is irreducible of dimension $p_1 + p_2 + q_1 + q_2 -4 - \delta_p - \delta_q$. We now consider the map $p'\colon\cW \to \cS$ obtained by projection on the first four factors. Fix $([a_1],[a_2],[b_1],[b_2]) \in \cS$ we look for $z$ in the fiber of $p'$. By conditions \eqref{item:W2} and \eqref{item:W3} the space $V_z$ meets $\langle a_1,a_2 \rangle$ and  $\langle b_1,b_2 \rangle$ in dimension exactly $1$. Given $[a] \in \bP(\langle a_1,a_2 \rangle) \setminus \{[a_1],[a_2]\}$, condition \eqref{item:W4} implies that there is a unique $z$ in the fiber such that $a \in V_z$. This implies that $p'$ is a locally trivial $\C^\times$-bundle over $\cS$ and thus is irreducible of dimension $p_1 + p_2 + q_1 + q_2 -3 - \delta_p - \delta_q = \dim M_2(X_u,X^v)$.
  \end{proof}
  
  \begin{lemma}\label{lemma:f-birat}
  The map $f : \cW \to \Gamma_2^{(3)}(X_u,X^v)$ given by $f([a_1],[a_2],[b_1],[b_2],z) = (x,y,z)$ with $V_x = \langle a_1,b_2 \rangle$ and $V_y = \langle a_2,b_1 \rangle$ is well-defined and birational.
  \end{lemma}
  
  \begin{proof}
  Since $\dim(\langle a_1, a_2, b_1, b_2 \rangle) = 4$, we have that $V_x$ and $V_y$ have dimension $2$ and are isotropic by the conditions $\omega(a_1,b_2) = \omega(a_2,b_1) = 0$. This implies that $x \in X_u$ and $y \in X^v$ and since $V_z \subset V_x + V_y$ we have that $(x,y,z) \in \Gamma_2^{(3)}(X_u,X^v)$ by Lemma \ref{lem:deg2}. 
  
  Conversely, for $(x,y,z)$ general in $\Gamma_2^{(3)}(X_u,X^v)$, we may assume that $x \in \oX_u$ and $y \in \oX^v$. We thus have $\dim (V_x \cap E_{p_1}) = 1$ so that $[a_1]$ is uniquely given by $\langle a_1 \rangle = V_x \cap E_{p_1}$. We also have  that $[b_1]$ is given by $\langle b_1 \rangle = V_y \cap E^{q_1}$. Furthermore, we may also assume that $V_x \cap V_z = V_y \cap V_z = 0$ thus $(\langle a_1 \rangle +V_z) \cap V_y$ has dimension $1$ and $[a_2]$ is fixed by the condition $\langle a_2 \rangle = (\langle a_1 \rangle +V_z) \cap V_y$. By the same argument $[b_2]$ is fixed by the condition $\langle b_2 \rangle = (\langle b_1 \rangle +V_z) \cap V_x$.
  \end{proof}
  
  Consider the restriction $\pi : \cW \to \Gamma_2(X_u,X^v)$ of the projection $\pi'$.

  \begin{lemma}\label{lemma:deg2-birat}
The map $\pi$ is birational if and only if one of the following holds 
    \begin{enumerate}
    \item $p_1 + q_2 < 2n$ and $p_2 + q_1 < 2n$,
    \item $p_1 + q_2 = 2n$, $p_2 + q_1 < 2n$ and $\max(\delta_p,\delta_q) =1$,
    \item $p_1 + q_2 < 2n$, $p_2 + q_1 = 2n$ and $\max(\delta_p,\delta_q) =1$.
    \end{enumerate}
    \end{lemma}
    \begin{proof}
{{We only prove the if part, the converse will follow from the proof of Proposition \ref{prop:deg2}. Note that if ($p_1 + q_2 = 2n$, $p_2 + q_1 < 2n$) or if ($p_1 + q_2 < 2n$, $p_2 + q_1 = 2n$), we also prove the only if condition.}}
      Assume first that $p_1 + q_2 < 2n$ and $p_2 + q_1 < 2n$. Then, for $z\in \Gamma_2(X_u,X^v)$ general, by Lemma~\ref{lem:deg2}, we have $\dim(V_z \cap (E_{p_1} + E^{q_2})) = 1 = \dim(V_z \cap (E_{p_2} + E^{q_1}))$. This defines unique elements $[a_1+ a_2] \in \bP(V_z \cap (E_{p_1} + E^{q_2})) $ and $[b_2+ b_1] \in \bP(V_z \cap (E_{p_2} + E^{q_1})) $ with $([a_1],[a_2],[b_1],[b_2],z) \in \cW$, proving part (1). 

      Assume that $p_1 + q_2 < 2n$ and $p_2 + q_1 = 2n$ (the case $p_1 + q_2 = 2n$ and $p_2 + q_1 < 2n$ is similar). Then, for $z$ general, we have $\dim(V_z \cap (E_{p_1} + E^{q_2})) = 1$ and this defines a unique element $[a_1+ a_2] \in \bP(V_z \cap (E_{p_1} + E^{q_2})) $. For any $[b] \in \bP(V_z)$, we have a unique decomposition $b = b_1 + b_2$ with $[b_2] \in \bP(E_{p_2})$ and $[b_1] \in \bP(E^{q_1}) $. If $\max(\delta_p,\delta_q) = 0$ then the conditions $\omega(a_1,b_2) = \omega(a_2,b_1) = 0$ are satisfied, and we get a fiber of dimension $1$ parametrized by $\bP(V_z)$. Otherwise, at least one of these (linear) conditions is non-trivial and there is a unique $[b] \in \bP(V_z)$ satisfying these conditions, proving part (2) and (3).
    \end{proof}

    \begin{lemma}\label{lemma:gamma2}
      When $\pi$ is not birational, 
      $$\Gamma_2(X_u,X^v) \subseteq \{z \in X \ | \ V_z \cap (E_{p_1} + (E_{p_2} \cap E ^{q_2}) + E^{q_1}) \neq 0 \}.$$ 
    \end{lemma}
    \begin{proof}
      In this case, by Lemma~\ref{lemma:deg2-birat}, either $p_1 + q_2 \geq 2n$ or $p_2+q_1 \geq 2n$. Without loss of generality, we assume $p_1 + q_2 \geq 2n$. Note that if $z$ is the image of a generic $(x,y,z)$ in $\Gamma_2(X_u,X^v)$, then $V_x\subset E_{p_2}$, $V_y\cap E^{q_1}\neq 0$ and $V_z\subset V_x+V_y$, so $0\neq V_z \cap (V_x+(V_y\cap E^{q_1}))\subset V_z\cap(E_{p_2}+E^{q_1})$.
    \end{proof}

    The reverse inclusion will follow from the proof of Proposition~\ref{prop:deg2}.
    \begin{prop}\label{prop:deg2}
      The general fiber of $\pi$ is unirational.
    \end{prop}

    \begin{proof}\label{prop:deg2-pi}
    {{In the cases of Lemma \ref{lemma:deg2-birat}, the map is birational. The remaining cases are as follows. Note that the discussion proves the only if part of Lemma \ref{lemma:deg2-birat}.}}  
      \item[Case 1: $p_1 + q_2 \geq 2n$ and $p_2 + q_1 > 2n$.] Consider the map $g_a\colon \cW \to \{ ([a_1],[a_2],z) \in \bP(E_{p_1})_\circ \times \bP(E^{q_2})_\circ \times X \ | \ a_1\not\in V_z \cap \langle a_1, a_2 \rangle \neq 0\text{ and }\omega(a_1,a_2) \neq 0\}$ obtained by projection, and let $\cW_\circ$ be the open subset in $\cW$ where this map has fibers of minimal dimension. It suffices to prove that  $F = \cW_\circ \cap \pi^{-1}(z)$ is unirational of positive dimension for $z \in \Gamma_2(X_u,X^v)$ general. Consider the restriction $g_a\colon F \to \bP(q^{-1}(V_z))$. A general fiber of this map is an open subset of the projective space $\{[b_1,b_2] \in \bP(p^{-1}(V_z)) \ | \ \omega(a_1,b_2) = 0 = \omega(a_2,b_1) \}$. This projective space has dimension $1 + p_2 + q_1 -2n - \delta_p -\delta_q \geq 0$, thus the map is dominant. Furthermore, we know that fibers of this map have constant dimension. It follows that $F$ is unirational of dimension $p_1 + p_2 + q_1 + q_2 - 4n + 2 - \delta_p - \delta_q \geq 1$. 
      \item[Case 2: $p_1 + q_2 > 2n$ and $p_2 + q_1 \geq 2n$.] This is similar to Case 1. 
      \item[Case 3: $p_1 + q_2 = 2n = p_2 + q_1$.] In this case, $\min(\delta_p,\delta_q) = 0$, and for fixed \(z\in \Gamma_2(X_u,X^v)\), we have $\dim p^{-1}(V_z) = \dim q^{-1}(V_z) = 2$. Therefore, $\pi^{-1}(z)$ is birational to $\bP^1 \times \bP^1$ if $\max(\delta_p,\delta_q) = 0$ or an open subset of a hyperplane section of bidegree $(1,1)$ in $\bP^1 \times \bP^1$ if $\max(\delta_p,\delta_q) = 1$. We are done in the former case, in the latter case, we only need to prove that the defining form is non-degenerate. We may assume $p_1 \leq q_1$ which implies $p_1 + 2 \leq q_1$, and we exhibit a $V_z$ for which the form is non-degenerate. Set $V_z = \langle e_{q_1} + e_{2n+1 - q_1} , e_{q_2} + e_{2n+1 - q_2} \rangle$ if $p_2 \geq q_1$ and $V_z = \langle e_{p_2} + e_{2n+1 - p_2} , e_{q_2} + e_{2n+1 - q_2} \rangle$ if $p_2 \leq q_1$. Both cases are similar, so we only deal with the latter. We have $q^{-1}(V_z) = \langle (0,e_{p_2} + e_{2n+1 - p_2}) , (0,e_{q_2} + e_{2n+1 - q_2}) \rangle$ and $p^{-1}(V_z) = \langle  (e_{p_2},e_{2n+1 - p_2}) , (e_{2n+1 - q_2},e_{q_2})\rangle$. The condition $\omega(a_1,b_2) = 0$ is always trivial and the condition $\omega(a_2,b_1) = 0$ is induced by a matrix of the form
      $$\begin{pmatrix}\pm 1& 0 \\ 0 & \pm 1\\
      \end{pmatrix}$$
      proving the result in this case. 
      \item[Case 4: $p_1 + q_2 < 2n$ and $p_2 + q_1 > 2n$.]  Note that \(\dim\cW\geq p_1+q_2+2n-3\). By Lemma~\ref{lemma:gamma2}, \(\dim\Gamma_2(X_u,X^v)\leq p_1+q_2+2n-4\). Therefore, fibers of \(\pi\) are positive dimensional.
      For $z\in \Gamma_2(X_u,X^v)$ general, by Lemma~\ref{lem:deg2}, we have $\dim(V_z \cap (E_{p_1} + E^{q_2})) = 1$ and therefore a unique element $[a_1+ a_2] \in \bP(V_z \cap (E_{p_1} + E^{q_2})) $. 
      
      For $([a_1],[a_2],[b_1],[b_2],z) \in \pi^{-1}(z)$ the elements $[a_1]$ and $[a_2]$ are fixed as above.
      Therefore, \(\pi^{-1}(z)\) is birational to the linear space
      \[
        \{[b_1,b_2] \in {\bP(p^{-1}(V_z))} \ |\ \omega(a_1,b_2) = \omega(a_2,b_1) = 0 \}.
      \]
      \item[Case 5: $p_1 + q_2 > 2n$ and $p_2 + q_1 < 2n$.] This is similar to Case 4. 
      \item[Case 6: $p_1 + q_2 = 2n$, $p_2 + q_1 < 2n$ and $\max(\delta_p,\delta_q) =0$.] This follows from the proof of Lemma \ref{lemma:deg2-birat} part (2).
      \item[Case 7: $p_1 + q_2 < 2n$, $p_2 + q_1 = 2n$ and $\max(\delta_p,\delta_q) =0$.] This is similar to Case 6.
\end{proof}

\begin{cor} \label{cor:gamma_2=}
  When $\pi$ is not birational, {{we have}}
      $$\Gamma_2(X_u,X^v) = \{z \in X \ | \ V_z \cap (E_{p_1} + (E_{p_2} \cap E ^{q_2}) + E^{q_1}) \neq 0 \}$$
      {{and $\Gamma_2(X_u,X^v)$ has rational singularities.}} 
\end{cor}
\begin{proof}
  When $p_1 + q_2 \geq 2n$, we have $E_{p_1} + (E_{p_2} \cap E ^{q_2}) + E^{q_1} = E_{p_2} + E^{q_1}$; when $p_2+q_1 \geq 2n$, we have $E_{p_1} + (E_{p_2} \cap E ^{q_2}) + E^{q_1} = E_{p_1} + E^{q_2}$. The result follows from Lemma~\ref{lemma:gamma2} and the proof of Proposition~\ref{prop:deg2}. The rational singularities statement follows from Proposition \ref{prop_rat_sing_smalldegreecurves}.
\end{proof}

\begin{lemma}\label{lemma:gamma1,1}
  We have 
  $$\Gamma_{1,1}(X_u,X^v) \subseteq \{z \in X \ | \ V_z \cap (E_{p_1} + (E_{p_2} \cap E ^{q_2}) + E^{q_1}) \neq 0 \}.$$
\end{lemma}
\begin{proof}
Note that there is no assumption on $\pi$ here.  Let \((x,y,z,t)\in \Gamma^{(4)}_{1,1}(X_u,X^v)\) be general, then \(V_t\subset V_x+ V_y\), \(V_t\cap V_z\neq 0\) and \(V_x+V_y\subseteq E_{p_1} + (E_{p_2} \cap E ^{q_2}) + E^{q_1}\), so \(V_z \cap (E_{p_1} + (E_{p_2} \cap E ^{q_2}) + E^{q_1}) \neq 0\), proving the result.
\end{proof}

The reverse inclusion will follow from the proof of the following Proposition.

\begin{prop}\label{prop:1,1}
  \begin{enumerate}
  \item {{$\Gamma_{1,1}(X_u,X^v)$ is non-empty if and only if $p_2+q_2 > 2n$.}}
    \item If \eqref{eqn:deg2} does not hold, then \(\pi_3^{1,1}(u,v)\) is GRCF.
    \item If \eqref{eqn:deg2} holds, then \(\pi_3^{1,1}(u,v)\) is generically finite of degree $2$.
  \end{enumerate}
  
\end{prop}
\begin{proof}
  Note that if \(p_2+q_2\leq 2n\), then \(\Gamma_{1,1}(X_u,X^v)=\emptyset\). Now assume \(p_2+q_2>2n\).

  Let \(z\in \Gamma_{1,1}(X_u,X^v)\) be general. We look for \(a\in V_z \cap (E_{p_1} + (E_{p_2} \cap E ^{q_2}) + E^{q_1})\setminus0\) with decompositions \(a=a_{[1]}+a_{[2]}+a'_{[1]}\), where $ a_{[1]} \in E_{p_1}\setminus0$, { $a_{[2]} \in (E_{p_2} \cap E^{q_2})\setminus (E_{p_1} \cup E^{q_1})$,} $a'_{[1]} \in E^{q_1}\setminus0$, and 
  \begin{equation}\label{eq1}
    \omega(a_{[2]},a_{[1]}) = 0, 
  \end{equation} 
  \begin{equation}\label{eq2}
    \omega(a_{[2]},a'_{[1]}) = 0. 
  \end{equation} 
  Then for $V_x = \langle a_{[1]},a_{[2]} \rangle$ and $V_y = \langle a'_{[1]},a_{[2]} \rangle$, we have that \((x,y,z)\in (\pi_3^{1,1}(u,v))^{-1}(z)\). Without loss of generality, we assume $$p_2 + q_1 \geq p_1 + q_2.$$ 
  Let $$m = \max(0,p_1+q_2 - 2n).$$ 
  Write $a = (a_1,\cdots,a_{2n})$ in the basis $(e_i)_{i \in [1,2n]}$. Then 
  $$\begin{array}{l}
    a_{[1]} = (a_1,\cdots,a_{p_1-m},a_{p_1-m+1}+\lambda_{p_1 - m + 1},\cdots,a_{p_1}+\lambda_{p_1},0,\cdots,0),\\
    a'_{[1]} = (0,\cdots,0,a_{2n+1-q_1} + \mu_{2n+1-q_1},\cdots,a_{p_2} + \mu_{p_2},a_{p_2+1},\cdots,a_{2n}),
    \end{array}$$
    and \(a_{[2]} = a - a_{[1]} -a'_{[1]}\),
    where the variables $\lambda_i$ and $\mu_j$ satisfy equations \eqref{eq1} and \eqref{eq2}. 
    
    \item[Case 1: \(p_2+q_1>2n\).] This condition guarantees that there is at least one variable \(\mu_{p_2}\).
    We first compute the coefficients $c^1_{p_2}$ and $c^2_{p_2}$ of $\mu_{p_2}$ in equations \eqref{eq1} and \eqref{eq2}, respectively. We get
    $$c^1_{p_2} =  \left\{ \begin{array}{ll}
    0 & \textrm{ if } p_2 < 2n+1 - p_1 \\
    -(a_{2n+1-p_2}+\lambda_{2n+1-p_2}) & \textrm{ if } 2n+1 - p_1 < p_2 \leq q_2 \\
    -a_{2n+1-p_2} & \textrm{ if } 2n+1 - p_1< p_2 > q_2 
    \end{array} \right. $$ and
    $$c^2_{p_2} =  \left\{ \begin{array}{ll}
    {\pm} a_{2n+1-p_2} & \textrm{ if } p_2 < 2n+1 - p_1 \\
    -\lambda_{2n+1-p_2} & \textrm{ if } 2n+1 - p_1 < p_2 \leq q_2 \\
    0 & \textrm{ if } 2n+1 - p_1 < p_2 > q_2. 
    \end{array} \right. $$
    If $p_1 + p_2 < 2n+1$, then equation \eqref{eq1} is trivial, and we can solve equation \eqref{eq2} in $\mu_{p_2}$ as long as $a_{2n+1 - p_2} \neq 0$, proving the result in this case. Assume $p_1 + p_2 > 2n+1$. If $q_1 + q_2 < 2n+1$, then equation \eqref{eq2} is trivial, and we can solve equation \eqref{eq1} in $\mu_{p_2}$ as long as $a_{2n+1 - p_2} \neq 0$ or $a_{2n+1 - p_2}+ \lambda_{2n+1 - p_2} \neq 0$, proving the result in this case. We may therefore also assume $q_1 + q_2 > 2n+1$. Since $p_2 + q_1 \geq p_1 + q_2$, we get $2(p_2 + q_1) \geq p_1 + p_2 + q_1 + q_2 \geq 4n+4$ thus $p_2 + q_1 \geq 2n+2$. In particular $p_2 - 1 \geq 2n+1 - q_1$. We compute  the coefficients $c^1_{p_2-1}$ and $c^2_{p_2-1}$ of $\mu_{p_2-1}$ in equations \eqref{eq1} and \eqref{eq2}, respectively:
    $$c^1_{p_2-1} =  \left\{ \begin{array}{ll}
      -(a_{2n+2-p_2}+\lambda_{2n+2-p_2}) & \textrm{ if } p_2 -1 \leq q_2 \\
      -a_{2n+2-p_2} & \textrm{ if } q_2 < p_2 -1\\
      \end{array} \right. $$ and
      $$c^2_{p_2-1} =  \left\{ \begin{array}{ll}
      -\lambda_{2n+2-p_2} & \textrm{ if } p_2 - 1 \leq q_2 \\
      0 & \textrm{ if } q_2 < p_2 -1. \\
      \end{array} \right.$$
      If $p_2 \leq q_2$, the minor associated to the variables $\mu_{p_2-1}$ and $\mu_{p_2}$ in the equations is equal to  $(a_{2n+2-p_2}+\lambda_{2n+2-p_2})\lambda_{2n+1-p_2} - (a_{2n+1-p_2}+\lambda_{2n+1-p_2})\lambda_{2n+2-p_2}$ so that we can solve the equations in this case and the result follows. If $p_2 = q_2 +1$, then the minor is equal to $ a_{2n+1-p_2}\lambda_{2n+2-p_2}$ and the result follows. Otherwise, $q_2 < p_2 -1$. 
      Notice that in this case $c^1_{p_2} =-a_{2n+1-p_2}$ while $c^2_{p_2}=0$. So it suffices to find a coefficient $i$ such that $c^2_i$ is different from zero; indeed, if we do so, then the determinant of the matrix whose rows are $(c^1_{p_2} , c^1_i)$ and $(c^2_{p_2}, c^2_i)$
      is different from zero, and we can solve the equations with respect to $\mu_{p_2}$ and $\mu_i$. Note that for $i=q_2$ we get
      $$c^2_{q_2} =  \left\{ \begin{array}{ll}
      -\lambda_{2n+1-q_2} & \textrm{ if } p_1+q_2 \geq 2n+1 \\
      a_{2n+1-q_2} & \textrm{ if } p_1+q_2 \leq 2n, \\
      \end{array} \right. $$ 
      proving the result in this case.

    \item[Case 2: \(p_2+q_1\leq 2n\).] In this case, we have \(E_{p_1} \oplus (E_{p_2} \cap E^{q_2}) \oplus E^{q_1}\subseteq \C^{2n}\). For any \(a\in V_z \cap (E_{p_1} + (E_{p_2} \cap E ^{q_2}) + E^{q_1})\setminus0\), we can write $a = a_{[1]} + a_{[2]} + a'_{[1]}$ in a unique way. 
    \item[Sub-case 2.1: \(p_1+q_2<2n\).] Then $E_{p_1} + (E_{p_2} \cap E^{q_2}) + E^{q_1}\subsetneq\C^{2n}$. For $z \in \Gamma_{1,1}(X_u,X^v)$ general, we thus have $V_z \cap (E_{p_1} + (E_{p_2} \cap E^{q_2}) + E^{q_1}) = \langle a \rangle$ with $a \neq 0$. 
    For $z$ general, we have $a_{[1]} \in E_{p_1} \setminus 0$, {$a_{[2]} \in E_{p_2} \cap E^{q_2} \setminus (E_{p_1} \cup E^{q_1})$}, and $a'_{[1]} \in E^{q_1} \setminus 0$. This implies the equalities $V_x = \langle a_{[1]},a_{[2]} \rangle$ and $V_y = \langle a'_{[1]},a_{[2]} \rangle$, 
     thus the map is birational.
    \item[Sub-case 2.2: $p_1 + q_2 = 2n = p_2 + q_1$ and \eqref{eqn:deg2} does not hold.]    
    We have $\C^{2n} = E_{p_1} \oplus (E_{p_2} \cap E^{q_2}) \oplus E^{q_1}$. Either $\max(\delta_p,\delta_q) = 0$, and the fiber over \(z\in \Gamma_{1,1}(X_u,X^v)\) general is isomorphic to $\bP(V_z)$. Otherwise, $p_2 - p_1 = q_2 - q_1 = 1$.
    Without loss of generality, we may assume $p_1 \leq q_1$. Then $p_1 \leq n-2$, $p_2 \leq n- 1$, $q_1 \geq n+1$, and $q_2 \leq n+2$. In particular $\delta_p = 0$ and $\delta_q = 1$. 
    We want to find $a \in V_z$ such that $\omega(a_{[2]},a'_{[1]}) = 0$. We have $E_{p_2} \cap E^{q_2} = \langle e_{p_2} \rangle$ thus $a'_{[1]} \in e_{p_2}^\perp$. Set $E := E_{p_1} \oplus (E_{p_2} \cap E^{q_2}) \oplus (E^{q_1} \cap e_{p_2}^\perp)$ which is of codimension $1$ in $\C^{2n}$ and for $z \in X$ general, we have $V_z \cap E = \langle a \rangle$ with $a \neq 0$. This implies that the map is birational.
    
    \item[Sub-case 2.3: \eqref{eqn:deg2} holds.] Again, we have $\C^{2n} = E_{p_1 } \oplus (E_{p_2} \cap E^{q_2} )\oplus E^{q_1}$. Without loss of generality, we may assume that $p_1 \leq q_1$. Our assumptions imply $q_1 \geq p_1 + 2$, $q_2 \geq p_2 + 2$, $\delta_p = 0$ and $\delta_q = 1$. Therefore, \eqref{eq1} is trivial. We claim that for \(z\in X\) general, equation \eqref{eq2} is a non-degenerate bilinear condition on \(\P(V_z)\). {Consider $V_z = \langle b, c \rangle$, where \(b=b_1+b_2+b_1'\) and \(c=c_1+c_2+c_1'\) for some \(b_1, c_1\in E_{p_1 }\), \(b_2, c_2\in E_{p_2}\cap E^{q_2}\), and \(b_1', c_1'\in E^{q_1}\). Then for \(a=\lambda b + \mu c \in V_z\) general, \(a_{[2]}=\lambda b_2+\mu c_2\), \(a'_{[1]}=\lambda b_1'+\mu c_1'\), and equation \eqref{eq2} is of the form \(\lambda^2\omega(b_2,b_1')+\mu^2\omega(c_2,c_1')\), proving the claim.}
\end{proof}

\begin{cor}\label{cor:gamma1,1}
{{Assume that $\Gamma_{1,1}(X_u,X^v)$ is non-empty, then we have}}
$$\Gamma_{1,1}(X_u,X^v) = \{z \in X \ | \ V_z \cap (E_{p_1} + (E_{p_2} \cap E ^{q_2}) + E^{q_1}) \neq 0 \}.$$ In particular, it has rational singularities.
\end{cor}
\begin{proof}
  This follows from Lemma~\ref{lemma:gamma1,1}, the proof of Proposition~\ref{prop:1,1}, and Proposition~\ref{prop_rat_sing_smalldegreecurves}.
\end{proof}

\section{Degree 1}
\label{sec:d=1}

Recall Condition \eqref{eqn:deg1}: $ p_1 + q_1 = 2n = p_2 = q_2$.
In this section we prove the following result.

\begin{thm}
\label{thm:d=1}
Assume that $d = 1$, we have
\begin{enumerate}
\item\label{item:deg1} The map $\ev_3^d(u,v): M_d(X_u,X^v) \to \Gamma_d(X_u,X^v)$ is $\GRCF$,
\item\label{item:deg1_rat_sing} If the map $\ev_3^d(u,v): M_d(X_u,X^v) \to \Gamma_d(X_u,X^v)$ is not birational, then $\Gamma_d(X_u,X^v) = \Gamma_{d-1,1}(X_u,X^v)$ has rational singularities.
\item\label{item:01birat} If \eqref{eqn:deg1} does not hold, then $\ev_3^{d-1,1}(u,v): M_{d-1,1}(X_u,X^v) \to \Gamma_{d-1,1}(X_u,X^v)$ is $\GRCF$, and it is birational as soon as $\ev_3^d(u,v)$ is.
\item\label{item:01_2to1} If \eqref{eqn:deg1} holds, then the map $\ev_3^d(u,v): M_d(X_u,X^v) \to \Gamma_d(X_u,X^v)$ is not birational, and $\ev_3^{d-1,1}(u,v): M_{d-1,1}(X_u,X^v) \to \Gamma_{d-1,1}(X_u,X^v)$ is generically finite of degree $2$.
\end{enumerate}
\end{thm}

In fact, we will explicitly describe when $\ev_3^1(u,v)$ is birational. Consider the following condition:
\begin{equation}
    \label{eqn:deg1_birat}
    \tag{L1}
    \left\{\begin{array}{ll}
 p_1+q_1\leq 2n-1+\min(\delta_p,\delta_q) &\mbox{ if } p_2\neq 2n \mbox{ or }q_2\neq 2n \\
p_1+q_1\leq 2n-1 &\mbox{ if } p_2=q_2= 2n .\\
\end{array}\right.
\end{equation}
{{We will see that condition \eqref{eqn:deg1_birat} is the one that ensures that $\ev_3^1(u,v)$ is birational.}}

\begin{proof}[Proof of Theorem \ref{thm:d=1}]
By the proof of Proposition \ref{prop:md-to-gammad3} (respectively Corollary \ref{cor:d-1-1}) when $d=1$, the map $\ev_{(3)}^1(u,v)$ (resp. $\ev_{(3)}^{0,1}(u,v)$) is birational. 

If \eqref{eqn:deg1_birat} holds then by Lemmas \ref{lem_pi_3^1}, \ref{lem_pi_3^1_bis}, and Lemma \ref{lemma:fibre-meets-open}, $\pi_3^1(u,v)$ is birational. If \eqref{eqn:deg1_birat} does not hold then by Lemmas \ref{lem_pi_3^1}, \ref{lem_pi_3^1_tris} and Lemma \ref{lemma:fibre-meets-open}, $\pi_3^1(u,v)$ is $\GRCF$ with positive dimensional fibers. Thus, by Proposition \ref{prop:replace} the map $\ev_3^1(u,v)$ is $\GRCF$; moreover, it is birational if \eqref{eqn:deg1_birat} holds, and it has positive dimensional fibers otherwise. This shows part \eqref{item:deg1}. We deduce part \eqref{item:deg1_rat_sing} from Lemmas \ref{lem_eq_gammas_lines} and \ref{lem_twotoone_lines}.

If \eqref{eqn:deg1_birat} holds then by Lemmas \ref{lem_pi_3^0,1}, \ref{lem_pi_3^0,1_bis}, and Lemma \ref{lemma:fibre-meets-open}, $\pi_3^{0,1}(u,v)$ is birational. If neither \eqref{eqn:deg1_birat} nor \eqref{eqn:deg1} hold then by Lemmas \ref{lem_pi_3^0,1}, \ref{lem_pi_3^0,1_tris} and Lemma \ref{lemma:fibre-meets-open}, $\pi_3^{0,1}(u,v)$ is $\GRCF$ with positive dimensional fibers. Thus, if \eqref{eqn:deg1} does not hold, by Proposition \ref{prop:replace} the map $\ev_3^{0,1}(u,v)$ is $\GRCF$; moreover it is birational if \eqref{eqn:deg1_birat} holds, and it has positive dimensional fibers otherwise. This shows part \eqref{item:01birat}. Finally, item \eqref{item:01_2to1} is a consequence of Lemma \ref{lem_twotoone_lines}.
\end{proof}

We start with a preliminary result about line neighborhoods. Throughout this section on degree $1$ curves, we will use the following addition map
$$ p:E_{p_1}\times E^{q_1} \to E_{p_1}+E^{q_1}\subset \CC^{2n} .$$

\begin{lemma} 
\label{lem_eq_gammas_lines}
If neither \eqref{eqn:deg1_birat} nor \eqref{eqn:deg1} hold then 
$$\Gamma_{1}(X_u,X^v)=\Gamma_{0,1}(X_u,X^v)=\{z\in X \mid \dim(V_z\cap E_{p_2}\cap E^{q_2})\geq 1\}$$
has rational singularities.
\end{lemma}

\begin{proof}
The statement about rational singularities is a consequence of Proposition \ref{prop_rat_sing_smalldegreecurves}. {{Recall also the inclusion $\Gamma_{0,1}(X_u,X^v) \subseteq \Gamma_1(X_u,X^v)$.}} Denote the RHS of the above formula by $A$. Consider $x\in X_u$, $y\in X^v$ such that there exists a line passing through $x,y$. This means that $\dim(V_x\cap V_y)\geq 1$. A point $z$ belongs to this line if and only if $V_z\cap (V_x\cap V_y)\neq 0 $ (and $V_z\subset V_x+V_y$ whenever $x\neq y$). This proves the inclusion $\Gamma_{1}(X_u,X^v)\subseteq A$. 

Conversely, we prove that $A$ is contained in $ \Gamma_{0,1}(X_u,X^v)$. Since $\Gamma_{0,1}(X_u,X^v)$ is closed, it is sufficient to show that a general point $z\in A$ belongs to $\Gamma_{0,1}(X_u,X^v)$. 

If $p_2=q_2=2n$ and $p_1+q_1\geq 2n+1$, then $A=X$. Since $z$ is general, we may assume that $V_z$ is not contained in $E_{p_1}$ nor in $E^{q_1}$, and thus contains a one-dimensional subspace $V_1\subset V_z$ such that $V_1\cap E_{p_1}=V_1\cap E^{q_1}=0$. Consider any point $(a,b)\in p^{-1}(V_1\setminus 0 )\cap ((V_1^\perp \cap E_{p_1}) \times E^{q_1})$ (such $(a,b)$ always exists since $\dim(p^{-1}(V_1))\geq 2$, $\dim E_{p_1} \geq 2$ and $V_1^\perp$ is a hyperplane in $E_{p_1}$). {Then $a \wedge b\neq 0$}, $a\perp V_1$ and $b\perp V_1$ (since $a+b\in V_1$), and the point $x\in X$ such that $V_x=\langle a,b\rangle$ belongs to $X_u\cap X^v$ and is on a line passing through $z$ (since $0\neq V_1\subset V_x\cap V_z$). This shows that $z\in \Gamma_{0,1}(X_u,X^v)$.

We now assume that $p_2\neq 2n$ or $q_2\neq 2n$ and fix $z\in A$ general. Consider $V_1\subset V_z\cap E_{p_2}\cap E^{q_2}$. Since $z$ is general in $A$, we can assume that $V_1\cap E_{p_1}=V_1\cap E^{q_1}=0$. If $p_1+q_1\geq 2n+1$ then construct $a,b$ with $a \wedge b \neq 0$ as above so that $z$ is on a line passing through $x\in X_u\cap X^v$ with $V_x=\langle a,b \rangle $. We have $z\in \Gamma_{0,1}(X_u,X^v)$. If $p_1+q_1= 2n$ then, since \eqref{eqn:deg1_birat} does not hold, we can assume that $\delta_p=0$. Since $V_1\subset E_{p_2}$ we deduce that $V_1^\perp \cap E_{p_1}=E_{p_1}$ and one can still find $(a,b)\in p^{-1}(V_1)$ with $a \wedge b\neq 0$ and $a\perp V_1$; again, one deduces as above that $z\in \Gamma_{0,1}(X_u,X^v)$.
\end{proof}

\subsection{The case of $\ev_3^1(u,v)$}

In this subsection, we prove all the results necessary to obtain part \eqref{item:deg1} of Theorem \ref{thm:d=1}. Define {$E_\circ := E_{p_2} \cap E^{q_2}$.} Recall that $\oX_u$ is the affine cell in $X_u$. Define $\epsilon_1:=\max(1, p_1+q_1-2n+2)$, $\epsilon_2:=\max(1, p_2+q_2-4n+2)$, and 
\begin{multline*}
  \cU_1=\{(x,y,z)\in \Gamma_1^{(3)}(X_u,X^v) \mid x, y,z\mbox{ distinct, }x\in \oX_u,y\in\oX^v,\omega|_{V_x+V_y}\neq 0,\\ V_x\cap V_y\subseteq E_\circ\setminus(E_{p_1}\cup E^{q_1}),\dim(V_z\cap E_\circ)\leq \epsilon_2, \dim(V_z\cap (E_{p_1}+E^{q_1}))\leq \epsilon_1\}.
\end{multline*}
Note that \(\cU_1\) is a dense open subset of \(\Gamma_1^{(3)}(X_u,X^v)\). The condition $\dim(V_z\cap (E_{p_1}+E^{q_1}))\leq \epsilon_1$ is essentially saying that as soon as $E_{p_1}+E^{q_1}\neq \CC^{2n}$, $\dim(V_z\cap (E_{p_1}+E^{q_1}))\leq 1$. Similarly, the condition $\dim(V_z\cap E_\circ)\leq \epsilon_2$ implies that as soon as $p_2\neq 2n$ or $q_2\neq 2n$ then $\dim(V_z\cap E_\circ)\leq 1$. Recall the definition of  $\pi_3^1(u,v):\Gamma_1^{(3)}(X_u,X^v)\to \Gamma_1(X_u,X^v)$ from Definition~\ref{def:ev}.

\begin{lemma}
\label{lem_pi_3^1}
Assume that $q_2\neq 2n$ and let \(z\in\Gamma_1(X_u,X^v)\) be a general point. 
\begin{enumerate}
\item If \eqref{eqn:deg1_birat} holds then $\cU_1\cap \pi_3^1(u,v)^{-1}(z)$ is a single point.
\item If \eqref{eqn:deg1_birat} does not hold, then $\cU_1\cap \pi_3^1(u,v)^{-1}(z)$ is rationally connected and positive dimensional.
\end{enumerate}
\end{lemma}

\begin{proof}
Since a triple of distinct points $(x,y,z)$ in $\cU_1\cap \pi^1_3(u,v)^{-1}(z)$ belong to the same line, we have $\dim(V_x\cap V_y \cap V_z)=1$ and $\dim(V_x+V_y+V_z)=3$. But then $V_x\cap V_y$ is also contained in $V_z$. Set $W = V_x \cap V_y \cap V_z$.
For $z$ general we may assume that $\dim(V_z\cap E_\circ)= 1$, thus $W$ is uniquely defined as the intersection $V_z\cap E_\circ$.  Denote by $a,b$ two vectors generating $V_x\cap E_{p_1}$ and $V_y\cap E^{q_1}$, respectively (here we are using the condition $x\in \oX_u,y\in\oX^v$ in the definition of $\cU_1$). Since $V_x\cap V_y\subset E_\circ\setminus(E_{p_1}\cup E^{q_1})$ we deduce that $W\cap E_{p_1}=0$ and $W\cap E^{q_1}=0$ and that $a$ and $b$ do not belong to $W$. However, since $\dim(V_x+V_y+V_z)=3$, $\langle a,b\rangle \cap V_z \neq 0$. Thus, there exist $(\alpha,\beta)\neq (0,0)$ such that $\alpha a + \beta b \in V_z$. The line generated by $\alpha a + \beta b$ is contained in $W':=V_z \cap (E_{p_1} + E^{q_1})$.

If $p_1+q_1\leq 2n-1$, $W'$ is a line, and it uniquely determines $\langle a \rangle \subset V_x$ and $\langle b \rangle \subset V_y$. We obtain that $V_x=W+\langle a\rangle$ and $V_y=W+\langle b\rangle$, i.e., $\pi^1_3(u,v)^{-1}(z) \cap \cU_1=\{(x,y,z)\}$.

Assume that $p_1+q_1\geq 2n$, which implies that $W'=V_z$. Let us also suppose that $\delta_p\leq \delta_q$ (the other case can be treated similarly). Recall the definition of the map $p:E_{p_1}\times E^{q_1} \to E_{p_1}+E^{q_1}$, and define the set 
$$\cW:= \{(c,d)\in p^{-1}(V_z) \mid c\in W^\perp, \dim(W+\langle c\rangle)=\dim(W+\langle d\rangle)=2, c+d\neq 0\}.$$ This is an open subset of {{a linear subspace of $p^{-1}(V_z)$}}. Notice that since $V_z\subset W^ \perp$, we also have $d\in W^\perp$. The space \(\cW\) has dimension at least $2$ as soon as \eqref{eqn:deg1_birat} does not hold, and it has dimension $1$ if $p_1+q_1=2n$ and $1=\delta_p=\delta_q$. For each point $(c,d)\in \cW$, $V_x=W+\langle c\rangle$, and $V_y=W+\langle d\rangle$, we have \((x,y,z)\in \pi^1_3(u,v)^{-1}(z) \cap \cU_1\). Via this map $\cW$ dominates the fiber $\pi_3^1(u,v)^{-1}(z)\cap \cU_1$, which is therefore always unirational, and birational if \eqref{eqn:deg1_birat} holds.
\end{proof}

\begin{lemma}
\label{lem_pi_3^1_bis}
Assume that $q_2=p_2=2n$ and let \(z\in\Gamma_1(X_u,X^v)\) be general. If \eqref{eqn:deg1_birat} holds then $\cU_1\cap \pi_3^1(u,v)^{-1}(z)$ is a single point.
\end{lemma}

\begin{proof}
Notice that in this case $V_z\cap E_\circ=V_z$ is not a line. However, the first steps of the proof of Lemma \ref{lem_pi_3^1} still hold if one takes {\(V_x\cap V_y\cap V_z\) }as the definition of $W$; then define $a,b,\alpha,\beta,W'$ as in the proof of Lemma \ref{lem_pi_3^1}.

Since $p_1+q_1\leq 2n-1$, $W'$ is a line, and it uniquely determines $\langle a \rangle \subset V_x$ and $\langle b \rangle \subset V_y$. Moreover, we have that $W$ is uniquely defined by $W=\Ker (\omega|_{\langle a\rangle+V_z=\langle b\rangle + V_z})$ (here we use that $\omega|_{V_x+V_y}\neq 0$). We obtain that $V_x=W+\langle a\rangle$ and $V_y=W+\langle b\rangle$, i.e., $\{(x,y,z)\}=\pi_3^1(u,v)^{-1}(z) \cap  \cU_1$.
\end{proof}

\begin{lemma}
\label{lem_pi_3^1_tris}
Assume that $q_2=p_2= 2n$ and let \(z\in\Gamma_1(X_u,X^v)\) be general. If \eqref{eqn:deg1_birat} does not hold, then $\cU_1\cap \pi_3^1(u,v)^{-1}(z)$ is rationally connected and positive dimensional.
\end{lemma}

\begin{proof}
If $p_1=2n-1$ then $X_u=X$ and if $q_1=2n-1$ then $X^v=X$; in both cases the result presents no difficulties, so we assume $p_1,q_1 \leq 2n-2$. Let $(x,y,z)$ be a point in $\cU_1\cap \pi_3^1(u,v)^{-1}(z)$. Then $W:=V_x\cap V_y \cap V_z$ is a line in $V_z$. Notice that $X=\Gamma_1(X_u,X^v)$. {{This follows from Lemma \ref{lem_eq_gammas_lines} if \eqref{eqn:deg1} does not hold, otherwise it follows from Lemma \ref{lem_twotoone_lines} below.}} For $z$ in $X$ general, we have $V_z\cap (E_{p_1}\cup E^{q_1})=0$. With our assumption, $p:E_{p_1}\times E^{q_1} \to \CC^{2n}$ is surjective. 
Thus there exist $(a,b)\in p^{-1}(V_z)$ such that \(\omega(a,b)\neq0\), $V_x=\langle a \rangle + W$ and $V_y=\langle b \rangle + W$. Consider all points $(a,b)\in p^{-1}(V_z)$ such that \(\omega(a,b)\neq0\), $b\perp W$ and $\dim(\langle a,W\rangle)=\dim(\langle b,W\rangle)=2$. Then since $a+b\in V_z$, $W \subset V_z$ and $V_z$ is isotropic, the condition $b\perp W$ implies $a\perp W$. Setting $V_x=\langle a \rangle + W$ and $V_y=\langle b \rangle + W$, we get that $(x,y,z)\in \cU_1$. Thus the fiber over \(z\) of $\pi_3^1(u,v)|_{\cU_1}$ is dominated by an open subset of $A:=\{ (a,b,w)\in p^{-1}(V_z) \times V_z \mid b\perp w\}$. 

Let $\mathrm{pr}_2$ be the second projection from $E_{p_1}\times E^{q_1}$. {{There exists  $V' \subset E^{q_1}$ a (general) $2$-dimensional subspace such that $\mathrm{pr}_2 (p^{-1}(V_z))=(E_{p_1} \cap E^{q_1}) + V'$. Consider the condition $b\perp w$ on $((E_{p_1} \cap E^{q_1}) + V') \times V_z$.}} Since $q_1\leq 2n-2$, we have $V'\cap V_z=0$. 
By genericity of $z$, the symplectic form defines a general bilinear form on $V'\times V_z$, and the same genericity holds when the form is restricted to $((E_{p_1}\cap E^{q_1})+V')\times V_z$; the zero locus of this form is thus unirational. 

This shows that $A$ is unirational since it is a locally trivial fibration over such a zero locus with rational fiber over \((b,w)\) equal to \(\{a\in E_{p_1}\mid a+b\in V_z\}\). There exists a birational map $\PP(A):=\{([a,b],[w])  \in \PP(p^{-1}(V_z)) \times \PP(V_z)\ | \ \omega(b,w) = 0\} \to \cU_1\cap \pi_3^1(u,v)^{-1}(z)$, $([a,b],[w])\mapsto ([\langle a,w \rangle],[\langle b,w \rangle],z)$ which is dominant; by composing with the projection $A\to \PP(A)$ we deduce that $\cU_1\cap \pi_3^1(u,v)^{-1}(z)$ is unirational. Since $(E_{p_1}\cap E^{q_1})+V'$ is at least two-dimensional, the general fiber of $\PP(A)\to \PP(V_z)$ is non-empty, and thus $\PP(A)$ - as well as $\cU_1\cap \pi_3^1(u,v)^{-1}(z)$ - is positive dimensional.
\end{proof}

\subsection{The case of $\ev_3^{0,1}(u,v)$}

In this subsection we prove all the necessary results to obtain part \eqref{item:01birat} of Theorem \ref{thm:d=1}. We will use the same definition of $E_\circ$ as in the previous subsection. Define \begin{multline*}
  \cU_{0,1}:=\{(x,x,z)\in \Gamma^{(3)}_{0,1}(X_u,X^v) \mid x\neq z,x\in \oX_u\cap \oX^v, \\ V_x\cap E_{p_1}\cap E^{q_1}=V_z\cap V_x\cap E_{p_1}=V_z\cap V_x\cap E^{q_1}=0\}.
\end{multline*}

This is a non-empty open hence dense subset of $\Gamma^{(3)}_{0,1}(X_u,X^v)$ that we will use in the following lemmas.

\begin{lemma}
\label{lem_pi_3^0,1}
Assume that $q_2\neq 2n$ and let \(z\in\Gamma_{0,1}(X_u,X^v)\) be general. 
\begin{enumerate}
\item If \eqref{eqn:deg1_birat} holds, then $\cU_{0,1}\cap \pi_3^{0,1}(u,v)^{-1}(z)$ is a single point. 
\item If \eqref{eqn:deg1_birat} does not hold, then $\cU_{0,1}\cap \pi_3^{0,1}(u,v)^{-1}(z)$ is rationally connected.
\end{enumerate}
\end{lemma}

\begin{proof}
Set $W:=V_{x}\cap V_z$ for $(x,x,z)\in\cU_{0,1}$. We have $W\in \PP(A)$ with $A = E_\circ\cap (E_{p_1}+E^{q_1})$. Moreover, for $z\in \Gamma_{0,1}(X_u,X^v)$ general, the line $W$ is uniquely determined by the intersection $V_z\cap A$. Consider the addition map
$$ p':(E_\circ \cap E_{p_1} )\times (E_\circ \cap E^{q_1}) \to E_\circ\cap (E_{p_1}+E^{q_1}) $$
and $B:=p'^{-1}(W)\cap \mathrm{ pr}_1^{-1}(W^\perp)$, where $\mathrm{ pr}_1$ is the first projection from $(E_\circ \cap E_{p_1}) \times (E_\circ \cap E^{q_1})$. Then $B$ is a vector space. Moreover, via $\mathrm{ pr}_1$, $B$ can be realized as a locally trivial fibration over $\mathrm{ pr}_1(B)$, so the projectivization $\cP$ of this fibration inside $ \PP(E_\circ \cap E_{p_1}) \times \PP(E_\circ \cap E^{q_1})$ is rational. Notice that since $W\notin \PP(E_{p_1})\cup \PP(E^{q_1})$ by the definition of $\cU_{0,1}$ and since we are intersecting with $\mathrm{ pr}_1^{-1}(W^\perp)$, it follows that if $(a,b)\in B$ then automatically $0\neq a\in W^\perp$ and $0\neq b\in W^\perp$. Let us denote by $\cP^\circ$ the open subset in $\cP$ dominated by points $(a,b)$ such that \(a\not\in E^{q_1}\), \(b\not\in E_{p_1}\). 
Then $\cU_{0,1}\cap \pi_3^{0,1}(u,v)^{-1}(z)$ is dominated by points $(x,x,z)$ such that $V_x=\langle a,b\rangle$ with $[(a,b)]\in \cP^\circ$. This means that $\cP^\circ$ dominates $\cU_{0,1}\cap \pi_{0,1}^{-1}(z)$, which is therefore unirational. Notice finally that if \eqref{eqn:deg1_birat} holds then the set $\cP$ has dimension equal to zero (the line $W$ uniquely defines the lines $\langle a \rangle$ and $\langle b \rangle$).
\end{proof}

\begin{lemma}
\label{lem_pi_3^0,1_bis}
Assume that $q_2=p_2=2n$ and let \(z\in\Gamma_{0,1}(X_u,X^v)\) be general. If \eqref{eqn:deg1_birat} holds then $\cU_{0,1}\cap \pi_3^{0,1}(u,v)^{-1}(z)$ is a single point.
\end{lemma}

\begin{proof}
Set $W:=V_{x}\cap V_z$ for $(x,x,z)\in\cU_{0,1}$ and $A := E_{p_1}+E^{q_1}$. We have $W\in \PP(E_{p_1}+E^{q_1}) $. Moreover, for general $z\in \Gamma_{0,1}(X_u,X^v)$ the line $W$ is uniquely determined by the intersection $V_z\cap A$. Notice that in this case $p'=p$, where $p'$ was defined in the proof of Lemma \ref{lem_pi_3^0,1}.
The space $B:=p^{-1}(W)$ is one-dimensional since $p_1+q_1\leq 2n-1$. Let $(a,b)\in B$ with $a\neq 0$ and $b\neq 0$. Then $\cU_{0,1}\cap \pi_3^{0,1}(u,v)^{-1}(z)$ is given by the single point $(x,x,z)$ such that $V_x=\langle a,b\rangle$. 
\end{proof}

\begin{lemma}
\label{lem_pi_3^0,1_tris}
Assume that $q_2=p_2= 2n$ and let \(z\in\Gamma_{0,1}(X_u,X^v)\) be general. If neither \eqref{eqn:deg1_birat} nor \eqref{eqn:deg1} hold, then $\cU_{0,1}\cap \pi_3^{0,1}(u,v)^{-1}(z)$ is rationally connected and positive dimensional.
\end{lemma}

\begin{proof}
As for irreducible lines, the case when $X_u$ or $X^v$ are equal to $X$ is trivial, so we assume $q_1,p_1 \leq 2n-2$. Let $(x,x,z)$ be a point in $\cU_{0,1}\cap \pi_3^{0,1}(u,v)^{-1}(z)$. Then $W:=V_x\cap V_z$ is a one-dimensional subspace in $V_z$. Notice that $X=\Gamma_{0,1}(X_u,X^v)$ by Lemma \ref{lem_eq_gammas_lines} and for $z$ in $X$ general, $V_z\cap (E_{p_1}\cup E^{q_1})=0$. With our assumption, the map $p:E_{p_1}\times E^{q_1} \to \CC^{2n}$ is surjective and not injective. 
In particular, there exist $(a,b)\in p^{-1}(W)$ such that 
$V_x= \langle a , b \rangle = \langle a \rangle + W$. 

Conversely, consider all points $(a,b)\in p^{-1}(W)$ such that $b\perp W$ and $\dim (\langle a \rangle+W)  = \dim (\langle b\rangle + W)  = 2$, and \(a\not\in E^{q_1}\), \(b\not\in E_{p_1}\). Then since $a+b\in V_z$, $b\perp W$ and $V_z$ is isotropic, we have that $a\perp W$. Setting $V_x = \langle a , b \rangle$ we get that $(x,x,z)\in \cU_{0,1}$. This proves that the fiber over \(z\) of $\pi_3^{0,1}(u,v)|_{\cU_{0,1}}$ is dominated by an open subset of $B:=\{ (a,b,w)\in p^{-1}({{V_z}}) \times V_z \mid b\perp w,{a+b\in \langle w \rangle}\}$. 

{{Let  $\mathrm{pr}_2 : E_{p_1}\times E^{q_1} \to E^{q_1}$ be the second projection. Since $p_1+q_1>2n$, the intersection $E_{p_1}\cap E^{q_1}$ is positive dimensional and contained in $\mathrm{pr}_2 p^{-1}(\langle w \rangle)$ for any $w \in V_z$. We claim that there is no $0\neq w\in V_z$ such that $\mathrm{pr}_2 p^{-1}(\langle w \rangle)\subset w^\perp$. Indeed, this would imply the inclusion $E_{p_1} \cap E^{q_1} \subset w^\perp$ and therefore the inclusion $w \subset V_z \cap (E_{p_1} \cap E^{q_1})^\perp$. This space is trivial except for $p_1 + q_1 = 2n+1$ for which $\langle w \rangle$ would be uniquely determined. However, by genericity of $z$, this $w$ will in general not satisfy the hyperplane condition $\mathrm{pr}_2 p^{-1}(\langle w \rangle) \subset w^\perp$, proving the claim.}} 

From the claim we deduce that \emph{every} non-zero $w\in V_z$ defines a $(p_1+q_1-2n)$-dimensional affine space $w^\perp \cap \mathrm{pr}_2 p^{-1}(\langle w \rangle)$. This implies that $B$ is a locally trivial fibration over $V_z$ with fiber isomorphic to $w^\perp \cap \mathrm{pr}_2 p^{-1}(\langle w \rangle)$, thus it is rational. There exists a birational map $\PP(B):=\{ ([a,b],[w]) \in \PP(p^{-1}(V_z)) \times \PP(V_z)\ | \ \omega(b,w) = 0,a+b\in \langle w \rangle \} \to \cU_{0,1}\cap \pi_3^{0,1}(u,v)^{-1}(z)$, $([a,b],[w])\mapsto ([\langle a,w \rangle],[\langle b,w \rangle],z)$ which is dominant; by composing with the projection $B\to \PP(B)$ we deduce that $\cU_{0,1}\cap \pi_3^{0,1}(u,v)^{-1}(z)$ is unirational. Since $(E_{p_1}\cap E^{q_1})$ is at least one-dimensional, the general fiber of $\PP(B)\to \PP(V_z)$ is non-empty, and thus $\PP(B)$ - as well as $\cU_{0,1}\cap \pi_3^{0,1}(u,v)^{-1}(z)$ - is positive dimensional.
\end{proof}

\subsection{The case of condition \eqref{eqn:deg1}}

For lines, condition \eqref{eqn:deg1} holds exactly when $\ev_3^{0,1}(u,v)$ is {{generically finite but}} not birational. More precisely, the following result holds.

\begin{lemma}
\label{lem_twotoone_lines}
If \eqref{eqn:deg1} holds, then $\Gamma_1(X_u,X^v)=\Gamma_{0,1}(X_u,X^v)=X$ and $\ev_3^{0,1}(u,v)$ is generically finite of degree $2$.
\end{lemma}

\begin{proof}
 By the proof of Corollary \ref{cor:d-1-1} when $d=1$, the map $\ev_{(3)}^{0,1}(u,v)$ is birational. Thus, it is sufficient to prove that $\pi_3^{0,1}(u,v)$ is generically finite of degree $2$, and by Lemma \ref{lemma:fibre-meets-open} we are reduced to prove that, for $z\in X$ general, $\cU_{0,1}\cap \pi_3^{0,1}(u,v)^{-1}(z)$ consists of two points. Then, the first part of the proof goes as in the proof of Lemma \ref{lem_pi_3^0,1_tris}: in this case the fiber over \(z\) of $\pi_3^{0,1}(u,v)|_{\cU_{0,1}}$ is  {{$\{ ([a,b],w)\in \PP(p^{-1}(V_z)) \times \P(V_z) \mid b\perp w \textrm{ and } a \wedge {b \neq 0,a+b\in \langle w \rangle}\}$.}}
We again need to look at elements $0\neq w\in V_z$ such that $\mathrm{pr}_2 p^{-1}(\langle w \rangle) \subset w^\perp$. This condition   is not empty anymore. In fact, it is a generic degree two hyperplane condition on $V_z$. Thus, there exist exactly two lines $\langle w \rangle$ and $\langle w' \rangle$ that satisfy it. Each of these two lines defines, modulo scalars, a unique couple $(a,b)$ and $(a',b')$, and therefore $V_x= \langle a , b \rangle$ and $V_{x'} = \langle a' , b' \rangle$. We deduce that $\cU_{0,1}\cap \pi_3^{0,1}(u,v)^{-1}(z)=\{(x,x,z),(x',x',z)\}$.

The fact that $\Gamma_1(X_u,X^v)=\Gamma_{0,1}(X_u,X^v)=X$ is a consequence of the fact that throughout the proof, the point $z$ was assumed to be general in $X$, and we have proved that $\pi_3^{0,1}(u,v)^{-1}(z)\neq \emptyset$.
\end{proof}

{This was the last result needed to prove Theorem \ref{thm:d=1}. Now we study the generically finite degree $2$ map above, as well as the similar one appearing for curves of degree two, from the point of view of quantum $K$-theory.}

\section{Non rationally connected cases}
\label{sec:2-to-1}

In this section we prove positivity in $\QK(X)$ when the map $\ev_3^{d-1,1}$ is not $\GRCF$ (which is, in cases \eqref{eqn:deg1} and \eqref{eqn:deg2}). The reason why we need to single out these cases is the fact that we would like to use Corollary \ref{cor:final-formula}, but we cannot apply Theorem \ref{thm:push-forward} since $\ev_3^{d-1,1}(u,v)$ is not $\GRCF$. We prove the following result.

\begin{prop}[see Proposition \ref{prop1:2-to-1}]
\label{prop:2-to-1}
The following holds in \(\QK(X)\):
\begin{enumerate}
\item If \eqref{eqn:deg1} holds, then  $\cO_u \star \cO^v = \cO_u^v-q+q\cO_{2n-2,2n}$. 
\item If \eqref{eqn:deg2} holds, then  we have
$$\cO_{p_1,p_2} \star \cO^{q_1,q_2} = q \cO_{p_1+q_1,2n}^{p_2+q_2-2n,2n}  - q^2 + q^2 \cO_{2n-2,2n}.$$
\end{enumerate}
\end{prop}

  \begin{remark}\label{rmk:power}
    When (Cd) holds, by Theorems \ref{thm:d=2} and \ref{thm:d=1}, fibers of \(\ev_3^d(u,v)\) is positive dimensional. By the projection formula, this implies that the coefficient of \(q^d\) in \([X_u]\star[X^v]\) is \(0\) in quantum cohomology. Proposition \ref{prop:2-to-1} shows that in this case, the maximum power of \(q\) appearing in \(\cO_u\star\cO^v\), which is \(q^d\), is greater than the maximum power of \(q\) appearing in \([X_u]\star[X^v]\).
  \end{remark}

{{To compute the quantum product when \eqref{eqn:deg1} or \eqref{eqn:deg2} holds, we will use Proposition \ref{prop_formula_simplification}. Our strategy is to prove that the computation of the quantum product $\cO_u\star \cO^v$ can be essentially reduced to the computation of the class of $\cO_{\Gamma_{d-1}(X_u,X^v)}$. So we start with results expressing the class of $\cO_{\Gamma_{d-1}(X_u,X^v)}$ in terms of Schubert classes. For instance, for $d = 1$, we have $\Gamma_{d-1}(X_u,X^v) = X_u^v$, and the following result expresses $\cO_{X_u^v}$ in terms of Schubert classes when \eqref{eqn:deg1} holds.}}
 
 \begin{prop}
 \label{prop_Xuv_schubert_expansion}
Let $p \in [1,n]$, $X_u=\{z\in X\mid V_z\cap E_p\neq 0\}$ and $X^v=\{z\in X\mid V_z\cap E^{2n-p}\neq 0\}$.
 \begin{enumerate}
 \item If $p < n$, then 
 $$\cO_u^v = \cO_{p,2n-p} + 2 \sum_{k = 1}^{p-1}\cO_{k,2n-k} - 2\cO_{p-1,2n-p} - 3\sum_{k = 1}^{p-2}\cO_{k,2n-1-k} + \sum_{k = 1}^{p-2}\cO_{k,2n-2-k}.$$
 \item  If $p = n$, then 
 $$\cO_u^v = 2 \sum_{k = 1}^{n-1}\cO_{k,2n-k} - \cO_{n-1,n} - 3\sum_{k = 1}^{n-2}\cO_{k,2n-1-k} + \sum_{k = 1}^{n-2}\cO_{k,2n-2-k}.$$
 \end{enumerate}
 \end{prop}
 
 \begin{proof}
Let $X^w = X^{r_1,r_2}$, we compute $ \chi_{Y \cap g X^w}(\cO_{Y \cap g X^w})$ for $Y = X_u^v$ and use Theorem \ref{thm:brion}. We claim that the following holds:
$$ \chi_{Y \cap g X^w}(\cO_{Y \cap g X^w}) = \left\{ 
\begin{array}{ll}
0 & \textrm{ if $r_1 + r_2 < 2n+1$}, \\
0 & \textrm{ if $r_1 + r_2 > 2n+1$ and $r_2 \leq 2n - p$}, \\
2 & \textrm{ if  $r_1 + r_2 = 2n+2$ and $r_2 > 2n+1 - p$ }, \\
1 & \textrm{ if $r_1 + r_2 = 2n+2$ and $r_2 = 2n+1 - p$}, \\
1 & \textrm{ if $r_1 + r_2 > 2n+2$ and $r_2 > 2n - p$}. \\
\end{array} \right.$$
The result then follows from 
Corollary \ref{cor:schub_expand}.

We are left to prove the claim. Set $Z = Y \cap g X^w$. We have
$$Z = \{ z \in X \ | \ V_z \cap E_p \neq 0 \neq V_z \cap E^{2n-p}, V_z \cap g.E^{r_1} \neq 0 \textrm{ and } V_z \subset g. E^{r_2} \}.$$
First note that if $z \in Z$, then $E_p \cap g. E^{r_2} \neq 0$, thus we must have $r_2 + p > 2n$. If this condition is satisfied, we have $V_z \subset (E_p \cap g.E^{r_2}) \oplus  (E^{2n-p} \cap g.E^{r_2})$, thus $g. E^{r_1} \cap ( (E_p \cap g.E^{r_2}) \oplus  (E^{2n-p} \cap g.E^{r_2}) )\neq 0$. These subspaces are of dimension $r_1$ and $p+r_2 - 2n + 2n-p + r_2 - 2n = 2r_2 - 2n$ are in general position in $g.E^{r_2}$, thus we also have $r_1 + 2r_2 - 2n > r_2$, \emph{i.e.} $r_1 + r_2 > 2n$, proving the first two cases.

Assume that $r_1 + r_2 \geq 2n+2$. Set $E =  E_p \cap g.E^{r_2}$, $F = E^{2n-p} \cap g.E^{r_2}$ and $L = g.E^{r_1}$. Let $z \in Z$, then $V_z=\langle a, b\rangle$ for some {{$a \in E \setminus 0$ and $b \in F \setminus 0$}}. Furthermore, there exists scalars $\lambda,\mu \in \C$ such that $\lambda a + \mu b \in L$. For $z$ in a dense open subset of $Z$, we have $\lambda \neq 0 \neq \mu$ (otherwise $a$ or $b$ lies in $L$). Conversely, if $c \in L \cap  (E \oplus F)$ has a decomposition $c = a + b$ with  $a \in E \setminus 0$ and $b \in F \setminus 0$, then $V_z = \langle a,b \rangle$ defines a point $z \in Z$ as soon as $\omega(a,b) = 0$. This proves that $Z$ is birational to the following variety
$$Z' = \{[c] \in \bP(L \cap (E \oplus F)) \ | \ c = a + b, a \in E\setminus 0, b \in F\setminus 0 \textrm{ and } \omega(a,b) = 0 \}.$$
Assume that $r_1 + r_2 \geq 2n+2$ and $r_2 = 2n+1 - p$, then $\dim E = 1$ and $E = \langle a \rangle$. We then have $Z' = \{[c] \in \bP(L \cap (E \oplus F)) \ | \ \omega(c,a) = 0\} =\bP(L \cap (E \oplus F) \cap E^\perp)$ proving the result in this case.

Finally, assume that $r_1 + r_2 \geq 2n+2$ and $r_2 > 2n+1 - p$, then $\dim E > 1$. The condition $c = a + b$ with $\omega(a,b) = 0$ defines a quadratic form on $\bP(L \cap (E \oplus F))$, proving the result in this case as well. Note that for $r_1 + r_2 = 2n+2$, we have $\bP(L \cap (E \oplus F))\simeq \bP^1$ so that the quadric is the union of two points.
 \end{proof}

We consider the case $d = 2$ and assume that \eqref{eqn:deg2} holds, i.e., we have $p_1 + q_2 = 2n = p_2 + q_1$, $p_2 - p_1 = q_2 - q_1 \geq 2$, and $\max(\delta_p,\delta_q) = 1$. We may assume $p_1 + p_2 \leq q_1 + q_2$, the other case being symmetric. In this case, we have $p_1 + p_2 \leq 2n - 2$, $\delta_p = 0$, $q_1 + q_2 \geq 2n+2$, and $\delta_q = 1$. {{Note that condition \eqref{eqn:deg2} implies the decomposition $\C^{2n} = E_{p_1} \oplus (E_{p_2} \cap E^{q_2}) \oplus E^{q_1}$.}}

\begin{lemma}
If \eqref{eqn:deg2} holds, then 
$$\Gamma_1(X_u,X^v) = \{ z \in X \ | \ V_z \cap (E_{p_2} \cap E^{q_2}) \neq 0 \neq V_z \cap (E_{p_1} + E^{q_1}) \}.$$
\end{lemma}

\begin{proof}
Let $z \in \Gamma_1(X_u,X^v)$. Then there exists $x \in X_u$ and $y \in X^v$ with $V_x \cap V_y \neq 0$ and $V_x \cap V_y \subset V_z \subset V_x + V_y$. In particular this proves the inclusion from left to right. Conversely, let $z$ in the right-hand side of the equality. Let $a \in V_z \cap (E_{p_2} \cap E^{q_2})$ and $b \in V_z \cap (E_{p_1} + E^{q_1})$ with $a \neq 0 \neq b$. Since $E_{p_2} \cap E^{q_2} \cap (E_{p_1} + E^{q_1}) = 0$, we have $V_z = \langle a,b \rangle$. Write $b = b_p + b_q$ with $b_p \in E_{p_1}$ and $b_q \in E^{q_1}$. If $b_p \neq 0 \neq b_q$, then since $\omega(a,b) = 0$ and $\delta_p = 0$, we have $\omega(a,b_p) = 0 = \omega(a,b_q)$ and setting $V_x = \langle a,b_p \rangle$ and $V_y = \langle a,b_q \rangle$, we have $x \in X_u$, $y \in X^v$, and $z$ lies on the line joining $x$ and $y$. If $b_q = 0$, then $z \in X_u$ and since $q_1 \geq p_1 + 2$, we may choose $c \in E^{q_1} \cap a^\perp \setminus 0$ so that setting $V_y = \langle a,c \rangle$, we have $y \in X^v$, proving the result. Finally, if $b_p = 0$, then $z \in X^v$. Choose any $c \in E_{p_1} \setminus 0$; since $\delta_p = 0$, setting $V_x = \langle a,c \rangle$, we have $x \in X_u$, proving the result.
\end{proof}

{Set $A = E_{p_1} + E^{q_1}$, $B = E_{p_2} \cap E^{q_2}$, $p=\dim A$, and $q=\dim B $. When \eqref{eqn:deg2} holds, we have $p = p_1 + q_1$, $q = p_2 + q_2 -2n = 2n - p$, and $\C^{2n} = A \oplus B = E_p \oplus E^q$.}
Define $X_\lambda = \{ z \in X \ | \ V_z \cap E_p \neq 0 \}$, $X^\mu = \{ z \in X \ | \ V_z \cap E^q \neq 0 \}$, and $X_\lambda^\mu = X_\lambda \cap X^\mu$.

\begin{lemma}
\label{gamma_c2_expansion}
If \eqref{eqn:deg2} holds, then $\cO_{\Gamma_1(X_u,X^v)} = \cO_{X_\lambda^\mu}$.
\end{lemma}

\begin{proof}

Define $Z_1 = \{ ([a],[b]) \in \bP(E_p) \times \bP(E^q) \ | \ \omega(a,b) = 0\}$ and $Z_0 = \{ ([a],[b]) \in \bP(A) \times \bP(B) \ | \ \omega(a,b) = 0\}$. We have isomorphisms $Z_1 \to X_\lambda^\mu$ and $Z_0 \to \Gamma_1(X_u,X^v)$ given by $([a],[b]) \mapsto \langle a,b \rangle$. 

Define the linear isomorphism $\psi :  \C^{2n} \to \C^{2n}$ via 
$$\psi(e_k) = \left\{ 
\begin{array}{ll}
e_k & \textrm{ for } k \in [1,p_1] \\
e_{k+p_2-p_1} & \textrm{ for } k \in [p_1+1,p_1+q_1] \\
e_{k-q_1} & \textrm{ for } k \in [p_1+q_1+1,2n]. \\
\end{array}\right.$$
For $t \in \C$ define $\varphi_t = t\ {\rm Id}_{\C^{2n}} + (1-t) \psi$. We have $\varphi_0(E_p) = A$ and $\varphi_0(E^q) = B$, while $\varphi_1(E_p) = E_p$ and $\varphi_1(E^q) = E^q$. Define
$$Z = \{ ([a],[b],t) \in \bP(E_p) \times \bP(E^q) \times \C \ | \ \omega(\varphi_t(a),\varphi_t(b)) = 0\}.$$
The form $\omega_t(a,b) =  \omega(\varphi_t(a),\varphi_t(b))$ induces a {{bilinear form $\Omega_t : E_p \times E^q \to \CC$.}} This {{form}} has maximal rank for $t = 1$ and rank $2n - (p_1 + p_2) \geq 2$ for $t = 0$. Let $U \subset \C$ be the open subset defined by $U = \{ t\in \C \ | \ \det(\varphi_t) \neq 0 \textrm{ and } {\rm Rk}(\Omega_t) \geq 2\}$. Note that $0,1 \in U$. Let $\pi : Z \to \C$ be the third projection. We have $\pi^{-1}(0) = Z_0$ and $\pi^{-1}(1) = Z_1$. We have a morphism $e : Z \to X,([a],[b],t) \mapsto \langle \varphi_t(a),\varphi_t(b) \rangle$ which restricts to an isomorphism {onto its image }on $Z_t := \pi^{-1}(t)$. {{We therefore only need to prove that the map $\pi^{-1}(U) \to U$ is flat. Since $Z$ is defined by a unique non-trivial equation, any irreducible component of $Z$ has dimension $p + q -2$. By Miracle Flatness, we only need to prove that the fibers of $\pi$ have dimension at most $p + q - 3$ which follows from the fact that ${\rm Rk}(\Omega_t) \geq 2$.}}
\end{proof}

{\begin{lemma}
  \label{lem_gamma_1_xv}
    We have \(\Gamma_1(X^v)=\{W_2\in X\ |\ \dim(W_2\cap E^{q_2})\geq1\}\).
  \end{lemma}
  
  \begin{proof}
    If \(W_2\in\Gamma_1(X^v)\), then there exists \(V_2\in X^v\) such that \(\dim(W_2\cap V_2)\geq 1\). Since \(V_2\subset E^{q_2}\), we must have \(\dim(W_2\cap E^{q_2})\geq1\). Therefore, \(\Gamma_1(X^v)\subseteq\{W_2\in X\ |\ \dim(W_2\cap E^{q_2})\geq1\}\). 
    
    For the reverse containment, let \(W_2\in X\setminus X^v\) such that \(\dim(W_2\cap E^{q_2})\geq1\). Let \(u\in(W_2\cap E^{q_2})\setminus0\). If \(W_2\cap E^{q_1}=0\), let \(v\in(E^{q_1}\setminus0)\cap\langle u\rangle^\perp\); otherwise, \(\dim(W_2\cap E^{q_1})=\dim(W_2\cap E^{q_2})=1\) and let \(v\in(E^{q_2}\setminus W_2)\cap \langle u\rangle^\perp\). By construction, \(\dim\langle u,v\rangle=2\) and \(\langle u,v\rangle\in X^v\). Moreover, \(L=\{V_2\ |\ \dim V_2=2\text{ and }\langle u\rangle \subset V_2\subset W_2+\langle v\rangle\}\) is a line in \(X\) connecting \(W_2\) and \(X^v\). 
  \end{proof}}

\begin{proof}[Proof of Proposition \ref{prop:2-to-1}]
From Corollary~\ref{cor_join_through_lines} and Proposition \ref{prop_formula_simplification} we deduce that $$
{{\kappa}_{u,v}^{w,d}}=\langle \cO_u,\cO^v, \cO_{w}^\vee \rangle_{d} - \sum_{\kappa} \langle \cO_u,\cO^v, \cO_{\kappa}^\vee \rangle_{d-1}  \langle \cO_\kappa,\cO_{w}^\vee \rangle_{1}.$$ Note that $\langle \cO_\kappa , \cO_w^\vee \rangle_{1} = 0$ except for $\Gamma_1(X_\kappa) = X_w.$ So fixing $\kappa$ with this property, we get 
$${{\kappa}}_{u,v}^{w,d} = \langle \cO_u , \cO^v , \cO_w^\vee \rangle_{d} - \sum_{\kappa,\, \Gamma_1(X_\kappa) = X_w} \langle \cO_u , \cO^v , \cO_\kappa^\vee \rangle_{d-1}.$$
{By Corollary~\ref{cor:dleq2}, it suffices to compute \({\kappa}_{u,v}^{w,d}\) for \(0\leq d\leq2\). }Furthermore, by Theorems \ref{thm:push-forward} and \ref{thm:all-d}, for any degree $r$, we have
$$\cO_{\Gamma_{r}(X_u,X^v)} = \sum_{\eta} \langle \cO_u , \cO^v , \cO_\eta^\vee \rangle_r \cO_\eta.$$
If {\eqref{eqn:deg1} holds, then $\Gamma_1(X_u,X^v) = X$ by Lemma \ref{lem_twotoone_lines}; if \eqref{eqn:deg2} holds, then $\Gamma_2(X_u,X^v) = X$ by Corollary \ref{cor:gamma_2=}, and $\cO_{\Gamma_1(X_u,X^v)}=\cO_{X_\lambda^\mu}$ by Lemma \ref{gamma_c2_expansion}. The result then follows from Proposition~\ref{prop_Xuv_schubert_expansion} and Lemma~\ref{lem_gamma_1_xv}.} 
\end{proof}

\section{Seidel representation}
\label{section:seidel}

For $X$ smooth and projective, Seidel introduced in \cite{seidel:pi_1} a representation of $\pi_1({\rm Aut}(X))$ into the group of invertibles of $\QH(X)_{\rm loc}$ where the quantum parameters are inverted. This representation was computed explicitly in \cite{chaput.manivel.ea:affine} (see also \cite{chaput.perrin:affine}). Recently these results were extended in the quantum $K$-theory of cominuscule spaces in \cite{buch.chaput.ea:seidel}. We prove that these results extend to $X = \IG(2,2n)$ for the quantum $K$-theory $\QK(X)$. 

Note that $\pi_1({\rm Aut}(X)) = \Z/2\Z$. This group can be realized as a subgroup of the Weyl group as follows. Fix $G = \Sp_{2n}$ and $T \subset B \subset G$ a maximal torus and a Borel subgroup. Let $W$ be the Weyl group associated to $(G,T)$ and $w_0$ be the longest element for the length induced by $B$. For $Y$ projective and homogeneous under $G = \Sp_{2n}$, let $P_Y$ be the parabolic subgroup containing a fixed Borel subgroup $B$ such that $Y \simeq G/P_Y$. For any such $Y$, let $w^Y$ be the minimal length representative of $w_0$ in $W/W_{P_Y}$. We call $Y = G/P_Y$ cominuscule if the unipotent radical of $P_Y$ is abelian. We have 
$$\pi_1({\rm Aut}(X)) \simeq W^{\rm comin} = \{1\} \cup \{ w^Y \in W \ | \ Y \textrm{ cominuscule} \} = \{1,w^{\LG(n,2n)} \}.$$
 We first prove the following geometric result. For $u,w \in W$, let $d_{\rm min}(u,w)$ be the smallest power of $q$ appearing in the quantum product $[X^u] \star [X^w]$. By \cite{chaput.manivel.ea:affine}, we have $ [X_{w_0w}] \star [X^u] = q^{d_{\rm min}(u,w)}[X^{wu}]$.

 \begin{thm}
 \label{thm:seidel-geom}
 Let $u \in W$, $w \in W^{\rm comin}$, then $\Gamma_{d_{\rm min}(u,w)}(X_{w_0w},X^u) = w^{-1}.X^{wu}$.
 \end{thm}
 
This result is a special case of a conjecture stated in  \cite{buch.chaput.ea:seidel} (for more recent results on this topic we refer to \cite{tarigradschi:curve}). The statement is easily true for $w = 1$ so let $w \in W^{\rm comin} \setminus \{1\}$. We have $X^w = \{ x \in X \mid V_x \subset E^n\}$ and $X_{w_0w} =  \{ x \in X \ | \ V_x \subset E_n\}$. We prove the following explicit result which implies Theorem \ref{thm:seidel-geom}. Let $X^u = X^{p_1,p_2} := \{ x \in X \ | \ V_x \cap E^{p_1} \neq 0 \textrm{ and } V_x \subset E^{p_2} \}$ with $p_1 < p_2$ and $p_1 + p_2 \neq 2n+1$.

\begin{prop} 
\label{prop:conjecture}
Let $w \in W^{\rm comin} \setminus \{1\}$ and set $d = d_{\rm min}(u,w)$, we have:
\begin{enumerate}
\item If $p_2 \leq n$, then $d = 2$ and $\Gamma_d(X_{w_0w},X^u) = w^{-1}.X^{p_1+n,p_2+n}$.
\item If $p_1 \leq n < p_2$, then $d = 1$ and $\Gamma_d(X_{w_0w},X^u) = w^{-1}.X^{p_2-n,p_1+n}$.
\item If $p_1 > n$, then $d = 0$ and $\Gamma_d(X_{w_0w},X^u) = w^{-1}.X^{p_1-n,p_2-n}$.
\end{enumerate}
\end{prop}

\begin{proof}
Note that we have $w^{-1} = w$ and that $w$ acts as follows on $(e_i)_{i \in [1,2n]}$:
$$w(e_i) = \left\{ 
\begin{array}{ll}
e_{i+n} & \textrm{ for $i \leq n$} \\ 
e_{i- n} & \textrm{ for $i > n$.}\\
\end{array}\right.$$
Note also that $d = d_{\rm min}(u,w) = \min\{d \ | \ \Gamma_d(X_{w_0w},X^u) \neq \emptyset\}$.

(1) Assume $p_2 \leq n$, note that this implies $\delta_p = 0$. For $e \leq 1$ and for any $z \in \Gamma_e(X_{w_0w},X^u)$, we have $V_z \cap E_n \cap E^{p_2} \neq 0$. But the condition $p_2 \leq n$ implies the vanishing $E_n \cap E^{p_2} = 0$ proving the equality $\Gamma_e(X_{w_0w},X^u) = \emptyset$. 
For $z \in \Gamma_2(X_{w_0w},X^u)$ general, there exists $x \in X_{w_0w}$ and $y \in X^u$ such that $V_z \subset V_x + V_y$ {by Lemma~\ref{lem:deg2}}. In particular $V_z \subset E_n \oplus E^{p_2}$ and $V_z \cap (E_n \oplus E^{p_1}) \neq 0$. Since $w^{-1}. E^{n+p_1} = E_n \oplus E^{p_1}$ and $w^{-1}. E^{n+p_2} = E_n \oplus E^{p_2}$, we get the inclusion $\Gamma_2(X_{w_0w},X^u) \subset w^{-1}.X^{wu}$. Conversely, for $z \in w^{-1}.X^{wu}$ general, then $V_z \subset E_n \oplus E^{p_2}$ and $V_z \cap (E_n \oplus E^{p_1}) \neq 0$. Pick a basis $(a,b)$ of $V_z$ such that $a \in V_z  \cap (E_n \oplus E^{p_1})$. Write $a = a_1 + a_2$ and $b = b_1 + b_2$ with $a_1,b_1 \in E_n$, $a_2 \in E^{p_1}$ and $b_2 \in E^{p_2}$. Then setting $V_x = \langle a_1,b_1 \rangle$ and $V_y = \langle a_2,b_2 \rangle$, we have $x \in X_{w_0w}$, $y \in X^u$ and $V_z \subset V_x + V_y$, proving the result.

(2) Assume $p_1 \leq n < p_2$. If $z \in X_{w_0w} \cap X^u$, then $V_z \cap E_n \cap E^{p_1} \neq 0$, but since $E_n \cap E^{p_1} = 0$, we get $\Gamma_0(X_{w_0w},X^u) = \emptyset$. For $z \in \Gamma_1(X_{w_0w},X^u)$ general,  there exists $x \in X_{w_0w}$ and $y \in X^u$ such that $\dim(V_x + V_y) = 3$ and $V_x \cap V_y \subset V_z \subset V_x + V_y$. In particular $V_z \subset E_n \oplus E^{p_1}$ and $0 \neq V_z \cap V_x \cap V_y \subset V_z \cap (E_n \cap  E^{p_2})$. Since $w^{-1}. E^{n+p_1} = E_n \oplus E^{p_1}$ and $w^{-1}. E^{p_2-n} = E_n \cap E^{p_2}$, we get the inclusion $\Gamma_1(X_{w_0w},X^u) \subset w^{-1}.X^{wu}$. 

Conversely, for $z \in w^{-1}.X^{wu}$ general, we have $V_z \subset E_n \oplus E^{p_1}$ and $V_z \cap (E_n \cap E^{p_2}) \neq 0$. Pick a basis $(a,b)$ of $V_z$ such that $a \in V_z  \cap (E_n \cap E^{p_2})$. Write $b = b_1 + b_2$ with $b_1 \in E_n$ and $b_2 \in E^{p_1}$. We have $\omega(a,b_2) = \omega(a,b) - \omega(a,b_1)$ and since $a,b \in V_z$ and $a,b_1\in E_n$, both terms vanish. Setting $V_x = \langle a,b_1 \rangle$ and $V_y = \langle a,b_2 \rangle$, we have $x \in X_{w_0w}$, $y \in X^u$, and $V_x \cap V_y \subset V_z \subset V_x + V_y$, proving the result.

(3) We have $\Gamma_0(X_{w_0w},X^u) = X_{w_0w}^u$. An element $z\in X$ lies in $X_{w_0w}^u$ if and only if we have $V_z \subset E_n \cap E^{p_2}$ and $V_z \cap (E_n \cap E^{p_1})\neq 0$. The result follows since $w^{-1}.E^{p_1-n} = E_n \cap E^{p_1}$ and $w^{-1}.E^{p_2-n} = E_n \cap E^{p_2}$.
\end{proof}
 
\begin{thm}
 \label{thm:seidel-Ktheo}
 Let $u \in W$, $w \in W^{\rm comin}$, then $\cO^u \star \cO^w = q^{d_{\rm min}(u,w)} \cO^{wu}$.
 \end{thm}
 
 \begin{proof}
Note that $\cO_{w_0w} = \cO^w$ and that the maps $\ev^d_3(w_0w,u) : M_d(X_{w_0w},X^u) \to \Gamma_d(X_{w_0w},X^u)$ and $\ev_3^{d-1,1}(w_0w,u) : M_{d-1,1}(X_{w_0w},X^u) \to \Gamma_{d-1,1}(X_{w_0w},X^u)$ have general fibers which are either empty or rationally connected. {Moreover, whenever the former map is not birational, we have \(\Gamma_d(X_{w_0w},X^u)=\Gamma_{d-1,1}(X_{w_0w},X^u)\) and that they have rational singularities (Theorem~\ref{thm:all-d}). 
These and Corollary~\ref{cor:final-formula} imply that }the powers of $q$ appearing in $\cO^{u} \star \cO^w$ are the same as those appearing in $[X_{w_0w}] \star [X^u]$. That is to say, {\(\cO^u \star \cO^w=q^d(\cO^u \star \cO^w)_d\) for \(d=d_{\rm min}(u,w)\). }
The result then follows from {Theorem~\ref{thm:seidel-geom} and }the fact that $\Gamma_{d_{\rm min}(u,w)-1,1} (X_{w_0w},X^u) = \emptyset$ since $\Gamma_{d_{\rm min}(u,w)-1} (X_{w_0w},X^u) = \emptyset$ by Proposition~\ref{prop:conjecture}.
\end{proof}

\section{Chevalley formula}
\label{section:chevalley}

In this section we obtain a quantum Chevalley formula for $\QK(X)$ (see Theorem \ref{thm_quantum_chevalley}). We begin by recalling the classical Chevalley formula.

\subsection{Classical Chevalley formula}

{Proposition~\ref{prop_classical_chevalley} below is a specialization of \cite{lenart.postnikov:affine}, where {we extend the notation \(\cO_{a,b}\)
by setting 
\begin{equation*}
  \cO_{a,b}=
  \begin{cases}
    [\cO_{X_{a,b}}] & 1\leq a < b\leq 2n \text{ and } a + b \neq 2n+1 \\
    0 & \text{otherwise}.
  \end{cases}
\end{equation*}}}

\begin{prop}[Classical Chevalley formula]
\label{prop_classical_chevalley}
The product \(\cO_{2n-2,2n}\cdot\cO_{q_1,q_2}\) in \(K(X)\) is equal to:
\begin{enumerate}
  \item \(\cO_{q_1-1,q_2}\) if \(q_1=q_2-1\);
  \item \(\cO_{q_1-1,q_2}+\cO_{q_1,q_2-1}-\cO_{q_1-1,q_2-1}\) if \(q_1<q_2-1\) and \(q_1+q_2\neq 2n+2, 2n+3\);
  \item \(\cO_{q_1-1,q_2}+\cO_{q_1,q_2-1}-\cO_{q_1-1,q_2-2}-\cO_{q_1-2,q_2-1}+\cO_{q_1-2,q_2-2}\) if \(q_1<q_2-1\) and \(q_1+q_2=2n+3\);
  \item \(2\cdot\cO_{q_1-1,q_2-1}+\cO_{q_1-2,q_2}+\cO_{q_1,q_2-2}-2\cdot\cO_{q_1-2,q_2-1}-2\cdot\cO_{q_1-1,q_2-2}+\cO_{q_1-2,q_2-2}\) when \(q_1+q_2=2n+2\) and {\(q_1\neq n\)};
  \item \(2\cdot\cO_{q_1-1,q_2-1}+\cO_{q_1-2,q_2}-2\cdot\cO_{q_1-2,q_2-1}-\cO_{q_1-1,q_2-2}+\cO_{q_1-2,q_2-2}\) when \(q_1=n\) and \(q_2=n+2\).
\end{enumerate}
\end{prop}

\begin{proof}
We sketch a geometric proof different from the proof in \cite{lenart.postnikov:affine}. Let \(X_u=X_{2n-2,2n}\) and \(X^v=X^{q_1,q_2}\). Since \(\cO_u\cdot\cO^v\)=\(\cO_u^v\), we apply Theorem \ref{thm:brion} to \(Y:=X_u^v\). Let \(g\in G\) be general, $X^w = X^{r_1,r_2}$, and \(Z=Y\cap g X^w=X_u\cap X^v\cap g X^w\). Assume \(Z\) is non-empty, then the dimension of \(X^v\cap gX^w\) must be at least one, which implies 
\begin{equation}\label{eqn:non-empty}
  \text{either }\begin{cases}
    q_1+r_2\geq 2n+2\\
    q_2+r_1\geq 2n+1
  \end{cases}
  \text{ or }\begin{cases}
    q_1+r_2\geq 2n+1\\
    q_2+r_1\geq 2n+2
  \end{cases}.
\end{equation}

Let \(F_1=E^{q_1}\cap gE^{r_2}\) and \(F_2=E^{q_2}\cap gE^{r_1}\). For \(x\in Z\) general, we have \[\dim\left(V_x\cap (F_1+F_2)\cap E_{2n-2}\right)=1\text{ and }V_x\cap F_1\cap F_2=V_x\cap F_1\cap E_{2n-2}=V_x\cap F_2\cap E_{2n-2}=0.\] Therefore, the vector space \(V_x\) is spanned by nonzero vectors \(v_1\in F_1\) and \(v_2\in F_2\) such that \(v_1+v_2\in V:=E_{2n-2}\cap (F_1+F_2)\setminus(F_1\cup F_2)\) and \(\omega(v_1,v_2)=0\). Consider the map given by addition
\[
  p: F_1\times F_2\to F_1+F_2.  
\]
Then \(Z\) is dominated by 
\[
  \mathcal{W}:=\{(v_1,v_2)\in p^{-1}(V)\mid \omega(v_1,v_2)=0\}.
\]

Note that if \(\dim F_1\geq 3\), then for any \(v_2\in F_2\setminus F_1\), we have \(\dim ({v_2}^\perp\cap F_1)\geq 2\) and \(\left(\langle v_2\rangle+({v_2}^\perp\cap F_1)\right)\cap E_{2n-2}\neq 0\). This implies that \(\cW\) is a locally trivial fibration over \(F_2\setminus F_1\) with the fiber being an open subset of the set of solutions to a system of linear equations. Hence, \(Z\) is unirational. Similarly, \(Z\) is unirational if \(\dim F_2\geq 3\).

Now assume \(\dim F_i\leq 2\) for \(i=1,2\), which implies \(q_1+r_1\leq 2n+1\). The remaining cases are as follows.

If \(q_1+r_1=2n+1\), then \(q_2+r_2=2n+3\). This implies \(q_2-q_1=r_2-r_1=1\) and \(\min\{\delta_q,\delta_r\}=0\). Hence, the condition \(\omega(v_1,v_2)=0\) is trivial and \(\cW\) is a locally trivial fibration over \(F_2\setminus F_1\) with the fiber being an open subset of the set of solutions to a system of linear equations. Hence, \(Z\) is unirational.

If \(q_1+r_1\leq 2n\) and \(\min\{\delta_q,\delta_r\}=0\), then $F_1 \cap F_2=\emptyset$, which implies that $p(\mathcal{W}) \cong \mathcal{W}$; then \(Z\) is birational to \(\PP\left(E_{2n-2}\cap(F_1+F_2) \right)\setminus\left(\PP(F_1)\cup\PP(F_2)\right)\). 

Otherwise, \(q_1+q_2=r_1+r_2=q_1+r_2=q_2+r_1=2n+2\). In this case, the condition \(\omega(v_1,v_2)=0\) is non-trivial and \(Z\) is birational to the vanishing set of a quadric on $\PP^1\cong \PP\left(E_{2n-2}\cap(F_1+F_2) \right)$.

In summary, \(\chi_Z(\cO_Z)=0\) if \eqref{eqn:non-empty} does not hold; assuming \eqref{eqn:non-empty}, we have 
\begin{eqnarray*}
  \chi_Z(\cO_Z)=
  \begin{cases}
    2 & \text{if }q_1+q_2=r_1+r_2=q_1+r_2=q_2+r_1=2n+2\\
    1 & \text{otherwise}.
  \end{cases}
\end{eqnarray*}
The rest follows from Corollary \ref{cor:schub_expand}.
\end{proof}

\subsection{Geometric properties of stable curves meeting a Schubert divisor}

Let \(\Gamma_{\underline{d}}(X^v)=\ev_3(\ev_2^{-1}(X^v))\), 
where \(\ev_1,\ev_2,\ev_3: M_{\underline{d}}\to X\) are the 
evaluation maps. From now on $X_u=X_{{2n-2,2n}}$ will be a Schubert divisor and $X^v=X^{q_1,q_2}$.

\begin{prop}\label{prop:chev-nbhds}
Let $X_u$ be a Schubert divisor. We have
\begin{enumerate}
\item\label{item:Gamma_d} $\Gamma_d(X_u,X^v) = \Gamma_d(X^v)$ for all $d > 0$.
\item\label{item:Gamma01} $\Gamma_{0,1}(X_u,X^v) = \Gamma_1(X_u^v)$.
\item\label{item:X} $\Gamma_{d-1,1}(X_u,X^v) = \Gamma_d(X_u,X^v) = X$ for all $d > 1$.
\item\label{item:rat-conn} The general fibers of the maps $\ev_3^d(u,v) : M_d(X_u,X^v) \to \Gamma_d(X_u,X^v)$  and $\ev_3^{d-1,1}(u,v) : M_{d-1,1}(X_u,X^v) \to \Gamma_{d-1,1}(X_u,X^v)$ are rationally connected for all $d\geq 0$ except for $\ev_3^{0,1}(u,v)$ when $q_1=2$ and $q_2=2n$, whose general fibers consist of two points.
\end{enumerate}
\end{prop}

\begin{proof}
Part~\eqref{item:Gamma_d} follows from the fact that when \(d>0\), a degree \(d\) curve must meet the divisor \(X_u\). For part~\eqref{item:Gamma01}, note that for a degree \(0\) curve meeting \(X_u\) and \(X^v\) is a point in \(X_u\cap X^v=X_u^v\). For Part~\eqref{item:X}, note that when \(d>1\), \(\Gamma_{d-1,1}(X_u,X^v)\supseteq\Gamma_{1,1}(X_u,X^v)\supseteq\Gamma_{1,1}(X^v)\), where the last containment follows from the fact that a degree \(1\) curve must meet \(X_u\); by \eqref{item:Gamma_d}, we have \(\Gamma_d(X_u,X^v) = \Gamma_d(X^v)\supseteq\Gamma_2(X^v)\). Finally, \(\Gamma_{1,1}(X^v)=\Gamma_2(X^v)=X\) by Lemma~\ref{lem:split-deg2}. Notice that since $X_u=X_{2n,2n-2}$, \eqref{eqn:deg2} is never satisfied and \eqref{eqn:deg1} is satisfied exactly when $q_1=2$ and $q_2=2n$; then part~\eqref{item:rat-conn} follows from Theorem \ref{thm:all-d}.
\end{proof}

Note that \(\dim M_1(X_u,X^v)=\dim X+2n-1-1-\codim X^v=\dim X+2n-1-1-(4n-2-q_1-q_2+\delta_q)=\dim X-2n+q_1+q_2-\delta_q\), and \(\dim\Gamma_1(X_u,X^v)=\dim\Gamma_1(X^v)=\dim X-(4n-2-q_2-2n+1)=\dim X-2n+1+q_2\) if \(q_2\leq 2n-1\), and \(\dim\Gamma_1(X_u,X^v)=\dim X\) if \(q_2=2n\). Therefore, as expected, \(\dim M_1(X_u,X^v)=\dim\Gamma_1(X_u,X^v)\) if and only if \eqref{eqn:deg1_birat} holds, i.e. if and only if
\begin{equation}\label{item:deg1birat}
  q_1=1,\ q_2\leq 2n-1
\end{equation}

\begin{lemma}\label{lemma:Gamma01}
$\Gamma_1(X_u^v)=\Gamma_1(X^v)$ unless $q_1=1$, in which case $\Gamma_1(X_u^v)=\{W_2\in X \mid \dim(W_2 \cap E^{q_2} \cap(E_{2n-2}\oplus E^1))\geq 1\}$.
\end{lemma}

\begin{proof}
The fact that $\Gamma_1(X_u^v)=\Gamma_1(X^v)$ is a consequence of Lemma \ref{lem_eq_gammas_lines} when $q_2\neq 2n$ and $q_1>1$ or $q_2=2n$ and $q_1>2$, and of Lemma \ref{lem_twotoone_lines} when $q_2=2n$ and $q_1=2$. Therefore, let us assume that $q_1=1$ and $q_2\leq 2n-1$. Note that 
  \[
    X_u^v=\{V_2\in X\ |\ \dim(V_2\cap E_{2n-2})\geq 1,\ E^1\subset V_2\subseteq E^{q_2}\}.
  \] 
 Set \(A=
  \{W_2\in X\ |\ \dim(W_2\cap (E_{2n-2}\oplus E^{1})\cap E^{q_2})\geq 1\}\). To see that \(\Gamma_1(X_u^v)\subseteq A\), let \(W_2\in \Gamma_1(X_u^v)\); then \(\dim(W_2\cap V_2)\geq 1\) for some \(V_2\in X_u^v\). Note that \(V_2=E^{1}\oplus(V_2\cap E_{2n-2}\cap E^{q_2})\subseteq (E_{2n-2}\oplus E^{1})\cap E^{q_2}\) and it follows that \(\Gamma_1(X_u^v)\subseteq A\).
  
  For \(\Gamma_1(X_u^v)\supseteq A\), let \(W_2\in A\setminus X_u^v\) and \(w\in W_2\cap (E_{2n-2}\oplus E^{1})\cap E^{q_2}\setminus 0\). If \(w\in E^{1}\), then let \(z\in E_{2n-2}\cap E^{q_2}\setminus 0\), and \(L=\{V_2\ |\ \langle w\rangle \subset V_2\subset W_2\oplus\langle z\rangle\}\) is a line in \(X\) containing \(W_2\) and the point \(\langle w,z \rangle\) is in \(X_u^v\). If \(w\not\in E^{1}\), we can write \(w=w_1+w_2\), where \(w_1\in E^{1}\) and \(w_2\in E_{2n-2}\cap E^{q_2}\). Since $q_1=1$ and $q_2\leq 2n-1$, \(w_1\perp w_2\). Therefore, \(L=\{V_2\ |\ \langle w\rangle\subset V_2\subset W_2+\langle w_1,w_2\rangle\}\) is a line in X containing \(W_2\) and the point \(\langle w_1, w_2\rangle\) is in \(X_u^v\).
\end{proof}

\begin{remark}
The above lemma implies that $[\cO_{\Gamma_1(X_u^v)}]=[\cO_{{q_2-1,2n}}]$ when $q_1=1$.
\end{remark}

\subsection{Quantum Chevalley formula}

By Corollary~\ref{cor:final-formula}, Theorem~\ref{thm:push-forward} and Proposition~\ref{prop:chev-nbhds}, 
\begin{equation*}
  (\cO_u\star\cO^v)_d=
  \begin{cases}
    0 & d\geq 2\\
    \cO_u^v & d=0
  \end{cases}
\end{equation*}
We thus only need to compute $(\cO_u\star\cO^v)_d$ when $d=1$. When \eqref{eqn:deg1} holds, i.e. when $q_1=2$ and $q_2=2n$, we can use Propositions \ref{prop:2-to-1} and \ref{prop_classical_chevalley} to deduce:
$$\cO_{2n-2,2n} \star \cO_{2,2n} =  2\cdot \cO_{1,2n-1} + \cO_{2,2n-2} - 2\cdot\cO_{1,2n-2} - q\cO_{2n-1,2n} + q\cO_{2n-2,2n}.$$

If \eqref{eqn:deg1_birat} holds, i.e. if \(q_1=1,q_2\leq 2n-1\), because of Propositions~\ref{prop_classical_chevalley}, \ref{prop:chev-nbhds} and Lemmas~\ref{lem_gamma_1_xv}, \ref{lemma:Gamma01}, we have 
\[
  \cO_{2n-2,2n}\star\cO_{1,q_2}=\cO_{1,q_2-1}+ q\cO_{q_2,2n}-q\cO_{q_2-1,2n}.
\]

In all other cases, \((\cO_u\star\cO^v)_1=0\).

Let us denote by
\[
  \cO_{a,b}=
  \begin{cases}
    [\cO_{X_{a,b}}] & 1\leq a<b\leq 2n \text{ and } a < b,\\
    q[\cO_{X_{b,2n}}] & a=0,\ b<2n,\\
    0 & \text{otherwise}.
  \end{cases}
\]

\begin{thm}[Quantum Chevalley formula]
\label{thm_quantum_chevalley} In \(QK(X)\),
the product \(\cO_{2n-2,2n}\star\cO_{q_1,q_2}\) equals:
\begin{enumerate}
  \item \(\cO_{q_1-1,q_2}\) if \(q_1=q_2-1\);
  \item \(\cO_{q_1-1,q_2}+\cO_{q_1,q_2-1}-\cO_{q_1-1,q_2-1}\) if \(q_1<q_2-1\) and \(q_1+q_2\neq 2n+2, 2n+3\);
  \item \(\cO_{q_1-1,q_2}+\cO_{q_1,q_2-1}-\cO_{q_1-1,q_2-2}-\cO_{q_1-2,q_2-1}+\cO_{q_1-2,q_2-2}\) if \(q_1<q_2-1\) and \(q_1+q_2=2n+3\);
  \item \(2\cdot\cO_{q_1-1,q_2-1}+\cO_{q_1-2,q_2}+\cO_{q_1,q_2-2}-2\cdot\cO_{q_1-2,q_2-1}-2\cdot\cO_{q_1-1,q_2-2}+\cO_{q_1-2,q_2-2}\) when \(q_1+q_2=2n+2\) and \(q_1\neq 2,\ n\);
  \item[(4.1)] \(2\cdot\cO_{q_1-1,q_2-1}+\cO_{q_1,q_2-2}-2\cdot\cO_{q_1-1,q_2-2}-\cO_{q_1-2,q_2-1}+\cO_{q_1-2,q_2-2}\) when \(q_1=2,\ q_2=2n\);
  \item \(2\cdot\cO_{q_1-1,q_2-1}+\cO_{q_1-2,q_2}-2\cdot\cO_{q_1-2,q_2-1}-\cO_{q_1-1,q_2-2}+\cO_{q_1-2,q_2-2}\) when \(q_1=n,\ q_2=n+2\).
\end{enumerate}
\end{thm}

\footnotesize

\bigskip 
Universit\'e C\^ote d'Azur, CNRS, Laboratoire J.-A. Dieudonn\'e, Parc Valrose, F-06108 Nice Cedex 2, {\sc France}.

{\it Email address}: {\tt vladimiro.benedetti@univ-cotedazur.fr}

\smallskip 

Centre de Mathématiques Laurent Schwartz (CMLS), CNRS, École polytechnique, Institut Polytechnique de Paris, 91120 Palaiseau, {\sc France}.

{\it Email address}: {\tt  nicolas.perrin.cmls@polytechnique.edu}

\smallskip 

Department of Mathematics, 225 Stanger Street, McBryde Hall, Virginia Tech University, Blacksburg, VA 24061 {\sc USA}.

{\it Email address}: {\tt weihong@vt.edu}


\begin{thebibliography}{BCMP18b}

\bibitem[ACT22]{anderson.chen.ea:on}
D.~Anderson, L.~Chen, and H.-H.~Tseng.
\newblock On the finiteness of quantum K-theory of a homogeneous space.(English summary)
\newblock {\em IMRN}, no.2, 1313--1349, 2022.

\bibitem[AGM11]{anderson.griffeth.ea:positivity} 
D.~Anderson, S.~Griffeth and E.~Miller.
\newblock Positivity and Kleiman transversality in equivariant  K -theory of homogeneous spaces.
\newblock {\em J. Eur. Math. Soc. (JEMS)} 13 (2011), no. 1, 57--84.

\bibitem[Bri02]{brion:positivity}
M.~Brion.
\newblock Positivity in the {G}rothendieck group of complex flag varieties.
\newblock volume 258, pages 137--159. 2002.
\newblock Special issue in celebration of Claudio Procesi's 60th birthday.

\bibitem[Buc02]{buch:littlewood}
A.~S. Buch.
\newblock A Littlewood-Richardson rule for the  K -theory of Grassmannians.
\newblock {\em Acta Math.} 189 (2002), no. 1, 37--78.

\bibitem[Buc03]{buch:quantum}
A.~S. Buch.
\newblock Quantum cohomology of {G}rassmannians.
\newblock {\em Compositio Math.}, 137(2):227--235, 2003.

\bibitem[BCMP13]{buch.chaput.ea:finiteness}
A.~S. Buch, P.-E. Chaput, L.~C. Mihalcea, and N.~Perrin.
\newblock Finiteness of cominuscule quantum {$K$}-theory.
\newblock {\em Ann. Sci. \'{E}c. Norm. Sup\'{e}r. (4)}, 46(3):477--494 (2013),
  2013.
  
  \bibitem[BCMP18]{buch.chaput.ea:chevalley}
A.~S. Buch, P.-E. Chaput, L.~C. Mihalcea, and N.~Perrin.
\newblock A {C}hevalley formula for the equivariant quantum {$K$}-theory of
  cominuscule varieties.
\newblock {\em Algebr. Geom.}, 5(5):568--595, 2018.
  
  \bibitem[BCMP22]{buch.chaput.ea:positivity}
  A.~S. Buch, P.-E. Chaput, L.~C. Mihalcea, and N.~Perrin.
\newblock Positivity of minuscule quantum {$K$}-theory.
\newblock Preprint {\em arXiv:2205.08630}.

\bibitem[BCP23]{buch.chaput.ea:seidel}
A.~S. Buch, P.-E. Chaput, and N. Perrin.
\newblock Seidel and Pieri products in cominuscule quantum K-theory.
\newblock Preprint arXiv:2308.05307.

\bibitem[BKT03]{buch.kresch.ea:gromov-witten}
A.~S. Buch, A.~Kresch, and H.~Tamvakis.
\newblock Gromov-{W}itten invariants on {G}rassmannians.
\newblock {\em J. Amer. Math. Soc.}, 16(4):901--915, 2003.

\bibitem[BKT09]{buch.kresch.ea:quantum}
A.~S. Buch, A.~Kresch, and H.~Tamvakis.
\newblock {Quantum Pieri rules for isotropic Grassmannians.}
\newblock {\em Invent. Math.}, 178:345--405, 2009.

\bibitem[BM11]{buch.mihalcea:quantum}
A.~S. Buch and L.~C. Mihalcea.
\newblock Quantum {$K$}-theory of {G}rassmannians.
\newblock {\em Duke Math. J.}, 156(3):501--538, 2011.

\bibitem[BM15]{buch.mihalcea:curve}
A.~S. Buch and L.~C. Mihalcea.
\newblock Curve neighborhoods of {S}chubert varieties.
\newblock {\em J. Differential Geom.}, 99(2):255--283, 2015.

\bibitem[CMP07]{chaput.manivel.ea:hidden}
P.-E. Chaput, L.~Manivel, and N.~Perrin.
\newblock  Quantum cohomology of minuscule homogeneous spaces. II. Hidden symmetries.
\newblock {\em IMRN} (2007), no. 22, Art. ID rnm107, 29 pp.

\bibitem[CMP09]{chaput.manivel.ea:affine}
P.-E. Chaput, L.~Manivel, and N.~Perrin.
\newblock Affine symmetries of the equivariant quantum cohomology ring of
  rational homogeneous spaces.
\newblock {\em Math. Res. Lett.}, 16(1):7--21, 2009.

\bibitem[CP11]{chaput.perrin:on}
P.-E. Chaput and N.~Perrin.
\newblock On the quantum cohomology of adjoint varieties.
\newblock {\em Proc. Lond. Math. Soc.}, no. 2, 103(3):294–330,  2011.

\bibitem[CP23]{chaput.perrin:affine}
P.-E. Chaput and N.~Perrin.
\newblock Affine symmetries in quantum cohomology: corrections and new results.
\newblock {\em Math. Res. Lett.},
30(2):341--374, 2023.

\bibitem[Elk78]{Elkik}
R.~Elkik.
\newblock Singularités rationnelles et deformations.
\newblock {\em Inventiones mathematicae}, no. 47:139–147,  1978.

\bibitem[FP97]{fulton.pandharipande:notes}
W.~Fulton and R.~Pandharipande.
\newblock Notes on stable maps and quantum cohomology.
\newblock In {\em Algebraic geometry---{S}anta {C}ruz 1995}, volume~62 of {\em
  Proc. Sympos. Pure Math.}, pages 45--96. Amer. Math. Soc., Providence, RI,
  1997.
  
\bibitem[Giv00]{givental:wdvv}
A.~Givental.
\newblock On the {WDVV} equation in quantum {$K$}-theory.
\newblock {\em Michigan Math. J.}, 48:295--304, 2000.
\newblock Dedicated to William Fulton on the occasion of his 60th birthday.

\bibitem[GHS03]{graber.harris.starr}
T.~Graber and J.~Harris and J.~Starr.
\newblock {\em Journal of the American Mathematical Society}, no. {1}, {57--67}, {2003}.

\bibitem[Gra01]{graham:positivity}
W.~Graham.
\newblock Positivity in equivariant Schubert calculus.
\newblock {\em Duke Math. J.} 109 (2001), no. 3, 599--614.

\bibitem[GK09]{graham.kumar:positivity}
W.~Graham and S.~Kumar.
\newblock On positivity in  T -equivariant  K -theory of flag varieties.
\newblock {\em IMRN} (2008), Art. ID rnn 093, 43 pp.

\bibitem[Kat18]{kato:loop}
S.~Kato.
\newblock Loop structure on equivariant K-theory of semi-infinite flag manifolds.
\newblock Preprin arXiv:1805.01718v5. 

\bibitem[KP01]{kim.pandharipande:the}
B.~Kim, R.~Pandharipande
\newblock The connectedness of the moduli space of maps to homogeneous spaces.
\newblock {\em The Symplectic geometry and mirror symmetry} (Seoul, 2000), 187--201. World Scientific Publishing Co., Inc., River Edge, NJ, 2001.

\bibitem[Kol86]{kollar:higher}
J.~Koll{\'a}r.
\newblock Higher direct images of dualizing sheaves. {I}.
\newblock {\em Ann. of Math. (2)}, 123(1):11--42, 1986.

\bibitem[LM06]{lenart.maeno:quantum}
C.~Lenart and T.~Maeno.
\newblock Quantum {G}rothendieck polynomials.
\newblock math.CO/0608232, 2006.

\bibitem[LP07]{lenart.postnikov:affine}
C.~Lenart and A.~Postnikov
\newblock Affine Weyl groups in $K$-theory and representation theory
\newblock {\em IMRN}, no.9, 2007.

\bibitem[Mih06]{mihalcea:positivity}
L.~C. Mihalcea.
\newblock Positivity in equivariant quantum Schubert calculus.
\newblock {\em Amer. J. Math.} 128 (2006), no. 3, 787--803.

\bibitem[Mih22]{mihalcea:lectures}
L.~C. Mihalcea.
\newblock Lectures on quantum $K$-theory for flag manifolds.
\newblock Lecture notes accessed on 08.09.2023, \verb|https://personal.math.vt.edu/lmihalce/QKlectures.pdf|.

\bibitem[Ric92]{richardson:intersections}
R.~W. Richardson.
\newblock Intersections of double cosets in algebraic groups.
\newblock {\em Indag. Math. (N.S.)}, 3(1):69--77, 1992.

\bibitem[Ros22]{rosset:quantum}
S.~Rosset.
\newblock Quantum K theory for flag varieties
\newblock Preprint arXiv:2202.00773.

\bibitem[Sei97]{seidel:pi_1} 
P. Seidel.
\newblock $\pi_1$ of symplectic automorphism groups and invertibles in quantum homology rings.
\newblock {\em Geom. Funct. Anal.}, 7(6): 1046--1095, 1997.

\bibitem[Smi21]{Sm21}
M.~Smirnov.
\newblock On the derived category of the adjoint {G}rassmannian of type {F}.
\newblock Available at
  \href{https://arxiv.org/abs/2107.07814}{arXiv:2107.07814}.

\bibitem[Tar23]{tarigradschi:curve} 
M. Ţarigradschi
\newblock Curve neighborhoods of Seidel products in quantum cohomology
\newblock Preprint arXiv:2309.05985.

\bibitem[Tho98]{thomsen:irreducibility}
J.~F.~Thomsen
\newblock Irreducibility of $\overline{M}_{0,n}(G/P,\beta)$.
\newblock {\em Internat. J. Math}, 9(3): 367--376, 1998.

\bibitem[Xu21]{xu:quantum} 
W.~Xu
\newblock Quantum K-theory of Incidence Varieties
\newblock Preprint arXiv: 2112.13036.

\end{thebibliography}
\end{document}